 \documentclass[draft]{article}

\usepackage{amsmath,amsfonts,amsthm,amssymb,amscd,cancel,color,mathabx}
\usepackage{enumitem}
\usepackage{verbatim}
\usepackage[dvipdfmx]{graphicx}
\usepackage{ulem,color}

\renewcommand{\Re}{\mathop{\rm Re}\nolimits}
\renewcommand{\Im}{\mathop{\rm Im}\nolimits}
\def\S{\mathhexbox278}

\theoremstyle{plain}
\newtheorem{theorem}{Theorem}[section]
\newtheorem{lemma}[theorem]{Lemma}
\newtheorem{proposition}[theorem]{Proposition}

\theoremstyle{definition}

\theoremstyle{remark}
\newtheorem{remark}[theorem]{Remark}
\newtheorem{notation}[theorem]{Notation}

\newtheorem{claim}[theorem]{Claim}

\newcommand{\R}{{\mathbb R}}

\newcommand{\Z}{{\mathbb Z}}

\def\im{{\rm i}}

\newcommand{\C}{\mathbb{C}}

\newcommand{\sqt}{\sqrt{2}}
\def\({\left(}
\def\){\right)}
\def\<{\left\langle}
\def\>{\right\rangle}
\newcommand{\sech}{{\mathrm{sech}}}

\newcommand{\rad}{{\mathrm{rad}}}    \newcommand{\diag}{{\mathrm{diag}}}
\numberwithin{equation}{section}

\begin{document}

\title{The asymptotic stability  on the line  of  ground states of the  pure power NLS with  $0<|p-3|\ll 1$}

\author{Scipio Cuccagna, Masaya Maeda }
\maketitle

\begin{abstract} For exponents $p$ satisfying $0<|p-3|\ll 1$ and only in the context of spatially even solutions  we prove that the ground states of the nonlinear Schr\"odinger equation (NLS) with pure power nonlinearity of exponent $p$ in the line are asymptotically stable. The proof is similar to a related result of Martel \cite{Martel2} for a cubic quintic NLS. Here we modify the second part of Martel's argument, replacing the second virial inequality for a transformed problem with a smoothing estimate on the initial problem, appropriately  tamed by multiplying the initial variables and equations  by a cutoff.
\end{abstract}

\section{Introduction}

We consider  the pure    power focusing  Nonlinear Schr\"odinger Equation  (NLS) on the line
\begin{align}\label{eq:nls1}&
  \im \partial _t  u +\partial _x^2  u = -f(u)  \text{   where }    f(u)=|u| ^{p-1}u    \text{ for $0<|p-3|\ll 1$. }
\end{align}
     We  consider only even solutions, eliminating translations  and simplifying the problem. In particular, we will study Equation \eqref{eq:nls1} in the space $H^ 1_\rad (\R )= H^ 1_\rad (\R ,\C ) $, of even functions in $H^ 1  (\R ,\C )$. It is well known that Equation   \eqref{eq:nls1} has standing waves, solutions with the form $u(t,x)=e^{\im \omega t}\phi_{\omega}(x)$. They are  obtained from  $\phi _\omega (x)=\omega ^{\frac 1{p-1}} \phi (\sqrt{\omega }x) $  with the explicit formula
\begin{align} &
    \phi (x) =  {\(\frac {p+1}2 \)^{\frac 1{p-1}}}{
\sech   ^{\frac 2{p-1}}\(\frac{p-1}2 x\)}  ,\label{eq:sol}
   \end{align}
 see formula (3.1) of Chang et al. \cite{Chang}.
Energy $\mathbf{E}$ and Mass $\mathbf{Q}$ are invariants of \eqref{eq:nls1}, where
  \begin{align}\label{eq:energy}
& \mathbf{E}( {u})=\frac{1}{2}   \| u' \| ^2 _{L^2\( \R \) } -\int_{\R} F(u)    \,dx   \text {  where }  F(u)    =\dfrac{ |u|  ^{p+1}}{p+1}, \\&   \label{eq:mass} \mathbf{Q}( {u})=\frac{1}{2}   \| u  \| ^2 _{L^2\( \R \) }.
\end{align}
  It is well known that
   $\phi _\omega$  minimizes $\mathbf{E}$ under the constraint $\mathbf{Q}=\mathbf{Q}(\phi _\omega )=:\mathbf{q}(\omega )$. Notice that $\mathbf{q}(\omega ) =\omega ^{ \frac{2}{p-1}-\frac{1}{2}}   \mathbf{q}(1)$.
     We have $\nabla \mathbf{E} (\phi _\omega)= -\omega   \nabla \mathbf{Q} (\phi _\omega)  $ which reads also
\begin{align} \label{eq:static}
   - \phi _\omega ''+\omega \phi _\omega - \phi _\omega ^p=0   .
\end{align}
Set now  for $\omega ,\delta \in \R_+:=(0,\infty)$ the set     $$\mathcal{U} (\omega  ,\delta  ) := \bigcup _{\vartheta _0\in \R }   e^{\im \vartheta _0}  D_{H ^1_\rad (\R )}({\phi}_{\omega },\delta  ), $$ where $D_X(u,r):=\{v\in X\ |\ \|u-v\|_X<r \}$.
The following was shown by Cazenave and Lions \cite{cazli}, see also Shatah  \cite{shatah} and Weinstein  \cite{W1}.

\begin{theorem}[Orbital Stability]  Let $p\in (1,5)$ and let $\omega _0 >0$. Then for any $\epsilon >0$  there exists a  $\delta >0$ such that for any  initial value    $u_0\in \mathcal{U} (\omega _0,\delta  ) $ then   the corresponding solution  satisfies
$u\in C^0  \(  \R ,   \mathcal{U} (\omega _0,\epsilon  ) \) $.
\end{theorem}
\qed

In order to study the notion of asymptotic stability, like in finite dimension, it is useful to have information on the  \textit{linearization} of \eqref{eq:nls1} at $\phi _\omega$, which we will see later has the following form
\begin{align}\label{eq:lineariz2} \partial _t   \begin{pmatrix}
 r_1 \\ r_2
 \end{pmatrix}      &=  \mathcal{L}_{\omega }  \begin{pmatrix}
 r_1 \\ r_2
 \end{pmatrix}  \text{  with }    \mathcal{L}_\omega := \begin{pmatrix}
0 & L_{-\omega} \\ -L_{+\omega} & 0
\end{pmatrix} ,
    \end{align}
where    \begin{align}\label{eq:lin1}&
  L _{+\omega}:=- \partial _x^2   +\omega -p \phi _\omega ^{p-1}   \\& \label{eq:lin0} L _{-\omega}:=- \partial _x^2   +\omega-  \phi _\omega ^{p-1}.
\end{align}
The linearization is better seen in the context of functions  in $H^ 1_\rad (\R , \R ^2 ) $ rather than in $H^ 1_\rad (\R ,\C ) $, because it is $\R$--linear rather than $\C$--linear.
For $p=3$ the operator $ \mathcal{L}_\omega$ is completely known very thoroughly, so for example all its plane waves are known explicitly, see section \ref{sec:intsy}.
Coles and Gustafson \cite{coles} proved that for  $0<|p-3|\ll 1$ the linearization $ \mathcal{L}_\omega$ has exactly one eigenvalue of the form  $\im \lambda   $ near $\im \omega$.  We set $\lambda (p,\omega):=\lambda  $. Furthermore
$0<\lambda (p,\omega) < \omega $    and $\dim \ker (\mathcal{L}_\omega -\im \lambda (p,\omega))=1$.
Let $\xi _{\omega}\in H^1(\R, \C^2) $ be  an appropriately normalized generator of $\ker (\mathcal{L}_\omega -\im \lambda (p,\omega))$, see \S \ref{sec:lin1}.

In this paper we prove  the following result.

\begin{theorem} \label{thm:asstab}
There exists $p_1<3<p_2$ s.t.\ for any $p\in (p_1,p_2)\setminus\{3\}$ and any $\omega _0 >0$, any $a>0$ and any $\epsilon >0$   there exists a  $\delta >0$ such that for any  initial value    $u_0\in  D _{H^1_\rad (\R ) } ({\phi}_{\omega  _0},\delta  ) $
there exist functions $ \vartheta , \omega \in C^1 \( \R , \R \)$,  $ z \in C^1 \( \R , \C  \)$ and $\omega_+>0$    s.t.\  the  solution  of \eqref{eq:nls1} with initial datum  $u_0$  can be written as
 \begin{align} \label{eq:asstab1}
    & u(t)= e^{\im \vartheta (t)} \(   \phi _{\omega (t)} + z(t) \xi _{\omega(t) }+ \overline{z}(t) \overline{\xi} _{\omega (t)} +\eta (t)\) \text{   with}
 \\&  \label{eq:asstab2}   \int _{\R } \|  e^{- a\< x\>}   \eta (t ) \| _{H^1(\R )}^2 dt <  \epsilon  \text{  where }\< x\>:=\sqrt{1+x^2} \\&  \label{eq:asstab20}
   \lim _{t\to \infty  }  \|  e^{- a\< x\>}   \eta (t ) \| _{L^2(\R )}     =0   \\&  \label{eq:asstab2}
   \lim _{t\to \infty}z(t)=0      ,\\&
\lim _{t\to \infty}\omega (t)= \omega _+ .\label{eq:asstab3}
\end{align}

\end{theorem}

\begin{remark}
Standing wave solutions in the integrable case $p=3$ are not asymptotically stable due   the existence of   \textit{breathers}, see Borghese et al. \cite[formula (1.21)]{borghese}, very close in $H^1(\R )$ to a soliton (take for example  $\eta _2 \longrightarrow 0$  in \cite[formula (1.21)]{borghese}). More broadly,
when $p=3$ it is possible  to \textit{add}  solutions using B\"acklund transformations.
In fact the situation resembles that of small energy solutions of NLS with a trapping linear potential with two or more eigenvalues when we treat the nonlinearity as a perturbation. Then the linear equation  has quasiperiodic solutions, due to linear superposition, while generically a nonlinear equation does  not, see \cite{CM2109.08108} for a 1 D result and therein for
references.
However, some form of asymptotic  stability holds also in the  $p=3$ case,  using different norms and  the theory of Integrable Systems, \cite{borghese,cupe2014,saalmann}.
\end{remark}

%

\begin{remark}\label{rem:asstab3--}  Notice that the fact that the $H^1(\R )$ norm of $\eta$ is uniformly bounded for all times, guaranteed by the orbital stability, and
Theorem \ref{thm:asstab}  imply   that
\begin{align}\label{eq:rem:asstab3--1}
  \lim _{t\to +\infty }  u(t) e^{-\im \vartheta (t)}=  \phi  _{\omega _+}  \text{  in }L^\infty _{\text{loc}}(\R ) .
\end{align}

\end{remark}

\begin{remark}\label{rem:extthm:asstab}  While we prove  Theorem  \ref{thm:asstab}  in the case that   $0<|p_j -3|\ll 1$  for $j=1,2$,   it is possible to ease significantly these hypotheses.
In fact, we emphasize that our theory is largely non perturbative.
What we need in the proof are the following facts:
\begin{description}
  \item[(i)]  $2 \lambda (p,1) >1 $;
  \item[(ii)]  we can take   $\gamma (p,1)\neq 0$, where this constant is related to the Fermi Golden Rule (FGR), see below, and is defined below  in \eqref{eq:fgrgamma};
  \item[(iii)]  given   the Jost functions  $f_3(\cdot ,0)$ and $g_3( \cdot ,0)$  introduced later  in  Sect. \ref{sec:lin},   which  depend analytically on $p$,   their  Wronskian  is nonzero: $ W[f_3(x,0), g_3( x,0)] \neq 0 $.
\end{description}
 According to numerical computations in  Chang et al. \cite{Chang}, condition (\textbf{i}) holds for all $2<p<3$, where there are no other eigenvalues  of the form $\im \lambda $ for $\lambda >0$, no resonance is observed at the threshold of the continuous spectrum. Furthermore, both $\gamma (p,1) $ and $W[f_3(x,0), g_3( x,0)]$ can be made to depend analytically on $p$. This would show that Theorem \ref {thm:asstab}    holds for all the $p\in (2,3)$ except for a discrete subset of $  (2,3)$.  Similarly, when we consider  $p>3$,     we have   $\lambda (3,1)=1 $  and $\lambda (5,1)=0 $ and, according to
  the numerical computations in  Chang et al. \cite{Chang}  (similar  results were in part already known:  the first author  learned about them by personal communication by M.I.Weinstein   in the year 2000), the function $(3,5)\ni p\to \lambda (p,1) $   is  strictly decreasing, there are no other eigenvalues of the form $\im \lambda $ with $\lambda >0$
 and no resonance is observed at the threshold of the continuous spectrum. So there will be a $p_0\in (3,5)$ such that  $2 \lambda (p,1) >1 $ for all $3<p <p_0$ (notice that $p_0\in (4,5)$ in \cite[fig. 1]{Chang} with $p_0$ quite close to 5). Since also conditions (\textbf{ii}) and (\textbf{iii}) will be true for all the $p\in (3,p_0) $ outside a discrete subset of  $ (3,p_0) $ we conclude that  outside a discrete subset of  $ (3,p_0) $ Theorem \ref{thm:asstab} continues to be true.
 That the FGR constants are nonzero  and  that condition (\textbf{iii}) holds,     are expected to be  generically true also for \textit{large} perturbations of the cubic NLS.
 Furthermore, our framework could in principle  be applied in higher dimensions.
\end{remark}

\begin{remark}
  \label{rem:martconj}  Martel \cite{Martel2} conjectures that for generic small perturbations of the cubic NLS the asymptotic stability result in \cite{Martel1,Martel2} is true. Here we focus only on pure power NLS's but our method goes some way to prove this conjecture. For our method to work, we always need that   the threshold of the continuous spectrum     be not a resonance, which should be true for generic perturbations. For the smoothing estimate we need  additionally condition (\textbf{iii}) in Remark \ref{rem:extthm:asstab}, which is    true for small perturbations.  If there is no eigenvalue in $(0, \im \omega)$ like in Martel \cite{Martel1} and Rialland \cite{rialland} and if the non resonance  condition holds then  our method proves Martel's conjecture.
   If there is one eigenvalue   $\lambda \in (0, \im \omega)$  (it is easy to show, proceeding along the lines in \cite{cupe2005} or, since this is 1 dimension, Coles and Gustafson \cite{coles},
    that there can be at most one such eigenvalue for small perturbations), condition (\textbf{i}) in Remark \ref{rem:extthm:asstab}  will be true. If, as expected,  generically the Nonlinear Fermi golden rule (FGR)  condition  (\textbf{ii}) in Remark \ref{rem:extthm:asstab}  holds,   then our method works. So, using our framework, to prove Martel's conjecture it remains only to prove that for generic small perturbations there is no threshold resonance and that if there is an eigenvalue   the FGR is true.  We think that also Martel's method in \cite{Martel2}  yields a similar result.
\end{remark}


Equation \eqref{eq:nls1} is one of the most classical Hamiltonian systems in PDE's and the asymptotic stability of its ground states has been a longstanding open problem. Attempts at solving it date back at least to the 80's, see Soffer and Weinstein \cite{SW1,SW2}. The Vakhitov--Kolokolov stability criterion yields the orbital stability exactly for $p<5$, while for $p\ge 5$  the ground states are orbitally unstable. On the other hand, proving asymptotic stability requires some form of spatial dispersion. It turns out that it is difficult to prove dispersion for  $p<5$, which is the opposite condition to the one utilized for example in  Strauss  \cite{strauss} to prove a form of asymptotic stability of vacuum. In essence, dispersion is a linear phenomenon, but for  $p<5$ the nonlinearity is strong and makes it difficult to treat the problem as a perturbation of a linear equation.
Whence the inability in the literature to deal with the asymptotic stability problem for equation \eqref{eq:nls1}, left unaddressed   from Buslaev and Perelman \cite{BP1,BP2,BS} on.
Another problem is the presence of nonzero eigenvalues of the linearization $ \mathcal{L}_\omega$. These eigenvalues slow  dispersion because the corresponding discrete modes tend to oscillate periodically and decay slowly, see \cite{BP2,BS,SW3,zhousigal}, and furthermore they are a drag to the dispersion of the continuous modes, on whose equation they exert a forcing. This is especially true in the case when there are eigenvalues close to 0, as happens for  example   for $p$ close to 1  or to 5.   A mechanism  first discussed by Sigal \cite{sigal},  the  Nonlinear Fermi Golden Rule (FGR), should allow to show that the   discrete modes  lose energy by nonlinear interaction with the continuous modes.

It is next to impossible to see the FGR,  without  utilizing the Hamiltonian structure of the NLS, see for example the complications in \cite{Gz}.  In  \cite{Cu1}  the FGR is seen    using  canonical coordinates and normal forms transformations. Recently papers like \cite{CM3} have simplified significantly \cite{Cu1}, eliminating the need of normal forms, thanks to the notion of Refined Profile, which is a generalization of
the families of ground states,    a sort of surrogate of a  (here   not existent)  family of quasiperiodic solutions  and   encodes the discrete modes in the problem. Finding the Refined Profile is elementary, but   requires Taylor expansions  of the nonlinearity, with the order   higher when there are    eigenvalues of   $ \mathcal{L}_\omega$ closer to 0. Since $f(u)$ is not smooth in $u$, this is one of the   main reasons why $p$ needs here to stay close to 3, where the eigenvalue is not close to 0. Even more difficult appears the problem  when the  power $p$   is such that $ \mathcal{L}_\omega$ has resonances at the thresholds of the continuous spectrum, except in the integrable case $p=3$. To see some of the difficulties, on  a different and non integrable model  involving a resonance, we refer to  the   partial results in  \cite{LS2023,PalPus24}. We stress that here the spectral configuration is   as in Martel \cite{Martel2} and that we prove the FGR like in Martel \cite{Martel2}.

\noindent Dispersion  has played a crucial role in stabilization problems. The sequel \cite{SW2}  to   \cite{SW1} was only possible because a result on dispersion for Schr\"odinger by Journ\'e et al. \cite{jss}. Strichartz estimates, in particular the 3 D  endpoint Strichartz estimate of Keel and Tao \cite{Kl-Tao}, were introduced by Gustafson et. al. \cite{GNT} and played an important role in the theory. Dimensions 1 and 2 were considered
by Mizumachi \cite{mizu08,M2}, whose use  of smoothing estimates  has provided us with crucial insights.
But ultimately,  in low dimensions Strichartz estimates have limited scope.  A very important turning point in the theory in 1 D has been Kowalczyk et al. \cite{KMM20} which, along with the further developments and refinements in   \cite{KMM3,KMMvdB21AnnPDE,KM22}, has exploited very effectively   virial inequalities. Recently Martel \cite{Martel1,Martel2} has applied and extended these  ideas  to the study of the asymptotic stability of two versions of the cubic--quintic NLS introduced by  Pelinovsky  et al. \cite{pel98}.
    Rialland \cite{rialland}   has generalized   \cite{Martel1}.  One of the most striking  features of the theory  initiated by Kowalczyk et al. \cite{KMM3}, is how easily  the  nonlinear  term involving only the continuous mode of the solution   is sorted out by
what we might call the \textit{high energy}  virial inequality, see inequality is  \eqref{eq:cruche} below, by means of a clever but simple integration by parts.  The same term  is     almost impossible to treat with   perturbative methods involving the Duhamel formula.  There exist  also different approaches, some, but not exclusively,  stemming from the theory of space--time resonances of Germain et al. \cite{germain1}. For a partial sample we refer for example to work of Delort  \cite{Delort},  Germain et al. \cite{GPR18}, Naumkin
\cite{naumkin2016}. Recently Germain and Collot \cite{germain2} have recovered and partially generalized Martel \cite{Martel1}.  This theory requires a certain degree of smoothness of the nonlinearity $f(u)$, so it is not easily applicable to the specific model \eqref{eq:nls1}.
We think that the framework in    Kowalczyk et al., to which we return, is more robust and easier to apply in stability problems.

\noindent   After the first \textit{high energy} virial inequality, the  papers \cite{KMM3,KMMvdB21AnnPDE,KM22,Martel1,Martel2}   utilize what we might call a \textit{low energy} virial inequality, which requires  new coordinates where the linearization is nontrapping. This has some similarities  with the subtraction of solitons to study dispersion by means of the Nonlinear Steepest Descent method of Deift and Zhou, as done for instance by
Grunert  and   Teschl
\cite{GTeschl}, although the details are very different. An interesting feature and a possible criticality of the   {low energy} virial inequality, is that the virial inequality
  produces a different linear operator, which also needs to be non--trapping.
While in \cite{Martel1,Martel2,rialland}, which deal with small perturbations of the cubic NLS, the two non--trapping conditions are shown to be equivalent, thanks to a result by Simon  \cite{simon}  on small perturbations of the Laplacian in dimensions 1 and 2, in general this might not be the case, so it is plausible that  in some cases     the second virial inequality method might require   restrictions not intrinsic but  rather due to the method of proof.
To take a concrete example,   in the first paper  \cite{CM2109.08108} of our own series inspired by the work of  Kowalczyk et al.\cite{KMM3},
the repulsivity Assumption 1.13  \cite{CM2109.08108}  is in fact unnecessary and  is used  only because of     the method of proof.  This is   the main insight and motivation for this paper. In  \cite{CM2109.08108}, besides the two virial inequalities, there  is a smoothing estimate, inspired  by Mizumachi \cite{mizu08,M2},
which in   \cite{CM2109.08108} appears because the FGR rule is proved in an overly complicated way (a simplification appears in \cite{CMS23}, motivated by \cite{KM22}). The insight in the present paper, is that, while it is obviously a good idea to prove the FGR as simply as possible,  it is possible to replace the
the second virial inequality by    smoothing estimates. We explain now some further reasons why this might be convenient.  Kowalczyk et al.\cite{KMM20,KMM3,KMMvdB21AnnPDE,KM22}   and   Martel \cite{Martel1,Martel2}   perform some  Darboux transformations, which are almost isospectral transformations which   allow to eliminate eigenvalues of the linearization in a controlled way. For scalar Schr\"odiger operators in the line the theory is  fully developed in Deift and Trubowitz \cite{DT}, with an important special case   discussed in Chang et al. \cite{Chang}.  The analogue for the  linearizations $\mathcal{L}_\omega$ is in Martel \cite{Martel1}  and Rialland \cite{rialland} in  the case without internal modes and in Martel \cite{Martel2} with just one internal mode.  At exactly the same time we posted the first version of this manuscript was posted Rialland \cite{Rialland2}, which follows closely  Martel \cite{Martel2}  with a result   for nonlinearities similar but more general than the cubic quintic.

It is not clear to us what are the Darboux transformations  when the configuration of the internal modes of $\mathcal{L}_\omega$ is more complicated and   if the space dimension is 2 or larger.
So  it is worth to develop some alternative method which does not use  Darboux transformations. The Kato--smoothing estimates are a classical tool, originating in Kato \cite{kato}, valid  in any dimension.
  The smoothing estimates are perturbative, based on the Duhamel formula. But there is no issue here of too strong  nonlinearity  because we only need to bound   the continuous mode  multiplied by a spatial cutoff. This means that  we can multiply the NLS by a cutoff,   taming the nonlinearity. The cutoff appears  also in the second virial inequalities in the theory of  Kowalczyk et al.  Obviously, in the equation we obtain an additional term,    delicate for us, represented by the commutator of  Laplacian  and cutoff. We treat it via a     specific smoothing estimate, see Lemma \ref{lem:smoothest1} below. In \cite{CM2109.08108} we used some standard bounds on the Jost functions of Schr\"odinger operators in  1 D to prove an analogous lemma.  Here, for $\mathcal{L}_\omega$ these bounds on the Jost functions are not as obvious and this  is one of the points where we exploit that our problem is a small perturbation of the cubic NLS,   the specific condition is (iii) in Remark \ref{rem:extthm:asstab} that appears generic and is in principle possible to check numerically in specific examples.
Finally,   for   a  rather long list of references on the subject up until 2020,
we   refer to our  survey \cite{CM21DCDS}.

\section{Linearization}\label{sec:lin1}

We return to a discussion of the linearization \eqref{eq:lineariz2}.
Weinstein \cite{W2}  showed that  for $ 1<p<5$ the   generalized kernel $ {N}_g(\mathcal{L}_\omega):=\cup_{j=1}^\infty \mathrm{ker} \mathcal{L}_{\omega}^j$ in $H^1 _\rad ( \R , \C^2 ) $ is
\begin{align}\label{eq:Ng}
 {N}_g(\mathcal{L}_{\omega })=\mathrm{span}\left \{ \begin{pmatrix}
 0 \\ \phi_{ \omega}
 \end{pmatrix} , \begin{pmatrix}
 \partial_\omega \phi_{\omega } \\ 0
 \end{pmatrix}          \right \} .
\end{align}
By symmetry reasons,   it known that the spectrum $\sigma \(  \mathcal{L}_{\omega }\) \subseteq \C $ is symmetric by reflection with respect of the coordinate axes. Furthermore,
by Krieger and Schlag   \cite[p. 909]{KrSch}  we know that  $\sigma \(  \mathcal{L}_{\omega }\) \subseteq \im \R  $. By standard Analytic  Fredholm theory   the essential spectrum is $(- \infty \im ,-\omega \im ] \cup
[ \im \omega , +\infty \im ) $. As already mentioned $0\in  \sigma \(  \mathcal{L}_{\omega }\)$. Numerical computations by
  Chang et al. \cite{Chang} show that for $p\in (2, 3)\cup (3,5)$, besides 0 there are two eigenvalues  of $   \mathcal{L}_{\omega }$, they are of the form $\pm \im \lambda $  with $\lambda >0$ and if we set as above $\lambda (p,\omega)=\lambda$,   we have $\lambda (3,\omega)=\omega $  and $\lambda (5,\omega)=0 $. As mentioned above
       Coles and Gustafson \cite{coles} corroborate  rigorously  the numerical computations of  Chang et al. \cite{Chang} for $0<|p-3|\ll 1$. Furthermore, since at $p=3$  the linearization $ \mathcal{L}_{\omega }$
   has only 0 as an eigenvalue, and $\pm \im \omega$ are resonances,     Coles and Gustafson \cite{coles}   imply that   besides
  $-\im  \lambda (p,\omega) , 0 , \im \lambda (p,\omega)$,   for $0<|p-3|\ll 1$ there are no other eigenvalues and that $\pm \im \omega$      are not resonances.

\noindent Let us consider the orthogonal  decomposition
\begin{align}
  \label{eq:dirsum} L^2_\rad (\R , \C ^2) =  {N}_g(\mathcal{L}_{\omega })\bigoplus {N}_g^\perp (\mathcal{L}_{\omega }^*)
\end{align}  We have,  for $\lambda  = \lambda (p,\omega)$,   a further decomposition
\begin{align}
  \label{eq:dirsum1}&{N}_g^\perp (\mathcal{L}_{\omega }^*) =   \ker (\mathcal{L}_{\omega }-\im \lambda )\bigoplus \ker (\mathcal{L}_{\omega }+\im \lambda ) \bigoplus X_c (\omega ) \text{ where } \\& X_c (\omega ) : = \(   {N}_g(\mathcal{L}_{\omega }^*) \bigoplus  \ker (\mathcal{L}_{\omega }^*-\im \lambda )\bigoplus \ker (\mathcal{L}_{\omega }^*+\im \lambda ) \) ^{\perp} . \label{eq:dirsum11}
\end{align}
We denote by $P_c$ the projection of $ L^2_\rad (\R , \C ^2)$  onto $X_c (\omega )$ associated with the above decompositions.

\noindent The space $L^2_\rad(\R , \C ^2)$ and the action of $\mathcal{L}_{\omega }$ on it is obtained by first  identifying $L^2_\rad(\R , \C  ) =L^2_\rad( \R ,  \R ^2  )  $ and then by extending   this action to  the completion of
$ L^2_\rad (\R ,  \R ^2  ) \bigotimes _\R \C  $ which is identified with $L^2_\rad (\R , \C ^2)$.  In $\C$ we consider the inner product
\begin{align*}
   \< z, w \> _{\C} =\Re \{z\overline{w}    \} =z_1w_1+z_2w_2  \text{ where } a_1=\Re a, \quad a_2=\Im a   \text{ for }a=z,w.
\end{align*}
This obviously coincides with the inner product in $\R ^2$ and expands as the standard sesquilinear  $  \< X  , Y  \> _{\C ^2}  =  X ^ \intercal \overline{Y}$  (row column product, vectors here are columns)  form   in $\C ^2$.  The operator of multiplication by $\im $  in $C=\R^2$ extends into the linear operator
$J ^{-1}=-J$ where   \begin{align}
J=\begin{pmatrix}
0 & 1 \\ -1 & 0
\end{pmatrix}.\nonumber
\end{align}
For  $u,v\in L^2_\rad (\R , \C ^2)$ we set $ \< u , v  \>    :=\int _\R \< u (x), v (x) \> _{\C ^2}  dx$. We have a natural symplectic form given by $\Omega :=\<  J ^{-1}\cdot  , \cdot   \>$ in both $ L^2(\R , \C ^2)$ and
$ L^2_\rad (\R , \R ^2)=L^2_\rad (\R , \C  )$, where equation \eqref{eq:nls1} is the Hamiltonian system  in  $  L^2_\rad (\R , \C  )$     with Hamiltonian the energy $E$ in \eqref{eq:energy}.  As we mentioned we consider
a generator $\xi _{\omega}\in  \ker (\mathcal{L}_{\omega }-\im \lambda )$.  Then for the complex conjugate $\overline{\xi}_{\omega} \in  \ker (\mathcal{L}_{\omega }+\im \lambda )$.
Notice   the well known and elementary $J \mathcal{{L}}_{\omega }=-\mathcal{{L}}_{\omega } ^*J$  implies that  $ \ker (\mathcal{L}_{\omega }^*+\im \lambda )= \mathrm{span}\left \{  J {\xi}_{\omega} \right \}$ and $ \ker (\mathcal{L}_{\omega }^*-\im \lambda )= \mathrm{span}\left \{  J \overline{{\xi}}_{\omega} \right \}$. Notice that in   Lemma 2.7  \cite{CM2} it is shown that we can normalize ${\xi}_{\omega}$  so that
\begin{align}
  \label{eq:xinormlize} \Omega (   {\xi}_{\omega} ,  {\xi}_{\omega} ) =- \im ,
\end{align}
  consistently with the fact that   the functional $\mathbf{E}(u)+\omega \mathbf{Q}(u)$ has a local minimum at $u=\phi _{\omega} $, see
\eqref{eq:enrexp}--\eqref{eq:enrexp1} later. Notice that
\eqref{eq:xinormlize} is the same as
\begin{align}
  \label{eq:xinormlize2} \Omega (   \Re {\xi}_{\omega} ,  \Im {\xi}_{\omega} ) = \frac{1}{2}  \text{  and }   \Omega (   \Re {\xi}_{\omega} ,   \Re {\xi}_{\omega} )=\Omega (    \Im  {\xi}_{\omega} ,   \Im  {\xi}_{\omega} )=0
\end{align}
where the latter is immediate by the skewadjointness of $J$.
Notice that
\begin{align}\label{eq:reimxi} &
  \xi _\omega = \(   \xi _1 , \xi _2 \) ^\intercal  \in  \ker (\mathcal{L}_{\omega }-\im \lambda )  \Leftrightarrow \left\{\begin{matrix}
     L_{-\omega }\xi _2 =\im \lambda \xi _1 \\ L_{+\omega }\xi _1 =-\im \lambda \xi _2
\end{matrix}\right.  \\& \nonumber   \left\{\begin{matrix} L_{-\omega }\xi _{2I} =  \lambda \xi _{1R}
      \\  L_{+\omega }\xi  _{1R} =-\lambda \xi  _{2I} \end{matrix}\right.     \text{  and }     \left\{\begin{matrix}
     L_{-\omega }\xi  _{2R} =-\lambda \xi  _{1I} \\   L_{+\omega }\xi _{1I} =  \lambda \xi _{2R}
\end{matrix}\right.   \text{ with }\xi _{jR}= \Re \xi _j   \text{ and  }\xi _{jI}=\Im \xi _j.
\end{align}
This implies that we can normalize so that
\begin{align}\label{eq:reimxi1} &  \xi _\omega = \(   \xi _1 , \xi _2 \) ^\intercal  \text{  with }  \xi _1=\Re  \xi _1 \text{  and }  \xi _2=\im \Im  \xi _2.
\end{align}
Hence condition \eqref{eq:xinormlize} becomes
\begin{align}
  \label{eq:xinormliz1} \int _{\R} \xi _1 \Im \xi _2 dx = 2 ^{-1}.
\end{align}

\begin{notation}\label{not:notation} We will use the following miscellanea of  notations and definitions.
\begin{enumerate}
 \item   We will set
\begin{align}
  \label{eq:notation1} \text{$\mathbf{e}(\omega):=\mathbf{E}( \phi _{\omega}) $, $\mathbf{q}(\omega):=\mathbf{Q}( \phi _{\omega}) $ and $\mathbf{d}(\omega):=\mathbf{e}(\omega)+\omega \mathbf{q}(\omega)$.}
\end{align}

\item  We denote by   $\diag (a, b)$   the diagonal matrix with first $a$ and then $b$ on the diagonal.

\item For $z\in \C$ we will use $z_1= \Re z $  and  $z_2= \Im  z $ and we will use the operators
\begin{align*}
  \partial _z := \frac{1}{2}\(  \partial _{z_1} -\im  \partial _{z_2 }\)  \text{ and } \partial  _{\overline{z}} := \frac{1}{2}\(  \partial _{z_1} +\im  \partial _{z_2 }\) .
\end{align*}

\item Like in the theory of     Kowalczyk et al.   \cite{KMM3},      we   consider constants  $A, B,  \epsilon , \delta  >0$ satisfying
 \begin{align}\label{eq:relABg}
\log(\delta ^{-1})\gg\log(\epsilon ^{-1}) \gg     A  \gg    B^2\gg B  \gg 1.
 \end{align}
 Here we will take $ A\sim B^3$, see Sect. \ref{sec:smooth1} below,  but in fact  $ A\sim B^n$ for  any   $n>2$ would make no difference.

 \item The notation    $o_{\varepsilon}(1)$  means a constant   with a parameter $\varepsilon$ such that
 \begin{align}\label{eq:smallo}
 \text{ $o_{\varepsilon}(1) \xrightarrow {\varepsilon  \to 0^+   }0.$}
 \end{align}
\item For  $\kappa \in (0,1)$    fixed in terms of $p$  and  small enough, we consider
\begin{align}\label{eq:l2w}&
\|  {\eta} \|_{  L ^{p,s}} :=\left \|\< x \> ^s \eta \right \|_{L^p(\R )}   \text{  where $\< x \>  := \sqrt{1+x^2}$,}\\&
\| \eta  \|_{  { \Sigma }_A} :=\left \| \sech \(\frac{2}{A} x\) \eta '\right \|_{L^2(\R )} +A^{-1}\left \|    \sech \(\frac{2}{A} x\)  \eta   \right\|_{L^2(\R )}  \text{ and} \label{eq:normA}\\& \|  {\eta} \|_{ \widetilde{\Sigma} } :=\left \| \sech \( \kappa \omega _0 x\)   {\eta}\right \|_{L^2(\R )} . \label{eq:normk}
\end{align}
\item We set
\begin{align}\label{eq:C+-}
 \C _\pm := \{ z\in \C : \pm \Im z >0   \} .
\end{align}
\item We will consider the Pauli matrices \begin{equation*}
   \sigma_1=\begin{pmatrix} 0 &
   1 \\
   1 & 0
    \end{pmatrix} \,,
   \quad
   \sigma_2= \begin{pmatrix} 0 &
   -\im \\
    \im  & 0
    \end{pmatrix} \,,
   \quad
   \sigma_3=\begin{pmatrix} 1 &
   0 \\
   0 & -1
    \end{pmatrix}.
   \end{equation*}

\item The point $\im \omega$ is a resonance for $\mathcal{L}_\omega$ if there exists a nonzero $ v\in L^\infty (\R , \C ^2)$
such that $\mathcal{L}_\omega v= \im \omega v$. Notice that for $p=3$ the point $\im \omega$ is a resonance, see \eqref{eqres1} for a $v$ when $\omega =1$. An elementary scaling yields the cases $\omega \neq 1$ from the  $\omega =1$ case.

\item Given two Banach spaces $X$ and $Y$ we denote by $\mathcal{L}(X,Y)$ the space of continuous linear operators from $X$ to $Y$. We write $\mathcal{L}(X ):=\mathcal{L}(X,X)$.

\item We have denoted by $P_c$ the projection on the space \eqref{eq:dirsum11} associated to the spectral decomposition
\eqref{eq:dirsum1}  of the operator $\mathcal{L}_{\omega}$. Later in \eqref{eq:opH} we will introduce an operator $H_{\omega}$ which is an equivalent to $\mathcal{L}_{\omega}$ and  obtained from $\mathcal{L}_{\omega}$ by a simple conjugation. By an abuse of notation we will continue to denote by $P_c$ the analogous spectral projection to the continuous spectrum component, only of $H_\omega$ this time.

\item We have the following elementary formulas,
\begin{align} \label{eq:derf1}&
  Df(u)X=\left . \frac{d}{dt} \(  |u+tX | ^{p-1}\( u+tX \) \)  \right | _{t=0}= |u  | ^{p-1}X + (p-1) |u  | ^{p-3}u \< u, X\> _\C \text{ and}\\&  D^2f(u)X^2= \left . \frac{d}{dt}    Df(u+tX )X  \right | _{t=0} \nonumber\\& = 2(p-1) |u  | ^{p-3} X \< u, X\> _\C + (p-1) |u  | ^{p-3}u |X|^2  + (p-1) (p-3)  |u  | ^{p-5}u\< u, X\> _\C ^2. \label{eq:derf2}
\end{align}

\item Following  the framework in Kowalczyk et al. \cite{KMM3} we   fix an even function $\chi\in C_c^\infty(\R , [0,1])$ satisfying
\begin{align}  \label{eq:chi} \text{$1_{[-1,1]}\leq \chi \leq 1_{[-2,2]}$ and $x\chi'(x)\leq 0$ and set $\chi_C:=\chi(\cdot/C)$  for a  $C>0$}.
\end{align}

 \end{enumerate}

\end{notation} \qed

  The group  $e^{t \mathcal{L}_{\omega }}$ is well defined in $L^2_\rad (\R , \C ^2)$, leaves invariant $L^2_\rad (\R , \R ^2)$ and the terms of the direct sums in \eqref{eq:dirsum} and \eqref{eq:dirsum1}.
The following result is an immediate consequence of a Proposition 8.1 in Krieger and Schlag \cite{KrSch}, since $\mathcal{L}_{\omega } $ as an easy  consequence of Coles and Gustafson \cite{coles}  is  for $0<| p-3|\ll 1$     admissible in the sense indicated in \cite{KrSch}.
\begin{proposition}\label{prop:KrSch} For any fixed $s>3/2$   there is a  constant $C_\omega $
such that
\begin{align} \label{eq:KrSch}
  \| P_c e^{t \mathcal{L}_{\omega }} : L ^{2,s}  (\R , \C ^2) \to  L ^{2,-s} (\R , \C ^2)  \| \le C_\omega \< t \> ^{-\frac{3}{2}} \text{  for all $t\in \R$.}
\end{align}

\end{proposition}\qed

We will need  a variation of the  last result, which we will  be rephrased later and proved as Lemma \ref{lem:smoothest} and  which is an analogue of Lemma 8.7 \cite{CM2109.08108}.
\begin{proposition} \label{lem:smoothest1} For  $s>3/2$ and $\tau >1/2$ there exists a constant $C>0 $ such that
 \begin{align}&   \label{eq:smoothest11}   \left \|   \int   _{0} ^{t   }e^{  (t-t')  \mathcal{L}_{\omega }}P_c(\omega )g(t') dt' \right \| _{L^2( \R ,L^{2,-s} (\R ))  } \le C  \|  g \| _{L^2( \R , L^{2,\tau}_\rad (\R ) ) } \text{ for all $g\in  L^2( \R , L^{2,\tau} (\R ) )$}.
\end{align}
\end{proposition}

We will need the following result, whose proof is based  on   an argument in \cite[Lemma 3.4]{CV}.

\begin{proposition}[Kato smoothing]\label{lem:smooth111} For any $\omega$ and for any $s>1$  there exists $c>$   such that
\begin{equation}\label{eq:smooth111}
   \|   e^{\im t \mathcal{L}_\omega}P_c u_0 \| _{L^2(\R , L ^{2,-s}(\R   ))} \le c \|  u_0 \| _{L^2(\R  )}.
\end{equation}
\end{proposition}

\section{Refined profile, modulation, continuation argument and proof of Theorem \ref{thm:asstab}}\label{sec:mod}

It is well known,  see Weinstein \cite{W2},  that
\begin{align}
  \label{eq:mangs} \mathcal{S}=\left \{  e^{\im \vartheta  } {\phi}_{\omega} : \vartheta \in \R , \omega >0     \right \}
\end{align}
is a symplectic  submanifold  of  $L ^2_\rad (\R, \C ) $.
We set
 \begin{align}
   \label{eq:refprof} \phi[\omega, {z}] = \phi _\omega + \widetilde{\phi}[\omega, {z}]  \text{ with } \widetilde{\phi}[\omega, {z}] :=z \xi +\overline{z} \overline{\xi}.
 \end{align}
For functions $\widetilde{z}_\mathcal{R}$,  $\widetilde{\vartheta}_\mathcal{R}$ and $\widetilde{\omega}_\mathcal{R}$ to be determined below we introduce
\begin{align*}&
  \widetilde{ z}[\omega,z]=\widetilde{z}_0 [\omega,z]+ \widetilde{z}_\mathcal{R}[\omega,z] \text{ with } \widetilde{z}_0[\omega, z]=\im \lambda z\\&\widetilde{\vartheta}[\omega,z]=\omega +  \widetilde{\vartheta}_\mathcal{R}[\omega,z] \text{  and }\widetilde{ \omega}[\omega,z]= \widetilde{\omega}_\mathcal{R}[\omega,z].
\end{align*}

\begin{proposition}\label{prop:refpropf}
There exist  $C^2$  functions  $\widetilde{z}_\mathcal{R}$,  $\widetilde{\vartheta}_\mathcal{R}$ and $\widetilde{\omega}_\mathcal{R}$   defined in a neighborhood of $(\omega _0, 0)\in \R _+\times \C$
with  \begin{align}\label{eq:estpar}
 |\widetilde{\vartheta}_\mathcal{R}| + |\widetilde{\omega}_\mathcal{R}| +|\widetilde{z}_\mathcal{R}|  \lesssim |z|^2
\end{align}
such that, if we set
\begin{align}\label{eq:phi_pre_gali}
R[\omega, {z}]:= \partial ^2_x\phi  [\omega, {z}]+ f(\phi [\omega, {z}]) - \widetilde{\vartheta}\phi [\omega, {z}]+ \im \widetilde{\omega}\partial_{\omega}\phi [\omega, {z}] +\im  D_{z}\phi  [\omega, {z}]\widetilde{z },
\end{align}
we have
\begin{align}&
 \| \cosh \( \kappa \omega  x  \)  {R}   [\omega, {z}]\| _{L^2(\R ) }\lesssim |z|^2 ,\label{estR}
\end{align}
with furthermore the following orthogonality conditions, for $z_1=\Re z$ and $z_2=\Im z$,
\begin{align}\label{R:orth} &
\< \im {R}[\omega, {z}], \phi [\omega, {z}]\>=\< \im {R}[\omega, {z}],\im  \partial_{\omega}\phi [\omega, {z}]\>  =\< \im {R}[\omega, {z}],\im \partial_{z_{ j}}\phi[\omega, {z}]\> = 0,\text{   for all $j=1,2 $.}
\end{align}

 \end{proposition}
\proof  From  \eqref{eq:static}   and
\begin{align*}
     D_{z}\widetilde{\phi}  [\omega, {z}]\widetilde{z } _0 =    \mathcal{L}_\omega \widetilde{\phi}  [\omega, {z}] =
 -\im \( -\partial _x ^2 \widetilde{\phi}  [\omega, {z}]  +\omega \widetilde{\phi}  [\omega, {z}] - Df(\phi _\omega )\widetilde{\phi}  [\omega, {z}] \)
\end{align*}
we obtain
\begin{align*}&
   \im  D_{z} {\phi}  [\omega, {z}]\widetilde{z } _0 =
    -\partial _x ^2 \widetilde{\phi}  [\omega, {z}]  +\omega  {\phi}  [\omega, {z}] - f(\phi [\omega, {z}]) +\widehat{R}   [\omega, {z}]  \text{ where }
    \\& \widehat{R}   [\omega, {z}] := f(\phi [\omega, {z}]) -  f(\phi _\omega)          - Df(\phi _\omega )\widetilde{\phi}  [\omega, {z}] .
\end{align*}
Since
\begin{align*}
  \widehat{R}   [\omega, {z}]= \int _{0} ^{1}\int _{0} ^{1} tD^2f(\phi _\omega +t s \widetilde{\phi}  [\omega, {z}]   ) dtds   \widetilde{\phi}  ^2 [\omega, {z}]
\end{align*}
we conclude that
\begin{align}\label{eq:esthatr}
  \| \cosh \( \kappa \omega  x  \)  \widehat{R}   [\omega, {z}]\| _{L^2(\R ) }\lesssim |z|^2.
\end{align}
Now we set
\begin{align*}
    {R}   [\omega, {z}] =\widehat{R}   [\omega, {z}] -\im  D_{z} {\phi}  [\omega, {z}]\widetilde{z } _\mathcal{R} + \widetilde{\vartheta } _\mathcal{R} {\phi}  [\omega, {z}] - \im \widetilde{\omega}_\mathcal{R}\partial_{\omega}\phi [\omega, {z}] .
\end{align*}
Setting $\phi  =\phi [\omega, {z}]$      and $\widehat{{R}}=\widehat{{R}}[\omega, {z}]$,
 the orthogonality conditions  \eqref{R:orth}
are equivalent to
\begin{align}\label{eqsysta}
   \mathbf{A}     \begin{pmatrix}
 \widetilde{\vartheta } _\mathcal{R} \\   \widetilde{\omega}_\mathcal{R}  \\ \widetilde{z } _\mathcal{R}
\end{pmatrix} = - \begin{pmatrix}
\<  \widehat{{R}}, \im \phi\> \\   \<  \widehat{{R}}, \partial _\omega \phi\>  \\   \< \widehat{{R}} , D _ z  \phi\> .
\end{pmatrix}
\end{align}
where, by $ D _{  z _1} \widetilde{\phi} =2 (  \xi _1 , 0)^\intercal $  and  $ D _{  z _2} \widetilde{\phi} =-2( 0, \Im  \xi _2 )^\intercal $,  with the right hand sides  defined in  \eqref{eq:reimxi},
\begin{align}\label{eq:matrixa}
  \mathbf{A} &= \begin{pmatrix}
 \cancel{  \< \phi , \im \phi \> } & -\< \partial _\omega \phi ,   \phi \>   & \cancel{ -\< D _ z  \phi ,   \phi \> } \\  \< \phi , \partial _\omega \phi \> & \cancel{- \< \im \partial _\omega \phi ,  \partial _\omega \phi  \>  }  &  -\<  \im D _ z  \phi ,   \partial _\omega \phi\> \\ \cancel{  \< \phi , D _ z  \phi \> } & -\< \im \partial _\omega \phi , D _ z  \phi \>   &  -\< \im D _ z  \phi ,   D _ z  \phi\>
\end{pmatrix}   \\&  =  \begin{pmatrix}
  \mathbf{q}'(\omega ) J ^{-1}  &   O(z)   \\ O(z)    & \left . \< J   D _{  z _i} \widetilde{\phi} ,  D _{  z _j}   \widetilde{\phi}\>  \right | _{i,j =1,2}
\end{pmatrix}  + o(z)  =      \begin{pmatrix}
  \mathbf{q}'(\omega ) J ^{-1}  &   0   \\ 0    &    J
\end{pmatrix}  + O(z)   ,\nonumber
\end{align}
where the cancelled terms are null.
Since  $ \left . \mathbf{A} \right | _{z=0} $  is invertible, we conclude that   $\mathbf{A}$ is invertible also for small $z$.  From \eqref{eq:esthatr}  and   \eqref{eqsysta} we obtain  \eqref{eq:estpar}  and \eqref{estR}.

\qed

The proof of the following   modulation is standard, see Stuart \cite{stuart}.

\begin{lemma}[Modulation]\label{lem:mod1}  Let $\omega _0 >0$. There exists an $\delta _0 >0$ and functions
$\omega   \in C^1(   \mathcal{U} (\omega _0,\delta _0   ), \R) $ and
$\vartheta \in C^1( \mathcal{U} (\omega _0,\delta _0   ), \R /\Z ) $  and $z \in C^1( \mathcal{U} (\omega _0,\delta _0   ), \C ) $
 such that for  any $u \in \mathcal{U} (\omega _0,\delta _0   )$
\small \begin{align}\label{61}
  & \eta (u):=e^{-\im \vartheta( u) } u -  \phi [\omega( u) , z(u)]    \text{ satisfies }
    \\& \nonumber
\< \eta (u),\im\phi [\omega ( u), {z}( u)]\>=\< \eta (u), \partial_{\omega}\phi [\omega( u), {z}( u)]\>  =\< \eta (u)],\partial_{z_{ j}}\phi[\omega ( u), {z}( u)]\> = 0,\text{   for all $j=1,2 $.}
\end{align}
 \normalsize
Furthermore we have the identities  $\omega  ({\phi}_{\omega })=\omega $,  $\vartheta  (  e^{\im \vartheta _0}  u)    = \vartheta  (    u)   + \vartheta _0 $   and $\omega  (  e^{\im \vartheta _0}  u)    = \omega  (    u)    $ and  $z  (  e^{\im \vartheta _0}  u)    = z  (    u)    $.
\end{lemma}
\qed

We  have now the ansatz
\begin{align}\label{eq:ansatz}
  u = e^{\im \vartheta} \(  \phi [\omega  , z ]+ \eta \) .
\end{align}
By orbital stability we can assume that there exists $\theta =\theta (t)$ such that
\begin{equation}\label{eq:orbsta1}
 \| u- e^{\im \theta} \phi _{\omega  _0}\| _{H^1} <\delta \text{  for all values of time.}
\end{equation}
Then, using the notation in \eqref{eq:notation1} and the fact that      $\mathbf{d}'(\omega )= \mathbf{q}(\omega)$
 it is standard to write
\begin{align}\nonumber &
 O(\delta ) =  \mathbf{E}(u) +\omega \mathbf{Q}(u)- \mathbf{e}(\omega _0 ) -\omega \mathbf{q}(\omega _0) = \mathbf{d}(\omega ) - \mathbf{d}(\omega _0)- \mathbf{d}'(\omega _0) (\omega -\omega _0) \\& +2^{-1}\< \(\mathbf{d}^2\mathbf{E}(\phi _{\omega}  )+\omega \mathbf{d}^2\mathbf{Q}(\phi _{\omega}  )\) r , r   \> +o\(\| r\| _{H^1}^2 \) + o\(\| r\| _{H^1} (\omega -\omega _0)\) , \label{eq:enrexp}
\end{align}
for $r= z \xi +\overline{z} \overline{\xi} +\eta$.  Now we have
\begin{align}   &
  \< \(d^2\mathbf{E}(\phi _{\omega}  )+\omega d^2\mathbf{Q}(\phi _{\omega}  )\) r , r   \> = \<  \mathcal{L}_{\omega} r , Jr   \>  =2\lambda |z|^2 + \<  \mathcal{L}_{\omega} \eta , J\eta   \>   \label{eq:enrexp1}
\end{align}
Since $\<  \mathcal{L}_{\omega} \eta , J\eta   \> \gtrsim  \| \eta \| _{H^1}^2$,  from the above  and the strict convexity of $\mathbf{d}(\omega )$ we conclude that
\begin{align}
  \label{eq:oertstab1} |\omega -\omega _0|+ |z|+ \| \eta \| _{H^1}\lesssim \sqrt{\delta } \text{   for all values of time}.
\end{align}
We will set
\begin{equation}\label{eq:dismodtheta}
  \Theta :=( \vartheta , \omega ,z) , \quad  \widetilde{\Theta} :=( \widetilde{\vartheta} , \widetilde{\omega} ,\widetilde{z}) \text{ and }\widetilde{\Theta}_{\mathcal{R}} :=( \widetilde{\vartheta}_{\mathcal{R}} , \widetilde{\omega} _{\mathcal{R}},\widetilde{z}_{\mathcal{R}})  .
\end{equation}
The proof of Theorem \ref{thm:asstab} is mainly  based on the following continuation argument.

\begin{proposition}\label{prop:continuation}
There exists  a    $\delta _0= \delta _0(\epsilon )   $ s.t.\  if
\begin{align}
  \label{eq:main2} \| \eta \| _{L^2(I, \Sigma _A )} +  \| \eta \| _{L^2(I, \widetilde{\Sigma}  )} + \| \dot \Theta - \widetilde{\Theta} \| _{L^2(I  )} + \| z^2 \| _{L^2(I  )}\le \epsilon
\end{align}
holds  for $I=[0,T]$ for some $T>0$ and for $\delta  \in (0, \delta _0)$
then in fact for $I=[0,T]$    inequality   \eqref{eq:main2} holds   for   $\epsilon$ replaced by $   o _{\varepsilon }(1) \epsilon $.
\end{proposition}
Notice that this implies that in fact the result is true for $I=\R _+$.   We will split the proof of Proposition \ref{prop:continuation} in a number of partial results obtained assuming the hypotheses of Proposition  \ref{prop:continuation}.
\begin{proposition}\label{prop:modpar} We have
  \begin{align}&
  \label{eq:modpar1}  \|\dot \vartheta -\widetilde{\vartheta} \|  _{L^1(I  )} + \|\dot \omega -\widetilde{\omega} \|  _{L^1(I  )}\lesssim  \epsilon ^2 ,\\&  \label{eq:modpar2}   \|\dot z -\widetilde{z} \|  _{L^2(I  )} \lesssim \sqrt{ \delta}\epsilon  ,\\&  \label{eq:modpar3}   \|\dot z \|  _{L^\infty(I  )} \lesssim  \sqrt{\delta} .
\end{align}

\end{proposition}

\begin{proposition}[Fermi Golden Rule (FGR) estimate]\label{prop:FGR}
 We have
 \begin{align}\label{eq:FGRint}
  \| z^2\|_{L^2(I)}\lesssim  A^{-1/2} \epsilon .
 \end{align}
 \end{proposition}

\begin{proposition}[Virial Inequality]\label{prop:1virial}
 We have
 \begin{align}\label{eq:sec:1virial1}
   \| \eta \| _{L^2(I, \Sigma _A )} \lesssim  A \delta + \| z^2\|_{L^2(I)} + \| \eta \| _{L^2(I, \widetilde{\Sigma}   )} + \epsilon ^2  .
 \end{align}
 \end{proposition}

\begin{proposition}[Smoothing Inequality]\label{prop:smooth11}
 We have
 \begin{align}\label{eq:sec:smooth11}
   \| \eta \| _{L^2(I, \widetilde{\Sigma}   )} \lesssim  o _{B^{-1}} (1)  \epsilon    .
 \end{align}
 \end{proposition}

%

\textit{Proof of Theorem \ref{thm:asstab}.}
It is straightforward that Propositions \ref{prop:modpar}--\ref{prop:smooth11}  imply Proposition \ref{prop:continuation} and thus the fact that we can take $I=\R _+$ in all the above inequalities. This in particular implies \eqref{eq:asstab2}. By $z\in L^4(\R  _+)$ and $\dot z \in L^\infty (\R  _+)$ we have  \eqref{eq:asstab3}.  

\noindent We next focus on the limit \eqref{eq:asstab20}. We first rewrite our equation,
entering the  ansatz \eqref{eq:ansatz} in \eqref{eq:nls1},  to obtain, for $\phi = \phi [\omega  , z ]$,
 \begin{align*}
  \im \dot \eta - \dot \vartheta \eta -   \dot \vartheta  \phi + \im  \dot \omega \partial _\omega \phi + \im D_z \phi \dot z =  -\partial _x^2  \eta   -\partial _x^2  \phi  -f(\phi  + \eta) .
 \end{align*}
 Then, adding and subtracting and using  \eqref{eq:phi_pre_gali}, for $R=R[\omega, {z}]$ we obtain
\begin{align}\label{eq:nls2}&
  \im \dot \eta   -(\dot \vartheta  -\widetilde{\vartheta} ) \eta -   (\dot \vartheta  -\widetilde{\vartheta} ) {\phi}   + \im  ( \dot \omega   - \widetilde{ \omega} ) \partial _\omega \phi + \im D_z \phi (\dot z -\widetilde{z}) = -\partial _x^2 \eta +\widetilde{\vartheta} \eta       \\& \nonumber - \( f( \phi  + \eta  )  -  f( \phi    ) \) - R    \\&      {   -\partial _x^2\phi  -  f( \phi    )  +\widetilde{\vartheta} \phi -  \im \widetilde{\omega}\partial_{\omega}\phi   -\im  D_{z}\phi   \widetilde{z } +R }    ,  \nonumber
\end{align}
where the last line equals 0,  because of the definition of $R$ in   \eqref{eq:phi_pre_gali}. Equation \eqref {eq:nls2}   rewrites \small
\begin{align}\label{eq:nls3}&
   \dot \eta   +\im (\dot \vartheta  -\widetilde{\vartheta} ) \eta +e^{-\im \vartheta}    D _{\Theta} \phi [\Theta ] \(  \dot \Theta  -\widetilde{\Theta}   \)      \\& =  \im  \(\partial _x^2    +D f( \phi  [\omega  , z ]  )  \) \eta-\im  \widetilde{\vartheta}   \eta  +  \im  \( f( \phi [\omega  , z ] + \eta  )  -  f( \phi  [\omega  , z ]  ) -D f( \phi  [\omega  , z ]  ) \eta\) +\im R[\omega  , z ] . \nonumber
\end{align}\normalsize
for $\phi [\Theta]  =  e^{ \im \vartheta} \phi[\omega, {z}]$.
 Let
  \begin{equation*}
  \begin{aligned}
   \mathbf{a}(t) &:= 2 ^{-1}\|  e^{- a\< x\>}   \eta (t)  \| _{L^2(\R )} ^2
  .
  \end{aligned}
  \end{equation*}
Then by the Orbital Stability
\begin{align}\label{eq:intparts}
  \dot {\mathbf{a}}  &= -\frac{1}{2} \< \left [   e^{- 2a\< x\>} , \im  \partial _x ^2  \right ] \eta , \eta \>
  \\&
   -\<
  e^{- a\< x\>}  \( \im  \dot \vartheta   \eta +e^{-\im \vartheta}     D _{\Theta} \phi [\Theta ] \(  \dot \Theta  -\widetilde{\Theta}   \)  \) , e^{- a\< x\>} \eta    \> \nonumber  \\& + \<
  e^{- a\< x\>}  \(     \im  \( f( \phi [\omega  , z ] + \eta  )  -  f( \phi  [\omega  , z ]  )  \) +\im R[\omega  , z ] \) , e^{- a\< x\>} \eta    \> =O( \epsilon ^2 ) \text{ for all times}.\nonumber
\end{align}
   Since we already know from \eqref{eq:asstab2} that $\mathbf{a}\in L^1(\R )$, we conclude that
$ \mathbf{a}(t) \xrightarrow{t\to +\infty} 0$. Notice that the integration by parts in \eqref{eq:intparts} can be made rigorous  considering that
if $u_0\in H^2(\R )$ by the well known regularity result by Kato, see \cite{CazSemi}, we have $\eta \in C^0\( \R , H^2 (\R )\)$ and the above argument is correct and by a standard  density argument the result  can be extended to $u_0\in H^1(\R )$.

Finally, we prove \eqref{eq:asstab3}.
Since $\mathbf{Q}(\phi_{\omega})=\mathbf{q}(\omega)$ is monotonic, it suffices to show $\mathbf{Q}(\phi_{\omega})$ converges as $t\to\infty$.
Next, from the conservation of $\mathbf{Q}$, the exponential decay of $\phi[\omega,z]$, \eqref{eq:asstab20} and \eqref{eq:asstab2}, we have
\begin{align*}
\lim_{t\to \infty}\(\mathbf{Q}(u_0)-\mathbf{Q}(\phi_{\omega (t)})-\mathbf{Q}(\eta (t))\)=0.
\end{align*}
Here, notice that by orbital stability we can  take  $a>0$ such that we have the  following, which will be used below,
\begin{equation*}
  e^{-2a\<x\>}\gtrsim    \max \{ |\phi[\omega (t) ,z(t)]|, |\phi[\omega (t),z(t) ]| ^{p-1} , |\phi[\omega(t) ,z(t)]| ^{p-2}\}  \text {  for all } t\in \R .
\end{equation*}
Thus, our task is now to prove $\frac{d}{dt}\mathbf{Q}(\eta) \in L^1$, which is sufficient to show the convergence of $\mathbf{Q}(\eta)$.
Now, from \eqref{eq:nls3}, we have
\begin{align*}
\frac{d}{dt}\mathbf{Q}(\eta)&=\<\eta,\dot{\eta}\>
=-\<\eta,e^{-\im \vartheta}     D _{\Theta} \phi [\Theta ] \(  \dot \Theta  -\widetilde{\Theta}   \) \>    +\<\eta,  \im  D f( \phi  [\omega  , z ]  )   \eta \>\\&\quad+\<\eta,  \im  \( f( \phi [\omega  , z ] + \eta  )  -  f( \phi  [\omega  , z ]  ) -D f( \phi  [\omega  , z ]  ) \eta\)\>+\<\eta, \im R[\omega  , z ]\>\\&
=I+II+III+IV.
\end{align*}
By the bound of the 1st and the 3rd term of \eqref{eq:main2}, we have $I\in L^1(\R _+)$.
Next, by \eqref{eq:derf1},  
we have $|Df(\phi[\omega,z])\eta|\lesssim e^{-a\<x\>}|\eta|$ for some small $a$.
Therefore,   $II\in L^1(\R _+)$.
$IV\in L^1(\R _+)$ follows from \eqref{estR} and $|z|^2\in L^2(\R _+)$.
For $III$, since $f$ is at least $C^2$ (we only consider $p>2$) we have
In analogy to similar partitions in \cite{CMMS23} which allow to offset the lack of differentiability of $f(u)$,  we partition the line  where $x$ lives as
\begin{align*}&
   \Omega_{1,t,s}=\{x\in \R\ |\  |s\eta(t,x)|\leq 2|\phi[\omega (t) ,z(t)]|\}  \text{ and }\\&  \Omega_{2,t,s}=\R\setminus \Omega_{1,t,s}=\{x\in \R\ |\ |s\eta(t,x)|> 2|\phi[\omega (t),z(t)]|\} ,
\end{align*}
Then, we have
\begin{align*}
III(t)=\sum_{j=1,2}\Re \int_0^1 ds \int_{\Omega_{j,t,s}}\im\overline{\eta}(t)  D^2 f(\phi[\omega (t),z(t)]+s\eta)(\eta (t),\eta (t))\,dx=:III_1(t)+III_2(t).
\end{align*}
For $III_1$, by $\|\eta\|_{L^\infty}\lesssim \epsilon\leq 1$   and by \eqref{eq:derf2}, we have
\begin{align*}
|III_1(t)|\lesssim \int_0^1ds\int_{\Omega_{1,t,s}} |\eta (t)| (|\phi[\omega (t),z(t)] |+|s\eta|)^{p-2}|\eta|^2\,dx \lesssim \int_{\R} |\phi[\omega (t),z(t)]|^{p-2}|\eta (t)|^2\,dx .
\end{align*}
Thus, we see $III_1\in L^1(\R _+ )$.
For $III_2$, since $|s\eta+\phi[\omega,z]|>\frac{1}{2}|s\eta|>0$
we can expand the integrand as
\begin{align*}
D^2f(\phi[\omega (t),z(t)]+s\eta (t) )(\eta  (t),\eta (t) )&=D^2f(s\eta (t) )(\eta (t) ,\eta (t) )\\& +\int_0^1D^3f(\tau\phi[\omega (t),z(t)]+s\eta (t))(\phi[\omega,z],\eta (t) ,\eta (t)) d\tau .
\end{align*}
Now, since $D^2f(s\eta)(\eta,\eta)=s^{p-2}p(p-1)|\eta|^{p-1}\eta$, we have
\begin{align*}
\Re \int_0^1 ds \int_{\Omega_{2,t, s}}\im\overline{\eta}  D^2f(s\eta (t) )(\eta (t) ,\eta (t) ) \,dx =0,
\end{align*}
because the integrand becomes purely imaginary.
Therefore, from the bound
\begin{align*}
 |D^3f(\psi)(w_1,w_2,w_3)|\lesssim |\psi|^{p-3}|w_1||w_2||w_3|
\end{align*}
and from
 \begin{align*}
   |\tau\phi[\omega (t) ,z (t) ]+s\eta (t) |\sim |s\eta (t) | \text{ for $\tau\in [0,1]$ in } \Omega_{2,t, s}
 \end{align*}
  we have
\begin{align*}
|III_2(t)|&\lesssim \int_0^1   ds \int_{\Omega_{2,t, s}} |\eta   (t) | |s\eta (t)|^{p-3}|\phi[\omega (t) ,z(t) ]||\eta (t)|^2\,dx \\& \lesssim \int_0^1 s^{p-3} \,ds \int _{\R }|\phi[\omega (t) ,z (t)]||\eta (t)|^2\,dx.
\end{align*}
Since    $p>2$  we have $\int_0^1s^{p-3}\,ds<\infty$ and  we see $III_2\in L^1(\R _+)$.
Therefore, we have the conclusion.
\qed

\section{Proof of Proposition  \ref{prop:modpar} }\label{sec:modpar}

\begin{lemma}\label{lem:lemdscrt} We have the estimates
 \begin{align} \label{eq:discrest1}
   &|\dot \vartheta -\widetilde{\vartheta} | + |\dot \omega -\widetilde{\omega} |  \lesssim  \(|z|^2+    \|   \eta \| _{\widetilde{\Sigma}}     \)   \|   \eta \| _{\widetilde{\Sigma}}    \\&  |\dot z -\widetilde{z} | \lesssim   \(|z| +    \|   \eta \| _{\widetilde{\Sigma}}     \)  \|   \eta \| _{\widetilde{\Sigma}}   .   \label{eq:discrest2}
\end{align}

\end{lemma}
\proof
Applying $\< \cdot ,\im e^{-\im \vartheta} D_{\Theta} \phi [\Theta ]   \boldsymbol{\Theta } \>  $  with $\boldsymbol{\Theta }\in \R ^4$   to \eqref{eq:nls3} and by the cancelations \eqref{R:orth},   we get   \small
\begin{align}\nonumber   &\< D _{\Theta} \phi [\Theta ] (\dot \Theta -\widetilde{\Theta} ) , \im  D _{\Theta} \phi [\Theta ]   \boldsymbol{\Theta } \>
  -      \<\eta , \im  e^{-\im \vartheta}    D _{\Theta}^2\phi [\Theta ]  ( \boldsymbol{\Theta } ,   \dot \Theta -\widetilde{\Theta}  ) \>  \\& \nonumber  \cancel {-  \<\eta ,    e^{-\im \vartheta}D _{\Theta}   \phi [\Theta ]   \boldsymbol{\Theta } \>  \widetilde{\vartheta} } -  \<\eta , \im  e^{-\im \vartheta}    D _{\Theta}^2\phi [\Theta ]  ( \boldsymbol{\Theta } ,   \widetilde{\Theta}  ) \> = \cancel {-    \widetilde{\vartheta}    \<     \eta,    e^{-\im \vartheta} D _{\Theta} \phi [\Theta ]   \boldsymbol{\Theta } \> }  \\&  + \< \eta ,   \(\partial _x^2    +D f( \phi  [\omega  , z ]  )  \)  e^{-\im \vartheta} D _{\Theta} \phi [\Theta ]   \boldsymbol{\Theta } \>   -\<     f( \phi  + \eta  )  -  f( \phi    ) -D f( \phi    ) \eta ,    e^{-\im \vartheta} D _{\Theta} \phi [\Theta ]   \boldsymbol{\Theta } \>   \label{eq:est1}.
   \end{align}
\normalsize
Setting also  $R[\Theta]  =  e^{ \im \vartheta} R[\omega, {z}]$, notice that equation  \eqref{eq:phi_pre_gali} can be written as
\begin{align*}
  \partial ^2_x\phi  [\Theta]+ f(\phi [\Theta]) +\im  D _{\Theta}   \phi [\Theta ]  \widetilde{\Theta}  =R[\Theta] .
\end{align*}
Differentiating  in $\Theta$,   we obtain
\begin{align*}
  \( \partial ^2_x   + Df(\phi [\Theta]) \) D _{\Theta} \phi [\Theta ]   \boldsymbol{\Theta } +\im  D _{\Theta}^2   \phi [\Theta ]  (\widetilde{\Theta}  , \boldsymbol{\Theta })+  \im  D _{\Theta}   \phi [\Theta ]  D _{\Theta}  \widetilde{\Theta} \boldsymbol{\Theta }   =D _{\Theta} R[\Theta] \boldsymbol{\Theta }.
\end{align*}
Since from $f( e^{ \im \vartheta} \phi)  = e^{ \im \vartheta} f(   \phi)$ we have $e^{ -\im \vartheta} D f( e^{ \im \vartheta} \phi) X=    D f(  \phi) e^{ -\im \vartheta} X$,  we get  \small
\begin{align*}&
      \<     \eta,  \(\partial _x^2    +D f( \phi )  \)  e^{-\im \vartheta} D _{\Theta} \phi [\Theta ]   \boldsymbol{\Theta } \>          =   \<  e^{ \im \vartheta}\eta ,  \(  \partial ^2_x + Df(\phi [\Theta])  \) D _{\Theta} \phi [\Theta ]   \boldsymbol{\Theta } \>   \\& = -\<  e^{ \im \vartheta}\eta , \im  D _{\Theta}^2   \phi [\Theta ]  (\widetilde{\Theta}  , \boldsymbol{\Theta })+ \cancel{ \im  D _{\Theta}   \phi [\Theta ]  D _{\Theta}  \widetilde{\Theta} \boldsymbol{\Theta }} -D _{\Theta} R[\Theta] \boldsymbol{\Theta }\>
\end{align*}
\normalsize  with cancellation following by  the modulation orthogonality \eqref{61}.
Entering this inside \eqref{eq:est1}   yields \small
\begin{align*}\nonumber   &\< D _{\Theta} \phi [\Theta ] (\dot \Theta -\widetilde{\Theta} ) , \im  D _{\Theta} \phi [\Theta ]   \boldsymbol{\Theta } \>
-      \<\eta , \im  e^{-\im \vartheta}    D _{\Theta}^2\phi [\Theta ]  ( \boldsymbol{\Theta } ,   \dot \Theta -\widetilde{\Theta}  ) \>  \\& \nonumber    \cancel{-  \<\eta , \im  e^{-\im \vartheta}    D _{\Theta}^2\phi [\Theta ]  ( \boldsymbol{\Theta } ,   \widetilde{\Theta}  ) \>} = \cancel{-\<  e^{ \im \vartheta}\eta , \im  D _{\Theta}^2   \phi [\Theta ]  (\widetilde{\Theta}  , \boldsymbol{\Theta })\> } +
\<  e^{ \im \vartheta}\eta ,
 D _{\Theta} R[\Theta] \boldsymbol{\Theta }\>   \\&    -\<     f( \phi  + \eta  )  -  f( \phi    ) -D f( \phi    ) \eta ,    e^{-\im \vartheta} D _{\Theta} \phi [\Theta ]   \boldsymbol{\Theta } \>,   \nonumber
   \end{align*}\normalsize
 where the cancellation  is obvious, since we have equal terms. So, from this  we get
\begin{align*}   &
   -  (\dot \omega -\widetilde{\omega}   )  \< \partial _\omega \phi , \phi \> + \< D _{z} \phi (\dot z -\widetilde{z})  , \phi \> + O\(   \|   \eta \| _{\widetilde{\Sigma}}       |\dot \Theta -\widetilde{\Theta} |\) \\& = \<   \eta ,
  \im  R   \>        -\<     f( \phi  + \eta  )  -  f( \phi    ) -D f( \phi    ) \eta ,    \im  \phi  \>
\end{align*}
which implies
\begin{align} \label{eq:omega}  &
   -  (\dot \omega -\widetilde{\omega}   )  \< \partial _\omega \phi , \phi \> + O\(    | z|    |\dot z -\widetilde{z}|   \)  + O\(  \|   \eta \| _{\widetilde{\Sigma}}        |\dot \Theta -\widetilde{\Theta} |\) \\& =   O\(   \|   \eta \| _{\widetilde{\Sigma}}    |z|^2\)  +  O\(  \|   \eta \| _{\widetilde{\Sigma}}    ^2\)  .\nonumber
\end{align}
Similarly, using   $
  \| \cosh \( \kappa \omega  x  \)  \partial _\omega \widehat{R}   [\omega, {z}]\| _{L^2(\R ) }\lesssim |z|^2,
 $
\begin{align} \label{eq:theta}  &
      (\dot \vartheta -\widetilde{\vartheta}   )  \< \phi , \partial _\omega  \phi \> +  O\(    | z|    |\dot z -\widetilde{z}|   \)   + O\(   \|   \eta \| _{\widetilde{\Sigma}}        |\dot \Theta -\widetilde{\Theta} |\) \\& =   O\(   \|   \eta \| _{\widetilde{\Sigma}}   |z|^2\)  +  O\(   \|   \eta \| _{\widetilde{\Sigma}}    ^2\)  .\nonumber
\end{align}
Finally we get the following which along the other formulas yields the lemma
\begin{align}  \label{eq:z}  &
    \< D_z\phi  (\dot z -\widetilde{z}) ,\im  \partial _{z_j}  \phi \>   +  O\(    | z|     |\dot \Theta -\widetilde{\Theta} |   \)    =  \<   \eta ,
    \partial _{z_j}  R   \>   +  O\(  \|   \eta \| _{\widetilde{\Sigma}}    ^2\) .
\end{align}
 \qed

\textit{Proof of Proposition  \ref{prop:modpar}.}
Lemma \ref{lem:lemdscrt} and \eqref{eq:main2} imply immediately \eqref{eq:modpar1}--\eqref{eq:modpar2}. Entering this,  \eqref{eq:estpar}  and \eqref{eq:oertstab1} in \eqref{eq:z}  we obtain \eqref{eq:modpar3}.

\qed

\section{The Fermi Golden Rule: proof of Proposition \ref{prop:FGR}}\label{sec:fgr}

The nonlinear Fermi Golden Rule (FGR) was an idea initiated by Sigal \cite{sigal} and further developed by Buslaev and Perelman \cite{BP2} and by Soffer and Weinstein \cite{SW3}. More complicated configurations were discussed in \cite{Cu1}, where the deep connection of the FGR with the Hamiltonian nature of the NLS was clarified. This comes about  because the FGR has to do with the fact that the integral of certain coefficients on appropriate spheres of the phase space associated to $\mathcal{L}_\omega$ are strictly  positive. The positivity is due to the fact that the coefficients are essentially squares, that is the product of pairs of factors which are complex conjugates to each other. It turns out that the factors are like this thanks  to the Hamiltonian structure of the NLS, which   gives relations between coefficients of the system, since they  are partial derivatives of a fixed  given function, the Hamiltonian $\mathbf{E}$, see \cite[pp. 287--288]{Cu1} for a heuristical explanation.  The rigorous argument in \cite{Cu1} needed various changes of variables, to get into canonical coordinates and normal forms. However the notion of Refined  Profile and the related modulation ansatz   provide a framework to prove the FGR in a direct way, without any need of a search  of canonical coordinates and of normal forms, see for example
  \cite{CM3}. The proof involves differentiating a Lyapunov functional which lately in the literature, especially  for space dimension 1,   is simpler than what would be analogous here to the energy $\mathbf{E}(\phi [\Theta ])$  used up until \cite{CM3,CM2109.08108}. A good reference for the simpler  Lyapunov functional is  Kowalczyk and Martel \cite{KM22} where the spectrum is rather simple while   for  a version with a more complicated spectral configuration we refer to \cite{CMS23}.
   In our current paper the spectrum is  like in Kowalczyk and Martel \cite{KM22} and Martel \cite{Martel2} and involves, using function and notation in \eqref{eq:chi},  the   functional
\begin{align}\label{eq:FGRfunctional}
\mathcal{J}_{\mathrm{FGR}}:=\< J   {\eta},\chi_A \(  {z}^{2} g ^{(\omega)}+ \overline{{z}}^{2} \overline{g}^{(\omega)} \)      \>  ,
\end{align}
with a nonzero  $g ^{(\omega)}\in L^\infty(\R , \C ^2)$  satisfying
\begin{align}
  \label{eq:eqsatg2} \mathcal{L}_{\omega}g ^{(\omega)}=2\im \lambda (p, \omega) g ^{(\omega)}.
\end{align}
That $g^{(\omega)} $ exists is known since   Krieger and Schlag \cite{KrSch}, see  Lemma 6.3, or earlier Buslaev and Perelman \cite{BP1}.
Notice that  if $g   $  solves \eqref{eq:eqsatg2}  for $\omega=1$ then $g ^{(\omega)} (x):= g\( \sqrt{\omega}x \)$ solves
\eqref{eq:eqsatg2}, where $\lambda (p, \omega) =\omega \lambda (p, 1)$.
We define the FGR constant $\gamma(\omega,p)$,   for $g ^{(\omega)} =\( g_1 ^{(\omega)} ,g_2 ^{(\omega)}   \) ^\intercal $,    by
\begin{equation}\label{eq:fgrgamma}
   \gamma ( p, \omega)  :=     \<    \phi _\omega ^{p-2}  \( p  \xi _1  ^2+  \xi _2  ^2 \) , g_1 ^{(\omega)} \> +2 \<    \phi _\omega ^{p-2}    \xi _1  \xi _2 , g_2  ^{(\omega)}\>.
\end{equation}
The non-degeneracy of this constant, which is usually assumed, but proved in this paper, is important.

\begin{lemma}\label{lem:FGRnondeg}
For $|p-3|\ll1$, we can choose $g^{(\omega)}$ so that $\gamma(\omega,p)\neq 0$.
\end{lemma}

The proof of Lemma \ref{lem:FGRnondeg} is given in section \ref{sec:FGRconst}.
 Notice that once we have have $\gamma (p, \omega ) \neq 0$,   we can multiply $g$ by a constant to get
 \begin{equation}\label{eq:normalizfgrc}
     (p-1)    \gamma (p, \omega)=1 .
 \end{equation}
    In the next lemma we will need the following
reformulation of equation \eqref{eq:nls3}, where we identify $J=-\im $,
\small
\begin{align}\label{eq:nls4}
   \dot \eta   &=   \mathcal{L}_{\omega}\eta  -  J  ( \widetilde{\vartheta}  _{\mathcal{R}} + \widetilde{\vartheta} - \dot \vartheta) \eta  -   e^{J \vartheta} D_\Theta \phi [\Theta]  (\dot \Theta  -\widetilde{\Theta})   +J  \( D f( \phi  [\omega  , z ]  ) - D f( \phi   _{ \omega }   )\) \eta
    \\&-J \( f( \phi [\omega  , z ] + \eta  )  -  f( \phi  [\omega  , z ]  ) -D f( \phi  [\omega  , z ]  ) \eta\) -J R[\omega  , z ].  \nonumber
\end{align}\normalsize
We have  the following.
\begin{lemma}\label{lem:FGR1}
There is a  $C^1 $  in time    function $\mathcal{I}_{\mathrm{FGR}}$, which satisfies $|\mathcal{I}_{\mathrm{FGR}}|\lesssim  \sqrt{\delta}$
such that
\begin{align}
\left|\dot{\mathcal{J}}_{\mathrm{FGR}}  + \dot{\mathcal{I}}_{\mathrm{FGR}}  -|z|^4     \right |    \lesssim  A ^{-1}   \(  |z|^4   + \| \eta \| _{\Sigma _A}^{2}+ \| \eta \| ^{2}_{\widetilde{\Sigma}}  \)    . \label{eq:lem:FGR11}
\end{align}
\end{lemma}
\proof
Differentiating $\mathcal{J}_{\mathrm{FGR}}$, we have \small
\begin{align*}
\dot{\mathcal{J}}_{\mathrm{FGR}}=&
\<  J  \dot{  \eta },\chi_A \(  {z}^{2} g ^{(\omega)}+ \overline{ z }^{2} \overline{g} ^{(\omega)}\)      \>
+\< J    \eta , \chi_A \( 2z    \widetilde{z}  g^{(\omega)} +  2\overline{z}   \overline{\widetilde{z}}  \overline{g} ^{(\omega)}\)            \>   \\&+ \<  J    \eta , \chi_A \( 2z  (\dot{z}  - \widetilde{z}) g ^{(\omega)}+  2\overline{z}  (\dot{\overline{z}}  - \overline{\widetilde{z}}) \overline{g}^{(\omega)} \)            \> \\&  + \<  J    \eta , \chi_A \({z}^{2} \partial _\omega g ^{(\omega)}+ \overline{ z }^{2}\partial _\omega \overline{g} ^{(\omega)}\)      \> (\dot \omega -\widetilde{\omega}) +    \<  J    \eta , \chi_A \({z}^{2} \partial _\omega g ^{(\omega)}+ \overline{ z }^{2}\partial _\omega \overline{g} ^{(\omega)}\)      \>  \widetilde{\omega}
 \\&  =:A_1+A_2+A_3 +A_4+A_5  .\nonumber
\end{align*}
\normalsize
We consider first the last three terms, the simplest ones.
By \eqref{eq:discrest2} we have
\begin{align*}
|A_3|&\lesssim
 \| {\eta}\chi_A\|_{L^1}|z | \   |\dot{z}  - \widetilde{z}| \lesssim     A ^{\frac{3}{2}}   A^{-1} \| \sech \( \frac{2}{A}x\) {\eta}\|_{L^2}  |\dot{z}  - \widetilde{z}|  \lesssim \sqrt{\delta} A ^{\frac{3}{2}}
 \( \| \eta \| _{\Sigma _A}^{2}+  \| \eta \| _{\widetilde{\Sigma}  }^{2} \) .
\end{align*}
Since $\|  \partial _\omega g ^{(\omega)} \|_{L^\infty}\lesssim \<x \>$, using also \eqref{eq:discrest1}  we have
  \begin{align*}
|A_4|&\lesssim A
 \| {\eta}\chi_A\|_{L^1}|z | ^4   |\dot{\omega}  - \widetilde{\omega}| \lesssim     A ^{\frac{5}{2}}  |z | ^4  \| \eta \| _{\Sigma _A}   |\dot{\omega}  - \widetilde{\omega}| \\&   \lesssim  {\delta}^5  A ^{\frac{5}{2}}
 \( \| \eta \| _{\Sigma _A}^{2}+   \| \eta \| _{\widetilde{\Sigma}  }^{2}  \) .
\end{align*}
Finally, using \eqref{eq:estpar}  we have
\begin{align*}
|A_5|&\lesssim A
 \| {\eta}\chi_A\|_{L^1}|z | ^2   |  \widetilde{\omega}| \lesssim     A ^{\frac{5}{2}}  |z | ^4  \| \eta \| _{\Sigma _A}      \lesssim  {\delta}^2  A ^{\frac{5}{2}}
 \( \| \eta \| _{\Sigma _A}^{2}+  |z | ^4   \) .
\end{align*}
Turning to the main terms,
we have
\begin{align*}
   A_2 & =  \<  J \eta , 2\im \lambda \chi_A \( z^2  g ^{(\omega)}- \overline{z} ^2     \overline{g} ^{(\omega)}\)            \> + \<  J    \eta , \chi_A \( 2z     \widetilde{z}_{\mathcal{R}}  g ^{(\omega)}+  2\overline{z}  \ \overline{\widetilde{z}}_{\mathcal{R}} \ \overline{g} ^{(\omega)}\)            \> = A _{21}+ A _{22}.
\end{align*}
By \eqref{eq:estpar}   proceeding like for $A_3$,
\begin{align*}
|A _{22}|&\lesssim
 \| {\eta}\chi_A\|_{L^1}|z | ^3   \lesssim  \sqrt{ \delta }  A ^{\frac{3}{2}}    \| \eta \| _{\Sigma _A}  |z| ^2 \lesssim \sqrt{ \delta} A ^{\frac{3}{2}}
 \( \| \eta \| _{\Sigma _A}^{2}+  |z|^{4} \) .
\end{align*}
By \eqref{eq:nls4} and  by \eqref{eq:eqsatg2}
for the cancellation,  we have
  \small\begin{align*}&
A_1+A _{21}=  -\xcancel{
\<  J
	 {\eta},\chi_A   \mathcal{L}_{\omega}  (  {z}^{2} {g} ^{(\omega)} + \overline{{z}}^{2} \overline{g}^{(\omega)} )\> }  + \xcancel{A _{21}}  +  ( \widetilde{\vartheta}  _{\mathcal{R}} + \widetilde{\vartheta} - \dot \vartheta)
\< \eta ,  \chi_A   (  {z}^{2} {g} ^{(\omega)} + \overline{{z}}^{2} \overline{g}^{(\omega)} )\>
\\&
-\< J   e^{-\im  \vartheta} D_\Theta \phi [\Theta]  (\dot \Theta  -\widetilde{\Theta})  , \chi_A   (  {z}^{2} {g}^{(\omega)}  + \overline{{z}}^{2} \overline{g} ^{(\omega)}) \>
\\
	& + \< (2\chi _A '\partial _x + \chi _A '')
	 {\eta} , {z}^{2} {g} ^{(\omega)} + \overline{{z}}^{2} \overline{g}^{(\omega)} \>  -\<  \( D f( \phi  [\omega  , z ]  ) - D f( \phi   _{ \omega }   )\) \eta  ,\chi_A   (  {z}^{2} {g} ^{(\omega)} + \overline{{z}}^{2} \overline{g} ^{(\omega)}) \> \\& +\<  f( \phi [\omega  , z ] + \eta  )  -  f( \phi  [\omega  , z ]  ) -D f( \phi  [\omega  , z ]  ) \eta  ,\chi_A   (  {z}^{2} {g} ^{(\omega)} + \overline{{z}}^{2} \overline{g}^{(\omega)} ) \>
+\< R ,\chi_A(  {z}^{2} {g}^{(\omega)}  + \overline{{z}}^{2} \overline{g} ^{(\omega)})) \>
\\&
=A_{11}+A_{12}+A_{13}+A_{14} + A_{15}+A_{16}.
\end{align*}
\normalsize
It is easy to see, and a rather routine computation repeated often in the literature, also using Lemma   \ref{lem:lemdscrt}, that
\begin{align*}
   \sum _{j=1}^{5}|A _{1j}| \lesssim \sqrt{\delta} A ^{\frac{3}{2}}
  \(  |z|^4   + \| \eta \| _{\Sigma _A}^{2}+ \| \eta \| ^{2}_{\widetilde{\Sigma}}  \)  .
\end{align*}
The key term for the FGR is $ A_{16}$. We  have \small
\begin{align*}
  A_{16} &=  \< f(\phi [\omega, {z}]) -  f(\phi _\omega)          - Df(\phi _\omega )\(  z \xi +\overline{z} \overline{\xi} \)  ,\chi_A(  {z}^{2} {g}^{(\omega)}  + \overline{{z}}^{2} \overline{g} ^{(\omega)})) \> \\& -
 \<    \im  D_{z} {\phi}  [\omega, {z}]\widetilde{z } _\mathcal{R} - \widetilde{\vartheta } _\mathcal{R} {\phi}  [\omega, {z}] + \im \widetilde{\omega}_\mathcal{R}\partial_{\omega}\phi [\omega, {z}]   ,\chi_A(  {z}^{2} {g} ^{(\omega)} + \overline{{z}}^{2} \overline{g}^{(\omega)} ))\>  =: A_{161} + A_{162}.
\end{align*}\normalsize
We claim
\begin{align}\label{eq:A152}
  | A_{162}| \lesssim  A ^{-1}   |z|^4.
\end{align}
Indeed, for example we have
\begin{align*}
  &  \widetilde{\omega}_\mathcal{R}    \< \im  \partial_{\omega}\phi [\omega, {z}],  \chi_A(  {z}^{2} {g} ^{(\omega)} + \overline{{z}}^{2} \overline{g} ^{(\omega)} )\> =  \widetilde{\omega}_\mathcal{R}    \< \im  \( \cancel{\partial_{\omega}\phi  _{\omega }}    +    z\partial _\omega \xi   +  +    \overline{z}\partial _\omega \xi  \) ,     {z}^{2} {g} ^{(\omega)} + \overline{{z}}^{2} \overline{g}^{(\omega)}  \> \\& -  \widetilde{\omega}_\mathcal{R}    \< \im  \partial_{\omega}\phi [\omega, {z}], \(1 - \chi_A   \) (  {z}^{2} {g} ^{(\omega)} + \overline{{z}}^{2} \overline{g}^{(\omega)}  )\>    = O\( z^5 \) + O\(  e ^{-\kappa \omega _0 A }z^4 \)
\end{align*}
where we used the orthogonality \eqref{eq:dirsum}--\eqref{eq:dirsum1}, the bound \eqref{eq:estpar} and the exponential decay of $\xi$. The other terms forming $A_{162}$ can be bounded similarly, yielding
\eqref{eq:A152}.
  We have  \small
\begin{align*}
  &A_{161}  =  2^{-1} \<   D^2f(\phi _\omega )\(  z \xi +\overline{z} \overline{\xi} \) ^{2}  ,\chi_A(  {z}^{2} {g}^{(\omega)}  + \overline{{z}}^{2} \overline{g}^{(\omega)} )) \> \\& +\int _{[0,1]^2} t
  \< \(  D^2f(\phi _\omega  +t \ s \(  z \xi +\overline{z} \overline{\xi} \) )-  D^2f(\phi _\omega )\) \(  z \xi +\overline{z} \overline{\xi} \) ^{2}  ,\chi_A(  {z}^{2} {g}^{(\omega)}  + \overline{{z}}^{2} \overline{g} ^{(\omega)})) \> dt \ ds
  \\&=:  A_{1611} + A_{1612}.
\end{align*}\normalsize
We have, taking $\delta>0$  small enough,
\begin{align*}
  | A_{1612}|\le o_{\delta}(1) |z|^4 \le A ^{-1}   |z|^4.
\end{align*}
Next,  by    \eqref{eq:derf2}   for $\xi = (\xi _1, \xi _2) ^ \intercal$,  $X=(z\xi _1 + \overline{z}\overline{\xi} _1)+\im (z\xi _2 + \overline{z}\overline{\xi} _2)$, $u= \phi _\omega$ and identifying $\C=\R ^2$, we have
\begin{align} \label{eq:fgrhess}
  D^2f(\phi _\omega )\(  z \xi +\overline{z} \overline{\xi} \) ^{2} = (p-1)\phi _\omega ^{p-2}  \begin{pmatrix} p (z\xi _1 + \overline{z}\overline{\xi} _1)^2+(z\xi _2 + \overline{z}\overline{\xi} _2) ^2
	    \\ 2(z\xi _1 + \overline{z}\overline{\xi} _1) (z\xi _2 + \overline{z}\overline{\xi} _2)
\end{pmatrix} .
\end{align}
Now we write
\begin{align*}
  A_{1611} & =  2^{-1} \<   D^2f(\phi _\omega )\(  z \xi +\overline{z} \overline{\xi} \) ^{2}  , (  {z}^{2} {g}^{(\omega)}  + \overline{{z}}^{2} \overline{g}^{(\omega)} )) \>  \\&- 2^{-1} \<   D^2f(\phi _\omega )\(  z \xi +\overline{z} \overline{\xi} \) ^{2}  ,(1-\chi_A)(  {z}^{2} {g}^{(\omega)}  + \overline{{z}}^{2} \overline{g}^{(\omega)} )) \> =:A_{16111}+A_{16112}.
\end{align*}
Then it is elementary that  by the exponential decay  as $|x|\to \infty $ of $\xi$ and $\phi _\omega $ we have
\begin{align*}
  |A_{16112}| \lesssim  A^{-1} |z|^4.
\end{align*}
We have
\begin{align*}
   A_{16111}& =   (p-1) |z|^4  \gamma (p, \omega)    \\& +  2(p-1) |z|^2 \( \<    \phi _\omega ^{p-2}  \( p |\xi _1|^2+ |\xi _2|^2 \) , z^2 g_1 ^{(\omega)}\> + 2\<    \phi _\omega ^{p-2}    (\xi _1  \overline{\xi} _2  + \overline{\xi} _1   {\xi} _2)   ,z^2  g_2^{(\omega)} \> \) \\& +   (p-1) \( \<    \phi _\omega ^{p-2}  \( p  \overline{\xi} _1 ^2+  \overline{\xi} _2 ^2 \) , z^4 g_1^{(\omega)} \> +2 \<    \phi _\omega ^{p-2}     \overline{\xi} _1  \overline{\xi} _2     ,z^4  g_2^{(\omega)} \> \)   \\&  =:  A_{161111}+ A_{161112} +A_{161113} .
\end{align*}
We claim that  there exists a function $\mathcal{I}_{\mathrm{FGR}}$ like in the statement such that
\begin{align}\label{eq:classnormform}
 \left |    - \frac{d}{dt} \mathcal{I}_{\mathrm{FGR}}  +  \sum _{j=2}^{3}  A_{16111j}     \right |\lesssim |z|^5+ |z | ^3   |\dot \Theta - \widetilde{\Theta}|.
\end{align}
For   $n=2,4$   write
\begin{align}\label{eq:nform}
   \frac{d}{dt}z  ^n&= n   \widetilde{ {z}}  z   ^{n-1} +n  \( \dot z -  \widetilde{ {z}}   \) z   ^{n-1}   =   n\im \lambda        z ^n +    n   \widetilde{ {z}}_{\mathcal{R}}  z   ^{n-1}     +n  \( \dot z -  \widetilde{ {z}}   \) z   ^{n-1}.
\end{align}
Then for example
\begin{align*}&
  \<    \phi _\omega ^{p-2}     \overline{\xi} _1  \overline{\xi} _2     ,z^4  g_2^{(\omega)} \> = \<    \phi _\omega ^{p-2}     \overline{\xi} _1  \overline{\xi} _2     ,   \frac{1}{4\im \lambda }
    \(  \frac{d}{dt} z^4  - 4  \widetilde{ {z}}_{\mathcal{R}}  z   ^{3} -4  \( \dot z -  \widetilde{ {z}}   \) z   ^{3} \)  g_2^{(\omega)} \>
\end{align*}
and applying the Leibnitz rule for the time derivative, it is easy to obtain the claim   \eqref{eq:classnormform}.
Using also the normalization in \eqref{eq:normalizfgrc} we obtain   \eqref{eq:lem:FGR11}.

\qed

\textit{Proof of Proposition \ref{prop:FGR}.
}
Integrating \eqref{eq:lem:FGR11}  we obtain $\| z^2 \| _{L^2(I)}   ^2 \lesssim \sqrt{\delta} + A ^{-1}\epsilon ^2$ yielding \eqref{eq:FGRint}.

\qed

\section{High energies: proof of Proposition \ref{prop:1virial}}\label{sec:vir}

The power of the method used by Kowalczyk et al. \cite{KMM3}  is seen at high energies, thanks
to a striking computation that deals with great ease with the $|\eta | ^{p-1}\eta$ term in the equation \eqref{eq:nls3}, see in particular formula (3.12) in  Kowalczyk et al. \cite{KMM3}.  Notice that methods of proof of dispersion based on the Duhamel formula, run into great trouble when dealing with the $|\eta | ^{p-1}\eta$ at low $p$'s.

\noindent  Following  the framework in Kowalczyk et al. \cite{KMM3}  and
using the function $\chi$  in \eqref{eq:chi}
we  consider the   function
\begin{align}\label{def:zetaC}
\zeta_A(x):=\exp\(-\frac{|x|}{A}(1-\chi(x))\)   \text {  and   }  \varphi_A (x):=\int_0^x \zeta_A^2(y)\,dy
\end{align}
  and  the vector field
\begin{align}\label{def:vfSA} \  S_A:=   \varphi_A'+2 \varphi   _{ A}\partial_x .
\end{align}
Next, we set  \begin{align*}
\mathcal{I} := 2^{-1}\<  \im \eta  ,S_A   \eta \>  .
\end{align*}
\begin{lemma}\label{lem:1stV1}
There exists a fixed constant $C>0$ s.t.
\begin{align} &
\|   \eta  \|_{\Sigma _A}^2   \le C\left [  \dot{\mathcal{I}} +  \|    \eta \| _{\widetilde{\Sigma}} ^2  +|\dot \Theta   - \widetilde{\Theta}| ^2+     |z |^4 \right ].   \label{eq:lem:1stV1}
\end{align}
\end{lemma}
\proof
From \eqref{eq:nls3}
\begin{align*}
  \dot {\mathcal{I}} = -\<  \dot \eta,    \im S_A\eta \>  &= -\<  \partial _x^2    \eta ,S_A\eta  \> -  \<    R ,S_A\eta  \>  -  \<    f( \phi  + \eta  )  -  f( \phi    )     ,S_A\eta  \>    \\& + O\(  |\dot \Theta   - \widetilde{\Theta}|   \|    \eta \| _{\widetilde{\Sigma}}  \).
\end{align*}
By \eqref{estR}, for a preassigned small constant $\delta _2>0$
\begin{align*}
  | \<    R ,S_A\eta  \>|\lesssim |z|^4 +\| \eta \| _{\widetilde{\Sigma}} ^2 +\delta _2 \| \eta \| _{ {\Sigma}_A} ^2 .
\end{align*}
From  \cite{Martel1} we have
\begin{align*}
  - \<  \partial _x^2    \eta ,S_A\eta  \> \ge   2 \|   (\zeta _A \eta)' \| _{L^2}^2- \frac{C }{A}  \|    \eta \| _{\widetilde{\Sigma}} ^2.
\end{align*}
Like in  \cite{Martel1}
\begin{align*}
  &  \<    f( \phi  + \eta  )  -  f( \phi    )     ,S_A\eta  \>   = -2  \<    F( \phi  + \eta  ) - F( \phi    )    -  f( \phi    ) \eta    , \zeta _A ^2  \>  \\&   -2 \<    f( \phi  + \eta  ) - f( \phi    )    -  f'( \phi    ) \eta    ,  \phi ' \varphi  _A    \>  +\<    f( \phi  + \eta  )  -  f( \phi    )     ,  \zeta _A ^2  \eta  \> = \sum _{j=1}^{3}B_j .
\end{align*}
 We have
\begin{align*}
 & B_1= -2\int _{[0,1]^3}  t_1 \< D^2f( t_3\phi  +t_1 t_2 \eta  ) (\eta, \phi ), \eta \zeta _A ^2\> dt_ 1dt_ 2dt_ 3 +2 \< F(\eta ) , \zeta _A ^2\>   \\&   B_3= \int _{[0,1]^2}   \< D^2f( t_2\phi  +t_1   \eta  ) (\eta, \phi ), \eta \zeta _A ^2\> dt_ 1dt_ 2
 -  \< f(\eta ) , \eta \zeta _A ^2\>
\end{align*}
This yields
 \begin{align*}
   |B_1|+|B_3|& \lesssim  \|    \eta \| _{\widetilde{\Sigma}} ^2 + \delta ^{p-1}A ^2  \| \eta \| _{\Sigma_{A}}^2,
 \end{align*}
where the crucial bound is
\begin{align}\label{eq:cruche}
  \int _{\R} |\eta|^{p+1} \zeta _A ^2 dx\lesssim A^2 \| \eta \| ^{p-1}_{L^\infty (\R )}   \| (\zeta _A\eta ) ' \| ^{2}_{L^2 (\R )},
\end{align}
see   Kowalczyk et al. \cite{KMM3}, see also \cite{CM19SIMA}.
We have
 \begin{align*}
   |B_2|& \lesssim  \|    \eta \| _{\widetilde{\Sigma}} ^2 .
 \end{align*}
Notice, see Lemma 6.2 \cite{CMS23},  that the following holds, which with the above yields the proof,  \begin{align*} \| \sech  \(   \frac{2}{A}x\) \eta '\|_{L^2}^2 + A^{-2}\| \sech  \(   \frac{2}{A}x\) \eta \|_{L^2}^2 \lesssim   \|   ( \zeta _A \eta )' \| _{L^2(\R )} ^2    +A ^{-1}  \|    \eta \| _{\widetilde{\Sigma}}  ^2.\end{align*}
\qed

 \textit{Proof of Proposition \ref{prop:1virial}.}
Integrating  inequality \eqref{lem:1stV1} we obtain \eqref{eq:sec:1virial1}.

\qed

\section{Low energies: proof of Proposition \ref{prop:smooth11}}\label{sec:smooth1}

While very effective and efficient at proving dispersion at  high energies thanks to inequality \eqref{eq:cruche}, the virial inequality of Kowalczyk et al. \cite{KMM3} is somewhat inefficient at low energies, because it places some   restrictions on the system  that seem   due to the  method of proof.  In fact, as we show below, we can replace
 the virial inequality  with smoothing estimates. This  because we only need to bound  $\|  {\eta} \|_{ \widetilde{\Sigma} }  =\left \| \sech \( \kappa \omega _0 x\)   {\eta}\right \|_{L^2(\R )}$, which has the rapidly decaying weight $\sech \( \kappa \omega _0 x\) $. It is enough to bound $\left \| \sech \( \kappa \omega _0 x\)   \chi _{B} {\eta}\right \|_{L^2(\R )}$ because the difference is, choosing  $B\sim \sqrt[3]{A}\ll \sqrt{A}$, a small fraction of $\| \eta \| _{\Sigma_{A}}$. To get a bound for $\left \| \sech \( \kappa \omega _0 x\)   \chi _{B} {\eta}\right \|_{L^2(\R )}$, we   can multiply equation \eqref{eq:nls4}  by $\chi _B$. The cutoff $\chi _B$ tames the term $\chi _B |\eta | ^{p-1}\eta $, which  is   very small and easy to bound.  There is a new term due to the commutator of $\chi _{B}$  and $\mathcal{L}_{\omega}$    which requires a new smoothing estimate, see Proposition \ref{lem:smoothest1}, not  already contained in Krieger and Schlag \cite{KrSch}.  Notice that
in the case of  scalar Schr\"odinger operators in $\R$  a version of  Proposition \ref{lem:smoothest1} is implied by Deift and Trubowitz \cite{DT}  and is easy to prove, see Sect. 8 \cite{CM2109.08108}.
 For convenience, in the study of dispersive and smoothing estimates  of \eqref{eq:nls4}  it is customary to  use a different coordinate system.
We consider  the matrix $U$ defined
by
\begin{align}\label{def:U}
	U = \begin{pmatrix}
		1 &
		1  \\
		\im &
		-\im      \end{pmatrix} \, , \quad
	U^{-1}= \frac 12  \begin{pmatrix}  1 &
		-\im   \\
		1 &
		\im    \end{pmatrix}   .\end{align}
	We have  $$U^{-1} J
	U= \im \sigma _3     \text{ where }  \sigma _3:=\diag (1, -1)
	.$$
   By  elementary computations
\begin{align}\label{eq:opH}
   & U ^{-1}\mathcal{L}_{\omega} U = \im  H _{\omega}  \text{  where } H _{\omega} =   \sigma _3 \(-\partial _x^2 +\omega \) +V_\omega , \\&
    V_\omega := -
    \frac{p+1}{2}\phi _\omega  ^{p-1}
    \sigma _3      -\im  \frac{p-1}{2}\phi _\omega  ^{p-1} \sigma _2  . \nonumber
\end{align}
Notice that we have the symmetries
\begin{align}\label{eq:symmlin1}
   \sigma _1 H _{\omega} = - H _{\omega}  \sigma _1  \text{  and } \sigma _3 H _{\omega} =   H _{\omega} ^* \sigma _3.
\end{align}
Applying $U^{-1}$ to equation \eqref{eq:nls4} we get
\small
\begin{align}\label{eq:nls5}
   \partial _t(  U^{-1} \eta )    &=   \im  H _{\omega} U^{-1} \eta -  \im \sigma _3   ( \widetilde{\vartheta}  _{\mathcal{R}} + \widetilde{\vartheta} - \dot \vartheta) U^{-1}  \eta \\& \nonumber  -   e^{\im \sigma _3 \vartheta} U^{-1}D_\Theta \phi [\Theta]  (\dot \Theta  -\widetilde{\Theta})   +\im \sigma _3   U^{-1} \( D f( \phi  [\omega  , z ]  ) - D f( \phi   _{ \omega }   )\) \eta
    \\&-\im \sigma _3   U^{-1} \( f( \phi [\omega  , z ] + \eta  )  -  f( \phi  [\omega  , z ]  ) -D f( \phi  [\omega  , z ]  ) \eta\) -\im \sigma _3   U^{-1}R[\omega  , z ].  \nonumber
\end{align}\normalsize
Set     $v:= \chi _B   U^{-1}\eta $.
 We denote  $P_d(\omega)= P_d(H _{\omega}) $ the discrete spectrum projection and   $P_c(\omega)=P_c(H _{\omega}) $ the continuous spectrum projection associated to $H _{\omega}$, which are closely related to the corresponding projections for $\mathcal{L}_\omega$, indeed  $ U P_a(H _{\omega})  U ^{-1} =  P_a(\mathcal{L}_{\omega})   $ for $a=c,d$. Then we write
 \begin{align}\nonumber & v=P_c(\omega _0) v + P_d(\omega _0) v   \text{ where we have}   \\&  P_d(\omega _0) v=  P_d(\mathcal{L}_{\omega}) \eta + U ^{-1} \( P_d(\mathcal{L}_{\omega _0}) - P_d(\mathcal{L}_{\omega}) \) \eta - P_d(\omega _0) \( 1 - \chi _B\)   U ^{-1} \eta  . \label{eq:calA}
    \end{align} 
     It is easy to check that
 \begin{align}
   \nonumber  \|  P_d(\omega _0) \( 1 - \chi _B\)   U ^{-1} \eta  \| _{L^{2,s}(\R ) } \le o_{B^{-1}}(1)   \| \eta \| _{\widetilde{\Sigma}}    \text{   for any $s\in \R$} .
 \end{align}
Similarly, we have
 \begin{align}  \nonumber  \|     U ^{-1} \( P_d(\mathcal{L}_{\omega _0}) - P_d(\mathcal{L}_{\omega}) \) \eta  \| _{L^{2,s}(\R ) } \le  \epsilon    \| \eta \| _{\widetilde{\Sigma}}    \text{   for any $s\in \R$} .
 \end{align}
 Finally, using the notation in \eqref{eq:refprof} and item 3 in Notation \ref{not:notation},
  \begin{align}&
     \|     U ^{-1} P_d(\mathcal{L}_{\omega}) \eta   \| _{L^{2,s}(\R ) } \lesssim |\< \eta , \im \phi _{\omega}\> |  + |\< \eta , \partial _\omega \phi _{\omega}\> | +\sum _{j=1}^2 \left |\< \eta ,  \partial _{z_j}\widetilde{\phi } [{\omega}, z]\> \right |    \text{   for any $s\in \R$} . \label{eq:error-11}
 \end{align}
 But then, by the definition of the refined profile and by  the orthogonality in \eqref{61} for the fact that the cancelled term is null,  we have
  \begin{align*}  |\< \eta , \im \phi _{\omega}\> | \le \cancel{ |\< \eta , \im {\phi } [{\omega}, z] \> |} +   \| \cosh \(  \kappa \omega _0 x \)    \( \phi _{\omega} - {\phi } [{\omega}, z] \)  \|  _{L^2}  \| \eta \| _{\widetilde{\Sigma}}
  \lesssim        |z| \| \eta \| _{\widetilde{\Sigma}}
  \end{align*}
  with similar estimates for the other terms in the right hand side of  \eqref{eq:error-11}. Hence we conclude 
 \begin{align*}&
     \|     U ^{-1} P_d(\mathcal{L}_{\omega}) \eta   \| _{L^{2,s}(\R ) } \lesssim     |z|   \| \eta \| _{\widetilde{\Sigma}}    \text{   for any $s\in \R$} .
 \end{align*}
 So we conclude that for the term in \eqref{eq:calA} we have
\begin{align}
  \label{eq:estcalA}   \| P_d(\omega _0) v   \| _{L^{2,s}(\R ) } = o_{B^{-1}}(1)   \| \eta \| _{\widetilde{\Sigma}}    \text{   for any $s\in \R$} .
\end{align}
 Setting $w= P_c(\omega _0)v$ and
 \begin{align} \label{eq:nls61}
   \varpi := \widetilde{\vartheta}  _{\mathcal{R}} + \widetilde{\vartheta} - \dot \vartheta +\omega -\omega _0 .
\end{align}
 we have      \small
\begin{align}\label{eq:nls6}
     \partial _t w&=   \im  H _{\omega _0} w   -\im  \varpi  P_c(\omega _0)\sigma _3        w   +\im  P_c(\omega _0) \sigma _3 \( 2 \chi ' _B \partial _x + \chi _B '' \)  U^{-1}\eta   \\& \nonumber  -\im  \varpi  P_c(\omega _0)\sigma _3    P_d(\omega _0) v  + \im P_c(\omega _0)  (V _{\omega _{0}}    - V _{\omega  })v  \\& \nonumber  -  P_c(\omega _0) \chi _B e^{\im \sigma _3 \vartheta} U^{-1}D_\Theta \phi [\Theta]  (\dot \Theta  -\widetilde{\Theta})   +\im P_c(\omega _0) \sigma _3 \chi _B   U^{-1} \( D f( \phi  [\omega  , z ]  ) - D f( \phi   _{ \omega }   )\) \eta
    \\&-\im P_c(\omega _0)  \chi _B\sigma _3   U^{-1} \( f( \phi [\omega  , z ] + \eta  )  -  f( \phi  [\omega  , z ]  ) -D f( \phi  [\omega  , z ]  ) \eta\) -\im P_c(\omega _0)\sigma _3  \chi _B  U^{-1}R[\omega  , z ]  . \nonumber
\end{align}\normalsize
There is a splitting   $P_c(\omega
)= P_+(\omega )+P_-(\omega )$    with  $P_\pm (\omega )$ the spectral
projections in $\R_\pm \cap \sigma _e(H_\omega )$.  Specifically  we have the following  for which we refer to \cite{CV}.
\begin{lemma}\label{lem:spcproj1} The following are bounded operators  $P_\pm (\omega )$   in $L^2\( \R \) $
\begin{equation}\label{eq:proj+-}
   \begin{aligned}
  &  P_+(\omega )u =\lim _{M \to  + \infty}\lim _{
\epsilon \to  0 ^+ }
 \frac 1{2\pi \im }
 \int  _{[\omega ,M] } \left [ R_{H_\omega}(E +\im \epsilon )-
R_{H_\omega}(E -\im \epsilon )
 \right ] udE \\&
P_-(\omega )u =\lim _{M \to  + \infty}\lim _{\epsilon \to  0 ^+
}
 \frac 1{2\pi \im }
 \int _{[-M,-\omega ] } \left [ R_{H_\omega}(E +\im \epsilon )-
R_{H_\omega}(E -i\epsilon ) \right ] udE
\end{aligned}
\end{equation}
and for any $M>0$ and
$N>0$ and for $C=C (N,M,\omega  )$ 
we have
\begin{equation}\|   (P_+(\omega
 )-P_-(\omega  )-P_c(H_\omega   )\sigma _3) f\|  _{L ^{2,M}\( \R \) }\le
C  \|      f\|  _{L ^{2,-N}\( \R \) }. \label{eq:spcproj11}
\end{equation}

\end{lemma} \qed

 A version of Lemma  \ref{lem:spcproj1}  was introduced by Buslaev and Perelman \cite{BP2}, see also \cite{BS}.


 We rewrite  \eqref{eq:nls6} as
  \small
\begin{align}
     \partial _t w&=   \im  H _{\omega _0} w   -\im  \varpi  (P_+(\omega _0
 )-P_-(\omega _0 ))  w   +\im \sigma _3 \( 2 \chi ' _B \partial _x + \chi _B '' \)U^{-1}\eta \nonumber \\& \label{eq:nls71}  -\im     \varpi  P_c(\omega _0)\sigma _3      P_d(\omega _0) v          -\im  \varpi  \left [  P_c(\omega _0)\sigma _3 -P_+(\omega _0
 )+P_-(\omega _0 )\right ]  w      \\& \label{eq:nls72} -\im  \varpi  P_c(\omega _0)\sigma _3 P_d(\omega _0) v + \im P_c(\omega _0) (V _{\omega _{0}}    - V _{\omega  })v   \\& \label{eq:nls73} -  P_c(\omega _0) \chi _B e^{\im \sigma _3 \vartheta} U^{-1}D_\Theta \phi [\Theta]  (\dot \Theta  -\widetilde{\Theta})   +\im P_c(\omega _0) \sigma _3 \chi _B   U^{-1} \( D f( \phi  [\omega  , z ]  ) - D f( \phi   _{ \omega }   )\) \eta
    \\&-\im P_c(\omega _0)  \chi _B\sigma _3   U^{-1} \( f( \phi [\omega  , z ] + \eta  )  -  f( \phi  [\omega  , z ]  ) -D f( \phi  [\omega  , z ]  ) \eta\) -\im P_c(\omega _0)\sigma _3  \chi _B  U^{-1}R[\omega  , z ]  .  \label{eq:nls74}
\end{align}\normalsize
We have
\begin{align}\label{eq:expv1}
   w &= \mathcal{U}(t,0) w(0)   +\im       \int _0 ^t   \mathcal{U}(t,t')    P_c(\omega _0) \sigma _3 \( 2 \chi ' _B \partial _x + \chi _B '' \)  U^{-1}\eta  dt' \\& + \int _0 ^t   \mathcal{U}(t,t')  \(  \text{lines \eqref{eq:nls71}-- \eqref{eq:nls74}}  \)   dt',\label{eq:expv2}
\end{align}
where $\mathcal{U}$ is the generator associated with the linear evolution $ \partial _t w=   \im  H _{\omega _0} w   -\im  \varpi  (P_+(\omega _0
)-P_-(\omega _0 ))  w$.
For $\alpha _{tt'}= \int _{t'}^{t}\varpi  ( \upsilon )  d \upsilon$,  $P_\pm = P_\pm (\omega _0 )$  and $P_c= P_c(\omega _0 )$,  expanding the exponential we get
\begin{align}\label{eq:expv21}
  \mathcal{U}(t,t'):= e ^{\im (t-t') H _{\omega _0}} P_c  e^{\im \alpha _{tt'}(P_+ -  P_- )} =P_c e ^{\im (t-t') H _{\omega _0}} \(   \cos \( \alpha _{tt'}   \) P_c +\im   \sin \( \alpha _{tt'}  \)   (P_+ -  P_- ) \)  .
\end{align}

\begin{lemma}\label{lem:estw}
 For $S > 3/2$  we have
 \begin{equation}\label{eq:estw1}
   \| w \|  _{L^2(I , L ^{2,-S}(\R ) )}\le   o_{B^{-1}}(1)   \epsilon.
 \end{equation}

\end{lemma}
\proof   Let us take $S>3/2$.      Taking the expansion in \eqref{eq:expv21}, we have
\begin{align*}
   \| \mathcal{U}(t,0) w(0) \| _{L^2(I , L ^{2,-S}(\R ) )} & \le   \|    e ^{\im t H _{\omega _0}}   w(0) \| _{L^2(\R, L ^{2,-S}(\R ))}
   \\&  + \|    e ^{\im t H _{\omega _0}}   (P_+-P_-) w(0) \| _{L^2(\R, L ^{2,-S}(\R ))} .\end{align*}
By  the analogue for $H_\omega $ of \eqref{eq:smooth111}, see \eqref{eq:smooth111b}, which is what we actually prove in Sect. \ref{sec:smoothing} below,    we have
 \begin{align}
   \|    e ^{\im t H _{\omega _0}}   w(0) \| _{L^2(\R, L ^{2,-S}(\R ))} \lesssim   \|
   w(0)  \| _{L^2 (\R )}.\nonumber
 \end{align}
 Similarly
  \begin{align*}
     \|    e ^{\im t H _{\omega _0}}   (P_+-P_-) w(0) \| _{L^2(\R, L ^{2,-S}(\R ))} &      \lesssim   \|
   (P_+-P_-) w(0)  \| _{L^2 (\R )} \lesssim   \|
     w(0)  \| _{L^2 (\R )}.
  \end{align*}
  The second term  in  \eqref{eq:expv1} is more delicate for more than one reason. First of all,  by \eqref{eq:expv21} and $\alpha _{tt'} =\alpha _{t0}-\alpha _{ t'0}$,
   we write  the integrand as  the sum
\begin{align*}&   e ^{\im (t-t') H _{\omega _0}}P_c  \left [    \cos \( \alpha _{t0} \)  \cos \( \alpha _{ t'0}\)  +  \sin \( \alpha _{t0} \)  \sin \( \alpha _{ t'0}\)   \right . \\& \left . +  \im  \( \cos \( \alpha _{t0} \)  \sin \( \alpha _{ t'0}\)  - \sin \( \alpha _{t0} \) \cos \( \alpha _{ t'0}\) \) (P_+-P_-) \right ] \sigma _3 \( 2 \chi ' _B \partial _x + \chi _B '' \)  U^{-1}\eta
\end{align*}
 This yields various terms, that can be bounded all in the same way, so that we bound only the last of them.
  We proceed like in \cite{CM2109.08108}.  We have
\begin{align*}&
   \|  \sin \( \alpha _{t0} \)    \int _0 ^t   e ^{\im (t-t') H _{\omega _0}} P_c (P_+-P_-)  \sigma _3  \( 2 \chi ' _B \partial _x + \chi _B '' \)  \cos \( \alpha _{ t'0}\) U^{-1}\eta  dt'  \| _{L^2(I , L ^{2,-S}(\R ) )} \\&
   \lesssim   \|      \int _0 ^t    e ^{\im (t-t') H _{\omega _0}}   (P_+-P_- - P_c \sigma _3)  \sigma _3  \( 2 \chi ' _B \partial _x + \chi _B '' \)  \cos \( \alpha _{ t'0}\) U^{-1}\eta  dt'  \| _{L^2(I , L ^{2,-S}(\R ) )}
   \\& +   \|       \int _0 ^t    e ^{\im (t-t') H _{\omega _0}} P_c    \( 2 \chi ' _B \partial _x + \chi _B '' \)  \cos \( \alpha _{ t'0}\) U^{-1}\eta  dt'  \| _{L^2(I , L ^{2,-S}(\R ) )}
   =:I_1+I_2 .
\end{align*}
For $I_1$ can use the estimate  \eqref{eq:KrSch} derived by Krieger and Schlag \cite{KrSch} and write \small
\begin{align}    \nonumber
   I_1 &\lesssim  \|  \int _0 ^t   \|  e ^{\im (t-t') H _{\omega _0}} P_c \| _{L^{2,S}\to L^{2,-S}}   \|  P_+-P_- - P_c \sigma _3\| _{L^{2 }\to L^{2, S}}  \| \( 2 \chi ' _B \partial _x + \chi _B '' \)   U^{-1}\eta  \| _{L^2(\R )} dt'    \| _{L^2(I   )}
   \\& \lesssim   \|  \int _0 ^t   \< t-t'\> ^{-3/2}  \| \( 2 \chi ' _B \partial _x + \chi _B '' \)   U^{-1}\eta  \| _{L^2(\R )} dt'    \| _{L^2(I   )}
   \lesssim      \|     \( 2 \chi ' _B \partial _x + \chi _B '' \)   \eta    \| _{L^2(I , L ^{2 }(\R ) )}  , \label{eq:bboundI1}
\end{align}\normalsize
 where we postpone completion of the analysis.  The term  $I_2$ is more delicate and is bounded by    Lemma \ref{lem:smoothest}, expressed for $H_\omega$ instead of $\mathcal{L}_\omega$, which is the same. So  for any $s>1/2$
\begin{align}  I_2& \lesssim
           \|     \( 2 \chi ' _B \partial _x + \chi _B '' \)   \eta    \| _{L^2(I , L ^{2,s}(\R ) )} . \label{eq:bboundI2}
\end{align}
 Now we have
\begin{align*}
  &\|     \( 2 \chi ' _B \partial _x + \chi _B '' \)  \eta     \| _{L^2(I , L ^{2,s}(\R ) )}  \lesssim B ^{s-1} \| \sech \( \frac{2}{A}x\) \eta ' \| _{L^2(I ,L^2(\R ))} \\& + B ^{s-2} \| 1_{B\le |x|\le 2 B} \sech \( \frac{2}{A} x\)  \eta   \| _{L^2(I ,L^2(\R ))} \lesssim B ^{s-1} \|   \eta   \| _{L^2(I ,\Sigma _A )} \\& +  B ^{s-1} \( \left \|   \( \sech \( \frac{2}{A}x\) \eta  \) '    \right \| _{L^2(I ,L^2(\R ))}   +  \|    \eta \| _{L^2(I , \widetilde{\Sigma})}
  \) =o_{B^{-1}} (1)   \epsilon,
\end{align*}
 where we used   $s\in (1/2,1) $   and, see   \cite{CM2109.08108},
\begin{align*}
  \|  1_{B\le |x|\le 2 B} u \| _{L^2(\R )}\lesssim \sqrt{\|  1_{B\le |x|\le 2 B} |x| \| _{L^1(\R )}} \( \left \|  u '    \right \| _{ L^2(\R  )}   +  \|    u \| _{  \widetilde{\Sigma} }\)  .
\end{align*}
This implies the following, yielding good  bounds for the terms in the right hand side of line \eqref{eq:expv1},
\begin{align*}
   I_1 + I_2 = o_{B^{-1}}(1)   \epsilon.
\end{align*} The terms in line \eqref{eq:expv2} can be similarly bounded using in particular the  analogue for $H_\omega $  of {Proposition} \ref{prop:KrSch}. The estimates are elementary and similar to \cite[Sect. 8]{CM2109.08108}.

\qed

\textit{Proof of Proposition \ref{prop:smooth11}.}
From \eqref{eq:calA},  \eqref{eq:estcalA} and \eqref{eq:estw1} we have $\| v\| _{L^2(I , \widetilde{\Sigma})}\lesssim o_{B^{-1}}(1) \epsilon$.  Next, from  $v = \chi _B   U^{-1}\eta $,  and thanks to the relation  $A\sim B^3$  set in \eqref{eq:relABg}    we have
\begin{align} \nonumber
   \| \eta \| _{\widetilde{\Sigma}}&\lesssim  \| v \| _{\widetilde{\Sigma}} + \| (1- \chi _B   )\eta \| _{\widetilde{\Sigma}}\lesssim \| v \| _{\widetilde{\Sigma}} + A ^{-2} \|  \sech \( \frac{2}{A}x  \) \eta \| _{L^2}\\& \lesssim \| v \| _{\widetilde{\Sigma}} + A ^{-1} \|   \eta \| _{\Sigma _A}\label{eq:relAB0}
\end{align}
So by \eqref{eq:sec:1virial1}  and \eqref{eq:FGRint}  we get  the following, which implies \eqref{eq:sec:smooth11},
\begin{align} \nonumber
   \| \eta \| _{L^2(I,\widetilde{\Sigma})}&\lesssim    \| v \| _{L^2(I,\widetilde{\Sigma})} + A ^{-1} \|   \eta \| _{L^2(I, \Sigma _A)}\\& \lesssim  o_{B^{-1}} (1) \epsilon +  A ^{-1} \( \| z^2\|_{L^2(I)} + \| \eta \| _{L^2(I, \widetilde{\Sigma}   )}\) \lesssim   o_{B^{-1}}(1) \epsilon +  A ^{-1}\| \eta \| _{L^2(I, \widetilde{\Sigma}   )}    .  \nonumber
\end{align}

 \qed

\section{The resolvent of the linearized operator}\label{sec:lin}

We will focus on the operator $H _{\omega}$ in  \eqref{eq:opH}. For the discussion it is enough to consider $\omega =1$, since the operators  for other values of $\omega $ are obtained by a scaling transformation.
We will set $H= H_1$ with  vector potential $V=V_1$.  We will set
\begin{align*}
  e_1:= \begin{pmatrix}
 1   \\ 0
\end{pmatrix} \text{   and } e_2:= \begin{pmatrix}
 0   \\ 1
\end{pmatrix} .
\end{align*}
Given two (column) functions $f,g:\R \to \C ^2$,  using the row column product, we consider the Wronskian
\begin{align*}
  W[f ,g](x) := f'(x) ^ \intercal g(x)- f (x) ^ \intercal g '(x).
\end{align*}

It is well known that $H  $ has Jost functions, discussed in \cite{BP1,KrSch}, which we subsume here.

\begin{proposition}
  \label{prop:jost1}For any $k\in \R$  there exists solutions $f_j (x,k)$ for $j=1,2,3,4$ of
  \begin{align}\label{eq:jost11}
     H u=(1+k^2) u
  \end{align}
 with for a fixed $C>0$ and for $x\ge 0$
  \begin{align}\label{eq:jost12} &
    f_j (x,k) =e^{\im (-1)^{j+1}xk}m_j (x,k) \text{   with }  \left | m_j (x,k) - e_1   \right | \le C  \< k \> ^{-1}  e ^{-(p-1) x}    \text{ for $j=1,2$},  \\& \label{eq:jost13}
    f_3 (x,k) =e^{ -\sqrt{2+k^2} x}m_3 (x,k) \text{   with }  \left | m_3 (x,k) - e_2   \right | \le C   \< k \> ^{-1}  e ^{-(p-1) x}   .
  \end{align}
  There is a solution $\widetilde{f}_4 (x,k)$ of \eqref{eq:jost11} with
  \begin{align}\label{eq:jost14t} &
    \widetilde{f}_4 (x,k) =e^{  \sqrt{2+k^2} x}\widetilde{m}_4 (x,k) \text{   with }  \left | \widetilde{m}_4 (x,k) - e_2   \right | \le C   \< k \> ^{-1}  e ^{-(p-1) x}   .
  \end{align}
  We have
 \begin{align}
   \label{eq:wronsk1} W[f_1,f_2]= 2\im k \text{ ,  } W[f_3,\widetilde{f}_4]= 2\sqrt{2+k^2} \text{ ,  } W[f_j,f_3]=0 \text{ for $j=1,2$.}
 \end{align}
  There is a unique choice of $c_1,c_2\in \C$ such that for
  \begin{align}\label{eq:jost14} {f}_4 (x,k) :=-c_1  f_1 (x,k)  -c_2 f_2 (x,k)+ \widetilde{f}_4 \Longrightarrow     W[f_j,f_4]=0 \text{ for $j=1,2$.}
     \end{align}

  \end{proposition} \qed

For the proof see \cite{KrSch}. Since the potential $V(x)$ is even, writing
  \begin{align}\label{eq:jost21} &
    g_j (x,k)  :=
    f_j (-x,k)
  \end{align}
yields analogous Jost functions with prescribed behavior as $x\to -\infty$.   Notice that since the potential $V(x)$ is exponentially decreasing, all the above Jost functions extend in the region $|\Im k |\le \delta _p$  for a small $\delta _p>0$.

\begin{remark}\label{rem:plwcub} For $p=3$ it is possible to write explicit formulas for the above Jost functions,  for $ f_j (x,k)$ for $j=1,2$ see \cite{kaup}. We write the formulas in  Sect. \ref{sec:intsy}.

\end{remark}

 We consider the matrices
\begin{align*}&  F_1 (x,k) = (f_1 (x,k), f_3 (x,k) ) \quad , \quad  F_2 (x,k) = (f_2 (x,k), f_4 (x,k) )  \, , \\&
  G_1 (x,k) = (g_2 (x,k), g_4 (x,k) ) \quad , \quad  G_2 (x,k) = (g_1 (x,k), g_3 (x,k) )  .
\end{align*}
For   matrix valued functions     $F= (\phi _1, \phi _2) $   and   $G= (\psi _1, \psi _2) $ we set
\begin{align*}
 W[F,G]:= F'(x)^ \intercal G(x) - F  (x)^ \intercal G '(x).
\end{align*}
By direct computation, see   \cite{KrSch},
\begin{align*}
 W[F,G] =    \begin{pmatrix} W[\phi_1,\psi _1]  &   W[\phi_1,\psi _2] \\ W[\phi_2,\psi _1]  &   W[\phi_2,\psi _2] \end{pmatrix}
 .
\end{align*}
Still quoting from \cite{KrSch}, we have the following.
\begin{lemma}\label{lem:trrefl} For any $k\in \R \backslash \{  0 \} $ there matrices $A(k) $ and $B(k)$, smooth in $k$ and  s.t.
   \begin{align}
     \label{eq:trrefl1}  F_1(x,k) =  G_1(x,k)  A(k) + G_2(x,k) B(k),
   \end{align}
   with $ A(-k)= \overline{A}(k)$, $ B(-k)= \overline{B}(k)$ and
   \begin{align} &  \label{eq:trrefl2}   G_2(x,k) =  F_2(x,k)  A(k) + F_1(x,k) B(k) ,    \\&    W[F_1(x,k),G_2(x,k)]  = A(k)^ \intercal   \diag (2\im k, -2\sqrt{2+k^2}) \label{eq:trrefl3} \\&  W[F_1(x,k),G_1(x,k)]  = - B(k)^ \intercal   \diag (2\im k, -2\sqrt{2+k^2}) .\label{eq:trrefl4}
   \end{align}
   Furthermore
   \begin{equation}\label{eq:complcon}
     \begin{aligned}
        & G_1(x,k)= F_2(-x,k)  \quad , \quad G_2(x,k)= F_1(-x,k) \\&  \overline{F_1( x,k)}= F_1( x,-k) \quad , \quad  \overline{F_2( x,k)}= F_2( x,-k).
     \end{aligned}
   \end{equation}

\end{lemma}  \qed

We set
\begin{align}\label{eq:defd}
  D(k):= W[ F_1(x,k) ,  G_2(x,k)] =  \begin{pmatrix} W[ f_1(x,k),g_1(x,k)]  &   W[ f_1(x,k),g_3(x,k)] \\ W[ f_3(x,k),g_1(x,k)]  &   W[f_3(x,k),g_3(x,k)] \end{pmatrix} .
\end{align}
The following holds, see   \cite{KrSch}.

\begin{lemma} \label{lem:eig0}For $k\neq 0$ the following are equivalent: \begin{itemize}
                                                                            \item  $\det A(k)=0$;
                                                                            \item $E=1+k^2$ is an eigenvalue of $H$;
                                                                            \item  $\det D(k)=0$.
                                                                       \end{itemize}      $E=1$ is neither a resonance nor an eigenvalue   of $H$ if and only if $\det D(0)\neq 0$; furthermore we have $D(-k)=\overline{D(k)}$   and $D( k) ^ \intercal = {D(k)}$.

\end{lemma}
The following holds, see   \cite{KrSch}.
\begin{lemma} \label{lem:resolv}For $k\ge 0$ the following extensions of the resolvent $R_H (E)$ from above and from below the real line hold, for $E=1+k^2$:
\begin{align}
  & \label{eq:ris+} R _{H}^+ (x,y,E)  =\left\{\begin{matrix}      -F_1(x,k)  D ^{-1}(k) G_2(y,k) ^ \intercal \sigma _3\text{   if  }  x\ge y   \\
  -G_2(x,k)  D ^{-1}(k) F_1(y,k) ^ \intercal \sigma _3\text{   if  }  x\le y
\end{matrix}\right.  \\&  \label{eq:ris-} R _{H}^- (x,y,E)  =\left\{\begin{matrix}      -F_1(x,-k)  D ^{-1}(-k) G_2(y,-k) ^ \intercal \sigma _3\text{   if  }  x\ge y   \\
  -G_2(x,-k)  D ^{-1}(-k) F_1(y,-k) ^ \intercal \sigma _3\text{   if  }  x\le y .
\end{matrix}\right.
\end{align}

\end{lemma}

We set $x^\pm =\max \{  \pm x, 0   \}$.
The main result of this section is the following.
\begin{proposition}\label{prop:boudres}  There exists a small constant  $\delta _3>0$ such that for any $ p $  with $0<|p-3|< \delta _3$
 there exists a constant  $C$ such that  for any $E\in (-\infty , -1]\cup [1, +\infty ) $ we have
    \begin{align}
       \label{eq:boudres1}   \left |       R _{H}^ \pm  (x,y,E)            \right | \le C  \left\{\begin{matrix}      (1+x^- +y ^+)\text{   if  }  x\ge y   \\
   (1+x^+ +y ^-)\text{   if  }  x\le y .
\end{matrix}\right.
         \end{align}
\end{proposition}
Assuming  Proposition \ref{prop:boudres}, we have the following.
\begin{lemma} \label{lem:LAP} For  $S>3/2$ and $\tau >1/2$  we have
 \begin{align}&   \label{eq:LAP1}   \sup _{E \in   \R  } \|   R ^{\pm }_{H }(E ) P_c \| _{L^{2,\tau}(\R ) \to L^{2,-S}(\R )} <\infty  .
\end{align}
\end{lemma}
\proof     First of all, from the proof of Proposition \ref{prop:boudres} it will be clear that
\eqref{eq:boudres1} holds for  any $E\in (-\infty , -a]\cup [a, +\infty ) $ for an $a\in (0,1)$ sufficiently close to 1.  Then in such a set  we proceed like in \cite {CM2109.08108}, we can ignore $P_c$  and  consider the square of the  Hilbert--Schmidt norm
 \begin{align} \nonumber &   \int _{\R} dx \< x \> ^{-2S} \int_{\R}  |R ^{+ }_{H }(x,y,z )| ^2  \< y \> ^{-2\tau} dy   =    \int _{\R} dx \< x \> ^{-2S} \int_{-\infty}^{x}  |R ^{+ }_{H }(x,y,z )| ^2  \< y \> ^{-2\tau} dy\\& + \int _{\R} dx \< x \> ^{-2S} \int_{x}^{+\infty}  |R ^{+ }_{H }(x,y,z )| ^2  \< y \> ^{-2\tau} dy.\label{eq:estresolv11}
\end{align}
 The second term in the right hand side  is bounded by
  \begin{align*}  &       \int _{x<y} \< x \> ^{-2S}  \< y \> ^{-2\tau}  \( 1+ x^++  y^- \) ^2   dx  dy
   \le  \int _{0<x<y } \< x \> ^{-2S+2}  \< y \> ^{-2\tau  }    dx  dy   \\&+ \int _{x<y<0} \< x \> ^{-2S}  \< y \> ^{-2\tau +2}    dx  dy + \int _{ x<0<y } \< x \> ^{-2S }  \< y \> ^{-2\tau  }    dx  dy =:\sum _{j=1}^{3}I_j.
\end{align*}
Then
\begin{align*}
  I_1\le  \int _{\R } \< x \> ^{-2S+2}  dx \int _\R   \< y \> ^{-2\tau  }      dy =:I_4 <\infty \text{  for $S>3/2$ and $\tau >1/2$.}
\end{align*}
Similarly $I_j< I_4$ for $j=2,3$. Similar estimates hold for the term in the first line in the right hand side of \eqref{eq:estresolv11}.
So now we need to consider the inequality in \eqref{eq:LAP1} only for $E\in [-a,a]$, in which case we can drop the superscript $\pm$. Then the result is trivial, because
\begin{align*} &
  \sup _{E \in  [-a,a] } \|   R  _{H }(E ) P_c \| _{L^{2,\tau}(\R ) \to L^{2,-S}(\R )} \le  \sup _{E \in  [-a,a] } \|   R  _{H }(E ) P_c  \| _{L^{2 }(\R ) \to L^{2 }(\R )}   \\& = \sup _{E \in  [-a,a] } \|   R  _{HP_c }(E )   \| _{R(P_c) \to R(P_c)}  <\infty
\end{align*}
by the invariance of $R(P_c)$,  the Range of $P_c$,  and by   $\sigma (H P_c)\cap (-1,1) = \emptyset $.

\qed

The following formula is inspired by Mizumachi \cite[Lemma 4.5]{mizu08}.
 \begin{lemma} \label{lem:lemma11} Let  for $g\in \mathcal{S}(\R \times \R , \C  ^2)$  with $P_c g(t) =g(t)$    \begin{align*}&  U(t,\cdot ) := \frac{\im } {2\pi}    \int  _{\R} e^{-\im E t}\( R ^{-}_{H }(E )+R ^{+}_{H }(E )   \)  g ^{\vee}( E , \cdot ) dE
\end{align*}
where $ g ^{\vee}$  is the inverse Fourier transform in $t$  of $g$. Then
 \begin{align} \label{eq:lemma11}2\int _0^t e^{-\im (t-t') H  }g(t') dt'   &=  U(t,x)  - \int  _{\R _-} e^{-\im (t-t') H  }g(t') dt'
 \\& + \int  _{\R _+} e^{-\im (t-t') H  }g(t') dt' .\nonumber
\end{align}
\end{lemma}
We postpone the proof of Lemma \ref {lem:lemma11} until the end of this section.
From Lemmas \ref{lem:LAP} and \ref{lem:lemma11} we conclude the following, inspired by Mizumachi \cite{mizu08}.
\begin{lemma} \label{lem:smoothest} For  $S>3/2$ and $\tau >1/2$ there exists a constant $C(S,\tau )$ such that we have
 \begin{align}&   \label{eq:smoothest1}   \left \|   \int   _{0} ^{t   }e^{-\im (t-t') H  }P_c g(t') dt' \right \| _{L^2( \R ,L^{2,-S}(\R ))  } \le C(S,\tau ) \|  g \| _{L^2( \R , L^{2,\tau}(\R ) ) }.
\end{align}
\end{lemma}
\proof The proof is verbatim in    \cite{CM2109.08108}.  We can use formula \eqref{eq:lemma11} and bound $U$, with the bound on the last two terms in the right hand side of \eqref{eq:lemma11} similar. Taking Fourier transform in $t$,
$$ \aligned &
\| U\|_{
L_{t}^2L^{2,-S}} \le   2 \sup _{\pm}\| R_{ H_{N+1} }^\pm (\lambda )
   \widehat{ g}(\lambda,\cdot )\|_{L_{\lambda  }^2 L^{2,-S}}   \le \\& \le
2\sup _{\pm}   \sup _{\lambda \in \R }
\|  R_{ H_{N+1} }^\pm (\lambda )  \| _{ L^{2,\tau}  \to L^{2,-S}   } \|
     \widehat{g} (\lambda,x) \|_{ L^{2,\tau} L_{\lambda  }^2 }\,
 \lesssim  &        \| g\|_{L_{t }^2L^{2,\tau} }.
\endaligned $$

 \qed

\textit{Proof of Proposition \ref{prop:boudres}.}   From  $ R _{H}^- (x,y,E) = \overline{R _{H}^+ (x,y,E)} $
for the bounds it is enough    consider     the case of $ R _{H}^ +$.    From
  \begin{align*}
    R _{H}^+ (x,y,E) =\sigma _3 R _{H}^+ (-y,-x,E) ^ \intercal  \sigma _3
  \end{align*}
      it is enough to consider case  $ x\ge y$. Finally
 \begin{align*}
   R _{H}^+ (x,y,E) =-\sigma _1    R _{H}^- (x,y,-E)  \sigma _1 ,
 \end{align*}
    it is enough to focus on $ E=1+k^2$.

 \noindent After the above reductions, we remark that it  is enough to consider the case when $k$ is close to 0.
  This is  because for $k$ away from 0  a better estimate, without the term $1+x^+ +y ^-$, is already  contained in \cite{KrSch}.  Notice that  the  estimates in \cite{KrSch}    there are some subtleties because, while for $x\ge 0 \ge y$ the desired estimate follows directly form the bounds here stated in Proposition \ref{prop:jost1}, for say  $0 > x \ge  y$ the bounds rely on  \eqref{eq:trrefl1}--\eqref{eq:trrefl2}  and on    formulas   \eqref{eq:trrefl3}--\eqref{eq:trrefl4}    which yield formulas like $ A(k) D ^{-1}(k)=\diag \(  \dfrac{1}{2\im  k}, -\dfrac{1}{\sqrt{2+k^2}} \)  $ which are responsible for some crucial cancellations.

\noindent In the sequel  we consider only the case  $ R _{H}^ +$   for   $E=1+k^2$  and $ x\ge y$ for $k$ close to   0.
We will follow the argument in   \cite{Cu3}, which unfortunately    has some mistakes, but  contains some useful insights that we will review, avoiding errors.  We will prove the following   where to simplify notation we set $f(x,k)=f_1(x,k)$. We will write $ D_k(x)=\frac{1-e^{-2\im kx}}{2k\im  } $ with $D_0(x)=x$.

\begin{lemma}\label{lem:lowenjost1}For $0<|p-3|\ll 1$ and $k\in \overline{\C }_+$ close to 0,      it is possible  to write
  $f(x,k) = e^{\im k
 x} {m}  (x,k)$  where
 \begin{align} \label{eq:lowenjost11}
   {m} (x,k)=&    {e}_1 - \int _x^{  \infty
}  D_k (  x-y )
\diag  (1,0) V(y)  {m}  (y,k) dt \\
& - \int _\R\frac{ e^{-\sqrt{k^2+2  }|x-y|
-\im k(x-y)}}{2\sqrt{k^2+2  }}\diag  (0,1) V(y)
 {m}  (y,k) dt.\nonumber
 \end{align}
 In particular    there exists a constant $C$ such that
 \begin{align} \label{eq:lowenjost12}
   | {m}
(x,k)-  {e} _1 |\le C (1+x^{-} ) \text{  for all $k$ near 0 and $x\in \R $.}
 \end{align}
\end{lemma}
\proof First of all it is easy to see that if ${m} (x,k)$ satisfies \eqref{eq:lowenjost11}, then $f(x,k) = e^{\im k
 x} {m}  (x,k)$ satisfies  \eqref{eq:jost11} and can be taken as the $f_1(x,k)$  in  {Proposition}
  \ref{prop:jost1}. Now let us write ${m}  (x,k)  =( {m}_1 (x,k), {m}_2  (x,k))^ \intercal$, where here $m_1$ and $m_2$ are the two components of $m$ and should not be confused with the $m_1$ and $m_2$ in {Proposition}
  \ref{prop:jost1}. For the first component of $m$ we have
\begin{align*}&
  m_1 (x,k) + \int _x^{  \infty }  D_k (  x-y )V_{11}(y)
m_1 (y,k) dy = 1- \int _x^{  \infty }  D_k (  x-y  )V_{12}(y)
m_2 (y,k) dy .
\end{align*}
It is elementary that the  operators
\begin{align*}
   A_{ij}(k)u := - \int _x^{  \infty }  D_k (  x-y  )V_{ij}(y)
u (y) dy
\end{align*}
are bounded within  the space $ \< x  ^-\>   L ^{\infty } (\R )$ endowed with norm $ \|    v \| _{ \< x  ^-\>  L^\infty (\R )} := \| \< x  ^-\> ^{-1}   v \| _{   L^\infty (\R )}$.  The following is standard and follows from  \cite[Lemma 1 p.130]{DT}.
\begin{claim} \label{claim:1} The operators $(1-  A_{ij}(k)) : \< x  ^-\>   L ^{\infty } (\R ) \to  \< x  ^-\>   L ^{\infty } (\R ) $ can be inverted with norm of the inverse uniformly bounded  in $k\in \overline{\C }_+$.
   \end{claim}
\proof  Consider $  (1-  A_{ij}(k)) u=v$ and write formally the series
\begin{align*}
   \sum _{n=0}^{\infty}  u_n   \text{   with }    u_n = A_{ij} (k) u _{n-1}    \text{  and  }  u_0=v.
\end{align*}
Then, like in \cite[p.132]{DT} and by $|D_k(x-y)|\le |x-y|$,   for $x_0=x$
\begin{align*}
  |u_n(x )|&\le  \int _{x \le x_1\le ...\le x_n} dx_1...dx _n \prod _{\ell=1}^{n}    (x_\ell- x_{\ell-1})|V_{ij}(x_\ell)|   \< x_n\>  \|  \< \cdot \> ^{-1} v \| _{L^\infty (\R )} \\& \le \frac{1}{n!}
  \( \int _{x }^{\infty} (y-x )  \< y\> |V_{ij}(y)| dy \) ^n   \|    v \| _{ \< x  ^-\>  L^\infty (\R )} .
\end{align*}
This means that the series is uniformly convergent  in half--lines  and that for any $x$ we have
\begin{align*}
  u(x) =v(x)- \int _x^{  \infty }  D_k (  x-y  )V_{ij}(y)
u (y) dy.
\end{align*}
Then
\begin{align*}
 | u(x)| &\le |v(x)| + \int _x^{  \infty } y     |V_{ij}(y)| |u (y)|  dy-x   \int _x^{  \infty }      |V_{ij}(y)| |u (y)|  dy \\& \le |v(x)| + \int _0^{  \infty } y     |V_{ij}(y)| |u (y)|  dy +x^-   \int _x^{  \infty }      |V_{ij}(y)| |u (y)|  dy  .
\end{align*}
This implies
\begin{align*}
  \< x  ^-\> ^{-1}  | u(x)| &\le  \|    v \| _{ \< x  ^-\>  L^\infty (\R )}
  + \int _x^{  \infty }  2\< y\> ^2 |V_{ij}(y)|   \< y  ^-\> ^{-1}  | u(y)| dy
\end{align*}
      which in turn implies the following, by an application of Gronwall's inequality,
  \begin{align*}
     \< x  ^-\> ^{-1}  | u(x)| &\le     \|    v \| _{ \< x  ^-\>  L^\infty (\R )}
     \exp \( \int _x^{  \infty }  2\< y\> ^2 |V_{ij}(y)|     dy  \)   .
  \end{align*}

\qed

Thanks to Claim \ref{claim:1}  we can write
\begin{align}\label{eq:lowenjost15}
  m_1 ( \cdot ,k )  =(1-  A_{11}(k)) ^{-1} 1 + (1-  A_{11}(k)) ^{-1} A_{12}(k)m_2 (\cdot ,k)
\end{align}
where if
\begin{equation}\label{eq:lowenjost13}
    |  m_2 (y,k)|\le C \text{  for all $y\in \R$ and for $k$ close to 0,}
\end{equation}
then
\begin{equation}\label{eq:lowenjost14}
    |  m_1 (x,k)|\le C \< x  ^-\>        \text{  for all $x\in \R$ and for $k$ close to 0.}
\end{equation}
For the second component of $m$ we have
\begin{align*}
   m_2 (x,k) &=  - \int _\R\frac{ e^{-\sqrt{k^2+2  }|x-y|
-\im k(x-y)}}{2\sqrt{k^2+2  }}  V_{21}(y)
 {m}_1  (y,k) dy \\& - \int _\R\frac{ e^{-\sqrt{k^2+2  }|x-y|
-\im k(x-y)}}{2\sqrt{k^2+2  }}  V_{22}(y)
 {m}_2  (y,k) dy .
\end{align*}
Using formula \eqref{eq:lowenjost15}, we  can eliminate in the last equation $m_1 ( \cdot ,k )$, obtaining an equation of the form
\begin{align}\label{eq:solm2}&
  (1+ \mathbf{{A}}(k))   m_2 (\cdot ,k) =  - \int _\R\frac{ e^{-\sqrt{k^2+2  }|x-y|
-\im k(x-y)}}{2\sqrt{k^2+2  }}  V_{21}(y)
 (1-  A_{11}(k)) ^{-1} 1  dy  \text{  where}\\&
\mathbf{{A}}(k)u_2:= \int _\R\frac{ e^{-\sqrt{k^2+2  }|x-y|
-\im k(x-y)}}{2\sqrt{k^2+2  }}  \(  V_{22}(y)
 u_2  (y )  +     V_{21}(y) (1-  A_{11}(k)) ^{-1} A_{12}(k)u_2   \)      dy .\nonumber
\end{align}
We want to solve this equation in $L^\infty (\R )$. The operator $\mathbf{{A}}(k)$ is compact from $L^\infty (\R )$ into itself.
If $\ker  (1+\mathbf{A}(k))=0$ for $k=0$  then the same is true for $k$ close to 0 and
by Fredholm theory,  equation  \eqref{eq:solm2}   is solvable  and  Lemma \ref{lem:lowenjost1} is proved.
So suppose now that there exists a nonzero $u _2\in  L^\infty (\R )$ such that $(1+\mathbf{{A}}(0))   u_2 =0$.  Then setting
\begin{align*}
  u_1:= (1-  A_{11}(0)) ^{-1} A_{12}(0)u_2,
\end{align*}
    with $u_1\in     \< x  ^-\>  L^\infty (\R ) $ by   Claim \ref{claim:1}, the pair $u:=(u_1,u_2)^ \intercal$ solves (recall  that $D_0(x)=x$)
\begin{align}\label{eq:solm3}
  u(x)=&      - \int _x^{  \infty
}   (  x-y )
\diag  (1,0) V(y)  u  (y ) dt \\
& - \int _\R\frac{ e^{-\sqrt{ 2  }|x-y|
  }}{2\sqrt{ 2  }}\diag  (0,1) V(y)
 {u}  (y ) dt \nonumber
\end{align}
and $ u(x)  $ is a solution of \eqref{eq:jost11} for $k=0$.
Since  $ u(x) \xrightarrow{x\to +\infty} 0$,  it follows that $ u(x) =c f_3(x,0) $ for a non zero  constant $c\in \C $. This,
   	  $u_1\in    \< x  ^-\>  L^\infty (\R )   $ and $u_2\in    L^\infty (\R )   $
 yield
\begin{align}
  \label{eq:bff_2} |f_3(x,0)|\lesssim 1+x^-.
\end{align}
  The latter is equivalent to
\begin{align}
  \label{eq:bff_22} W[f_3(x,0), g_3( x,0)] =0.
\end{align}
 If this is not true for $p=3$, by continuity of the dependence on the parameter  $p$    of the solutions of system  \eqref{eq:jost11},   \eqref{eq:bff_2} is not true
  for any $p$ close to 3.
 \begin{claim} \label{lem:plwcubic} For $p=3$ we have  $|f_3(x,0)| \sim  e ^{\sqrt{ 2}\ |x|}$   as $x\to -\infty$.
   \end{claim}
 \proof Follows immediately from  formula  \eqref{eq:plwave2} below which yields a function proportional to $g_3(x,0)$ and from $ f_3(x,0)= g_3( -x,0)$. \qed

  From Claim \ref{lem:plwcubic}  we conclude that   $ W[f_3(x,0), g_3( x,0)]\neq 0$  for $p=3$ and also for $p$ close to 3. Hence  $\ker  (1+\mathbf{{A}}(0))=0$.  This completes the proof of Lemma \ref{lem:lowenjost1}.  \qed

 \textit{Proof of Proposition \ref{prop:boudres}: continuation and end.}  We have already discussed the fact that the desired bound for the kernel  $ R _{H}^ +(x,y,E)$   for   $E=1+k^2$  and $ x\ge y$ and $k$ outside a neighborhood of 0 are true by Krieger and Schlag \cite{KrSch}. So now we consider the case when $k$ is small. Then  by the bound  \eqref{eq:lowenjost12} for the  $f_1(x,k)$ in \eqref{eq:jost12} and by the exponential decay to 0  for $x\to -\infty$ of $g_3(x,k) =f_3(-x,k)$, it follows that
 \begin{align*}
     W[ f_1(x,k),g_3(x,k)]=  W[ f_3(x,k),g_1(x,k)]=0.
 \end{align*}
  So the matrix $D(k)$ in \eqref{eq:defd} is diagonal. This then implies, similarly to the proof in Kriger and Schlag \cite{KrSch}, that  for $ x\ge y$ and for  $k$ small,
 \begin{align}&
   R _{H}^+ (x,y,1+k^2)=  (f_1(x,k)  , 0 )D ^{-1}(k) (g_1(y,k)   , 0 )  ^ \intercal + (0  , f_3(x,k) )D ^{-1}(k) (0 , g_3(y,k)  )  ^ \intercal .\label{eq:diagD}
 \end{align}
 We bound this  for $0>x>y$. The first term can be bounded by a constant times $\< x ^-\> $ because $|g_1(y,k)|\lesssim 1 $  for $y<0$ and, by  \eqref{eq:lowenjost12},
   $|f_1(x,k)|\lesssim \< x ^-\> $ for $x\le 0$. The second term  is uniformly bounded, because   $|f_3(x,k)|\lesssim  e^{\sqrt{2+k^2} |x|} $ for $x\le 0$  and  $|g_3(y,k)|\lesssim e^{-\sqrt{2+k^2} |y|}$
  for $y\le 0$  when $k$ is sufficiently small. So this yields
  \begin{align*}
     | R _{H}^+ (x,y,1+k^2)| \lesssim  \< x ^-\>  \text{  for $0\ge x\ge y$ and for $k$ close to 0.}
  \end{align*}
  By exploiting the symmetries due to $g_j(x,k)=f_j(-x,k)$ and by similar estimates, we obtain also the estimate
 \begin{align*}
     | R _{H}^+ (x,y,1+k^2)| \lesssim  \< y ^+\>  \text{  for $ x\ge y \ge 0$ and for $k$ close to 0.}
  \end{align*}
  So we obtained the estimate      \eqref{eq:boudres1}   for all $ x\ge y  $.  This completes the proof of Proposition \ref{prop:boudres}.
  \qed

\textit{Proof of Lemma \ref{lem:lemma11}.}
The group $e^{\im t H}$ is continuous and, see \cite[Lemma 6.11]{KrSch} equibounded, with infinitesimal generator $\im H$.  Then
     for $a>0$ and $u_0,v_0\in L^2(\R , \C^2)$ by the Hille Yoshida theorem, Goldstein \cite[p. 17]{gold},   we have
\begin{align*}&
   \<\im R_{H }(\lambda -\im a)u_0 ,v_0\> = \int _{0}^{+\infty}   \<e^{\im t  (H-\lambda +\im a)}u_0dt,v_0\>    = \int _{0}^{+\infty} e^{-\im t   \lambda  }  \<e^{\im t  (H  +\im a)}u_0,v_0\>  dt\text{  and }\\&  \<- \im R_{H }(\lambda +\im a)u_0 ,v_0\> = \<\int _{-\infty }^{0}e^{\im t  (H-\lambda -\im a)}u_0 dt   ,v_0\>   = \int _{-\infty }^{0} e^{-\im t   \lambda  }  \<e^{\im t  (H  -\im a)}u_0,v_0\> dt .
\end{align*}
Then
\begin{align*} &
 \frac{1}{2\pi}  \int _\R   e^{ \im t   \lambda  } \<\im R_{H }(\lambda -\im a)u_0 ,v_0\> d\lambda = \chi _{\R _+} (t)\<e^{\im t  (H  +\im a)}u_0   ,v_0\>  \text{   and }\\& - \frac{1}{2\pi}  \int _\R   e^{ \im t   \lambda  } \<\im R_{H }(\lambda +\im a)u_0 ,v_0\> d\lambda = \chi _{\R _-} (t)\<e^{\im t  (H  -\im a)}u_0   ,v_0\> .
\end{align*}
So for $g$, which for convenience we take in $  \mathcal{S}(\R \times \R , \C  ^2)\cap C _{c} (\R _t , L^2(\R _x, \C^2))$, we have
\begin{align*}&
    \frac{\im}{2\pi}\int _\R e^{ \im t   \lambda  }\im R_{H }(\lambda -\im a)  g ^{\vee}( \lambda , \cdot ) d\lambda =\int _{-\infty }^{t} e^{\im (t-t') (H+\im a)} g(t') dt'     \text{   and}\\&
    -\frac{\im}{2\pi}\int _\R e^{ \im t   \lambda  }\im R_{H }(\lambda +\im a)  g ^{\vee}( \lambda , \cdot ) d\lambda =\int _{t }^{+\infty} e^{\im (t-t') (H-\im a)} g(t') dt'   . \end{align*}
Summing up and after an elementary manipulation, for $t>0$ we have
\begin{align}& \label{eq:mizfor1}
   \frac{\im}{2\pi}\int _{\R} e^{ \im t   \lambda  } \(  R_{H }(\lambda -\im a)  + R_{H }(\lambda +\im a)  \) g ^{\vee}( \lambda , \cdot ) d\lambda \\&  = \int_{-\infty}^{0}   e^{- (t-t')a} e^{\im (t-t') H} g(t') dt'   - \int _{0 }^{+\infty}  e^{  (t-t')a} e^{\im (t-t') H} g(t') dt' \nonumber \\&\nonumber + \int_{0}^{  t} e^{- (t-t')a} e^{\im (t-t') H}  g(t') dt'+ \int_{0}^{t}  e^{  (t-t')a} e^{\im (t-t') H} g(t') dt' \xrightarrow{a\to 0^+} \\& \int_{-\infty}^{0}    e^{\im (t-t') H} g(t') dt'   - \int _{0 }^{+\infty}   e^{\im (t-t') H} g(t') dt'   + 2\int_{0}^{t}   e^{\im (t-t') H}  g(t') dt'\label{eq:mizfor2}
\end{align}
where the limit of the right hand side holds in $L^2(\R )$   by $ e^{\pm a t'}g(t')  \xrightarrow{a\to 0^+} g$ in
$L^1(\R ,L^2(\R _x, \C^2))$ and by the Strichartz estimates, see Keel and Tao \cite[Theorem 1.2]{Kl-Tao}. We now focus at the limit of line \eqref{eq:mizfor1} as $a\to 0^+$   when  $P_c g(t) =g(t)$ for all times. We claim that
\begin{align}& \label{eq:mizfor3}
\lim _{a\to 0^+ } \text{line \eqref{eq:mizfor1}}= \frac{\im}{2\pi}\int _{\R} e^{ \im t   \lambda  } \(  R_{H }^-(\lambda  )  + R_{H }^+(\lambda  )  \) g ^{\vee}( \lambda , \cdot ) d\lambda \text{ in } L^{2,-s}\( \R \)
\end{align}
for $s>3/2$.  We distinguish between three cases.   For $ \lambda \in [-1+\alpha , 1-\alpha    ]$ for any fixed $\alpha \in (0,1)$ we have uniform convergence of the resolvents in the operator norm, where $P_c g ^{\vee}( \lambda) =g ^{\vee}( \lambda)$ avoids the singularities of the resolvent. For   such that $ \Re z \in (-\infty , -1-\alpha    ]\cup  [1+\alpha , \infty )$  and $\Im z \ge 0$ (resp.  $\Im z \le 0$)   it is possible to apply the 3 dimensional theory in \cite{Cu} to conclude that $ R_{H }(z)$  is continuous as a function with values in the space of of bounded operators operator from   $L^{2, s}\( \R \)$  to $L^{2,-s}\( \R \)$ for $s>1/2$.  We finally consider the case when  $\lambda  $ is close to $\{ 1, -1  \}$. For symmetry reasons it is not restrictive to consider the limit of
$R_{H }(\lambda +\im a)$   for $\lambda$   close to 1 focusing on the corresponding integral kernel in the region $x\ge y$. Notice that Lemma \ref{lem:lowenjost1} continues to be true for
$  k\in \C \backslash [0,+\infty ) $ with $k$ near 0. The resolvent of $R_{H }(\lambda +\im a)$ is given by  \eqref{eq:ris+} with $\lambda +ia =1+k^2$ and we continue to  have the diagonalization of $D(k)$ yielding to \eqref{eq:diagD}. Then $R_{H }(x,y,\lambda +\im a)$  satisfies the estimate \eqref{eq:boudres1} and by dominated convergence  we obtain the desired limit.

\qed

\section{Proof of Proposition \ref{lem:smooth111}}\label{sec:smoothing}

 We will prove the  following version for $H$, that is in the case $\omega =1$, of \eqref{eq:smooth111}, where $P_c$ is the spectral  projection associated to the continuous spectrum of $H$,  for any $u_0\in L^2(\R  , \C^2)$ and for a fixed constant $c>0$,
\begin{equation}\label{eq:smooth111b}
   \|   e^{\im t \mathcal{L} }P_c u_0 \| _{L^2(\R , L ^{2,-s}(\R   ))} \le c \|  u_0 \| _{L^2(\R  )}  .
\end{equation}       Let $g(t,x) \in   \mathcal{S}(\R^2) $ with $g(t)=P_c(H  )g(t)$.
Then
\begin{align*}
   \int _\R \< e^{-\im tH  } u_0,\sigma _3
g\>   dt&=  \frac{1}{\sqrt{2\pi}\im }\int _{\R}
  \<
 (R_{H  }^{+}(E  )-R_{H
}^{-}(E   ))   u_0,\sigma _3  \overline{\widehat{g}}(E
)\> _x  dE   \\& = \frac{1}{\sqrt{2\pi}\im  }\int_{ {\sigma _c(H)} }
  \<
 (R_{H  }^{+}(E   )-R_{H
}^{-}(E   ))    u_0,  \sigma _3\overline{\widehat{g}}(E
)\> _x  dE   .
\end{align*}
Then from Fubini   we have
$$\aligned & \left  |  \int _\R \< e^{-\im  tH  }  u_0, \sigma _3
g\>   dt \right
|
 \le
\| (R_{H  }^{+}(E   )-R_{H  }^{-}(E  ))
 u_0\|_{L^{2, -s}_xL^2_E (\sigma _c(H  ))}     \| {g}
\|_{L^{2, s}_x L^2_t  }
 .
  \endaligned $$
So now we need to show that
 \begin{align*}
    \| (R_{H  }^{+}(E   )-R_{H  }^{-}(E  ))
 u_0\| _{L^{2, -s}_xL^2_E  (\sigma _c(H  ))} \lesssim \|  u_0 \| _{L^2(\R )},
 \end{align*}
 where the subscripts $x$ and $E$ indicate the variables of integration.
We can split between  $E  $ away from the thresholds of $\sigma _c(H  )$,  where the corresponding    bound is obtained thanks to the corresponding bound  for the flat operator $\sigma _3 (-\partial ^2_x+1)$  like in the 3 dimensional case,  proved in \cite{CPV},
 and the case when $E  $ is close  to   the thresholds $\pm 1 $. More generally, we will show that  for $s>1$  there is a constant $C_s$
\begin{align}\label{locsmooth}
    \|  R_{H  }^{\pm}(E   )
 u_0\| _{L^2   (|E-1|\ll 1,     L^{2, -s}(\R ) )}   \le     C_s \|  u_0 \| _{L^2(\R )},
 \end{align}
with an analogous estimate   valid   near $-1$.   Let us consider the  scalar  Schr\"odinger operator $h = -\partial _x ^2   +  \sech ^2 \(  \frac{p-1}{2} x \)$. Then  we claim that
$\<x\> ^{-s}$  is $h$--smoothing in the sense of Kato \cite{kato}, which implies that for $s>1$  there is a constant $C_s$ such that
\begin{align}\label{eq:katosm}
    \|  R_{\sigma _3 (h +1)  }^{\pm}(E   )
 u_0\| _{L^2 _E  (\R,     L^{2, -s}_x(\R ) )}   \le C_s  \|  u_0 \| _{L^2(\R )}.
 \end{align} Since $ \sigma _3 (h +1) $ is selfadjoint,
by   (5.3) in
Theorem 5.1 \cite{kato},   \eqref{eq:katosm} will follow  if for a fixed $C>0$
\begin{equation}\label{eq:katosm1}
  \| \langle x
\rangle ^{-s}R_{\sigma _3 (h +1)    } (z ) \langle x \rangle ^{-s} \| _{  L^{2}(\R ) \to
L^{2}(\R ))}<C \text{  for all $z$ with $0<|\Im z |  $}.
\end{equation}
From $R_{\sigma _3 (h +1) } (z ) =\diag \(  R_{h  } (z -1) ,  -R_{h  } (z +1)  \)$,
\eqref{eq:katosm1}  follows from
\begin{equation}\label{eq:katosm2}
  \| \langle x
\rangle ^{-s}R_{  h  } (z ) \langle x \rangle ^{-s} \| _{  L^{2}(\R ) \to
L^{2}(\R ))}<C \text{  for all $z$ with $0<|\Im z |  $}.
\end{equation}
The kernel of $R_{  h  } (z )$ for $x<y$, with an analogous formula for $x>y$,   for $\arg \sqrt{z} \in [0,\pi] $  is \begin{align}&   R _{h  }(z ) (x,y) =   \frac{T(\sqrt{z})}{2\im \sqrt{z}}     f_- (x, \sqrt{z})   f_+ (y, \sqrt{z})   =   \frac{T(\sqrt{z})}{2\im \sqrt{z}}     e^{\im \sqrt{z} (x-y)}       m_- (x, \sqrt{z})   m_+ (y, \sqrt{z})  ,   
\nonumber\end{align}
where  the  Jost functions  $f_{\pm } (x,\sqrt{z} )=e^{\pm \im \sqrt{z} x}m_{\pm } (x,\sqrt{z} )$  solve $ h u=z u$ with\begin{align*} \lim _{x\to +\infty }   {m_{ + } (x,\sqrt{z} )}  =1 =\lim _{x\to -\infty }  {m_{- } (x,\sqrt{z})}  .\end{align*}  These functions  satisfy, see Lemma 1 p. 130 \cite{DT},\begin{align} 
\nonumber &  |m_\pm(x, \sqrt{z} )-1|\le  C _1 \langle  \max \{ 0,\mp x \}\rangle\langle \sqrt{z} \rangle ^{-1}\left | \int _x^{\pm \infty}\langle y \rangle  \sech ^2 \(  \frac{p-1}{2} y \)     dy \right |     \\ &  |m_\pm(x, k )-1|\le    \<    \sqrt{z}\>  ^{-1}  \left | \int _x^{\pm \infty}  \sech ^2 \(  \frac{p-1}{2} y \)   dy \right |    \exp \(  \<    \sqrt{z}\>  ^{-1}  \left | \int _x^{\pm \infty}   \sech ^2 \(  \frac{p-1}{2} y \)  dy \right | \)   ,\nonumber
\end{align} while, since $h$ has no 0 resonance,  $T(k) =\alpha k (1+o(1))$  near $k=0$ for some $\alpha \neq 0$ and $T(k) = 1+O(1/k) $ for $k\to \infty$ and $T\in C^0(\R )$, see Theorem 1 \cite{DT}.
Then (here for $z\in \R$ we are taking $R_{  h  }^{+} (z )$)
 \begin{align*}
   \| \langle x
\rangle ^{-s}R_{  h  } (z ) \langle x \rangle ^{-s} \| _{  L^{2}(\R ) \to
L^{2}(\R ))}^2 &\le  \int _{x<y} \langle x
\rangle ^{-2s}  \langle y
\rangle ^{-2s}  |R _{h  }(z ) (x,y)| ^2 dxdy  \\& +  \int _{x>y} \langle x
\rangle ^{-2s}  \langle y
\rangle ^{-2s}  |R _{h  }(z ) (x,y)| ^2 dxdy =:\mathcal{A}+\mathcal{B} .
 \end{align*}
The two terms on the right can be estimated similarly, so we bound only the first. It is easy to see   that, like in Proposition \ref{prop:boudres},
\begin{align}
  \left |       R _{h  } (z ) (x,y)            \right | \le C  \left\{\begin{matrix}      (1+x^- +y ^+)\text{   if  }  x\ge y   \\
   (1+x^+ +y ^-)\text{   if  }  x\le y ,   \end{matrix}   \right .  \label{eq:bdRh}
\end{align}
where in fact these estimates are what  inspired \eqref{eq:boudres1}.
Then
\begin{align*}
   \mathcal{A} &\lesssim    \int _{0<x<y} \langle x
\rangle ^{2-2s}  \langle y
\rangle ^{-2s}    dxdy +  \int _{ x<y<0} \langle x
\rangle ^{ -2s}  \langle y
\rangle ^{2-2s}    dxdy  +  \int _{  x<0<y} \langle x
\rangle ^{ -2s}  \langle y
\rangle ^{-2s}    dxdy \\&=: A_1+A_2+A_3.
\end{align*}
Then
\begin{align*}
   A_1 & = \int _{\R_+} dx \langle x
\rangle ^{2-2s}  \int _{x}^{+\infty}  \langle y
\rangle ^{-2s}    dy \lesssim  \int _{\R_+}  \langle x
\rangle ^{3-4s}    dx<+\infty \text{ for } 3-4s<-1 \Longleftrightarrow s>1.
\end{align*}
Similarly $A_2$,  obviously also $A_3$,  and $\mathcal{B}$ are bounded for   $s>1$.  So this yields \eqref{eq:katosm2} and \eqref{eq:katosm1}. In particular this implies \eqref{eq:katosm}.
 Now we can express
 \begin{align*}
    H = \sigma _3 (h +1) + \widetilde{V}
 \end{align*}
 where $\widetilde{V}  =M_0  \ \sech ^2 \(  \frac{p-1}{2} x \)$, for $M_0$ a constant matrix.  We can  factor
 \begin{align*}
    \widetilde{V} =B^*A \text{  with }   B^* =  \langle x
\rangle ^{  s}  \widetilde{V}   \text{  and }A=    \langle x
\rangle ^{- s}    .
 \end{align*}
 Now, for $\Im z >0$, for $ Q_0 (z) =A R_{\sigma _3 (h +1)    } (z ) B^*$, we have
 \begin{align*}
  A R_H(z)= \( 1+  Q_0 (z) \) ^{-1}A R_{\sigma _3 (h +1)    } (z ).
 \end{align*}
 The function $ Q_0 (z)$ extends as  an element of $C^0(\overline{\C} _+ \backslash \sigma _p (H), \mathcal{L}  (L^2 ))$ with values in the space of compact operators of $L^2=L^2(\R , \C^2)$ in itself.
  Furthermore, $\( 1+  Q_0 (z) \) ^{-1}$ extends into a bounded operator except for those $z\in \R$ for which $\ker \( 1+  Q_0 (z) \)  \neq 0$.  For $z$ near but not equal to 1, by standard arguments that can be seen in \cite{CPV}, this implies that $z$ is an eigenvalue of $H$, but this is not possible since here our $z$'s are taken much closer to 1 than $\lambda (p,1)$, which is the only positive eigenvalue of $H$. The other possibility is that $\ker \( 1+  Q_0 (1) \)  \neq 0$.   We exclude this  proceeding by contradiction. If
  \begin{align*}
     \( 1+  Q_0 (1) \)  w  =0 \text{ with }  w  \neq 0
  \end{align*}
    then  $\psi = R_{\sigma_3(h+1 )}^{+} (1 ) B^\ast w$  satisfies
   \begin{align*}
      \( \sigma_3(h+1 ) -1   \) \psi =B^\ast w =- B^*A  R_{\sigma_3(h+1 )}^{+} (1 ) B^\ast w = -\widetilde{V}\psi
   \end{align*}
 and so
$\psi \neq 0$  is
 a
nontrivial distributional solution of $(H  -1 )u=0$. We claim that $\psi \in L^\infty (\R )$. In fact, for $g=B^\ast w$,
\begin{align*}
   |\psi(x) |\le  \int _{x<y}  |R_{\sigma _3 (h +1)  }^{+}(x,y, 1   )|   \ | g(y)| dy +  \int _{x>y}  |R_{\sigma _3 (h +1)  }^{+}(x,y, 1   )|   \ | g(y)| dy =:B_1(x) +B_2(x).
\end{align*}
Then by \eqref{eq:bdRh} we have
\begin{align*}
  B_1(x) \le  \int _{x<y}  (1+x^+ +y ^-)   \ | g(y)| dy
\end{align*}
If now $x<0$ then by the rapid decay of $g$  we get
\begin{align*}
  B_1(x) \le    \int _{\R}  (1+|y|) \ | g(y)| dy <\infty .
\end{align*}
If $x>0$ we write
\begin{align*}
  B_1(x) \le    \int _{x<y}  (1+|x|) \ | g(y)| dy   \le \int _{\R}  (1+|y|) \ | g(y)| dy  <\infty .
\end{align*}
So $B_1\in L^\infty (\R )$. By a similar argument we obtain $B_2\in L^\infty (\R )$  and hence also $\psi \in L^\infty (\R )$.
But then 1 is a resonance for $H$, which is not true.  So we conclude that $\ker \( 1+  Q_0 (1) \)  = 0$. Then
\begin{align*}
  &\|  R_{H  }^{\pm}(E   )
 u_0\| _{L^2   (|E-1|\ll 1,     L^{2, -s}(\R ) )} \\&= \|    \( 1+  Q_0 (E) \) ^{-1}  \< x \> ^{-s}  R_{\sigma _3(h+1)  }^{\pm}(E   )
 u_0\| _{L^2   (|E-1|\ll 1,     L^{2 }(\R ) )}\\& \lesssim  \|      \< x \> ^{-s}  R_{\sigma _3(h+1)  }^{\pm}(E   )
 u_0\| _{L^2   (|E-1|\ll 1,     L^{2 }(\R ) )}  \lesssim \|  u_0 \| _{L^2(\R )}
\end{align*}
and \eqref{locsmooth} for the $+$ and for $-$. This completes the proof of \eqref{eq:smooth111b} and so also of Proposition
  \ref{lem:smooth111}.

  \section{Explicit Jost functions  of the linearization for $p=3$.}\label{sec:intsy}

When $p=3$ the Jost functions discussed in \S  \ref{sec:lin} have been  explicitly known  since Kaup \cite{kaup}. In fact it was shown that    these Jost functions  can be expressed in terms of the solutions of the Lax pair system. In turn, the latter ones can be expressed in terms of the solutions of the Lax pair system for the null solution of the NLS  using B\"{a}cklund transformations.
However here we will use some transformations in  Martel \cite{Martel1,Martel2}  to  write these explicit formulas. It is not restrictive to take $\omega =1$.
 For  $\omega =1$ we have $L _{+ }=L _{0 }$ and $L _{- }=L _{1 }$ where,  Chang et al. \cite{Chang},
\begin{align}\label{eq:opLj}
 L_{j }:=  - \partial _x^2   +1-k_{j-1}(p)k_{j }(p) \frac{2}{p+1}\phi   ^{p-1} \text{  for } j=0,1,2,... \text{ and } k_j(p) := \frac{p+1}{2}-\frac{j(p-1)}{2}.
\end{align}
 Notice that
 \begin{align*}
  k_1(p)=1 \, , \quad k_2(p)= \frac{3-p}{2} \text{  and } k_3(p)= 2-p.
 \end{align*}
 When $p=3$,  $L_{2 }= L_{3 } =- \partial _x^2   +1$.
  Let
  \begin{align*}
     S_1=S(k_1(p)):=\partial _x + k_1(p)  \tanh \(  \frac{p-1}{2} x    \)  =  \partial _x +     \tanh \(  \frac{p-1}{2} x    \)   .
  \end{align*}
  Martel  \cite{Martel1,Martel2} exploits the following formula,  which we derive for reader's sake,
 \begin{align}
   & S_1^2 L_0 L_1=  S_1^2 L_0S_1^*S_1= S_1S_1^*  L_3 S_1^2=L_2L_3  S_1^2,\nonumber
 \end{align}
 where the first and last equalities follow from (3.14) and the   second  from (3.24) in Chang et al. \cite{Chang}.
Taking the adjoint   we obtain
\begin{align}\label{eqmartelid2}
   &   L_1 L_0 (S_1^*)^2   =   (S_1^*)^2  L_3L_2.
 \end{align}
This formula is exploited in  Martel  \cite{Martel1,Martel2} to show that  starting from
 \begin{align}\label{eqmartelid3}
    \left\{\begin{matrix}      L_2 w_1=\lambda w_2   \\
    L_3 w_2=-\lambda w_1
\end{matrix}\right.
 \end{align}
 we get
 \begin{align}\label{eqmartelid4}
    \left\{\begin{matrix}     \xi_1:=(S_1^*)^2 w_ 1   \\
   \xi_2:= - \frac{1}{\lambda} L_0 \xi_1
\end{matrix}\right.   \Longrightarrow     \left\{\begin{matrix}      L_1 \xi_2=\lambda \xi_1   \\
    L_0 \xi_1=-\lambda \xi_2 ,
\end{matrix}\right.
 \end{align} where $ L_0 \xi_1=-\lambda \xi_2$  is true by definition and
 \begin{align*}   L_1 \xi_2 =  - \frac{1}{\lambda}    L_1L_0 (S_1^*)^2 w_ 1 = - \frac{1}{\lambda}    (S_1^*)^2  L_3L_2 w_ 1 =  -(S_1^*)^2  L_3  w_ 2=\lambda (S_1^*)^2 w_1= \lambda \xi_1 .
\end{align*}
  For $p=3$, by  $L_{2 }= L_{3 } =- \partial _x^2   +1$  for $\lambda =\im (1+k^2)  $ we consider solutions to
\eqref{eqmartelid3}  of the form
\begin{align}\label{eq:formw00}
  (w_1,w_2) ^\intercal =  \left\{\begin{matrix}   e^{\im k x} ( 1, -\im ) ^\intercal      \\
   e^{\mu x} ( 1, \im ) ^\intercal \text{ where }\mu :=\sqrt{2+k^2}
\end{matrix}\right.
\end{align}
and by \eqref{eqmartelid4}, after elementary computations, we obtain Jost functions for $\left . \mathcal{L}_{\omega}\right | _{\omega =1}$ for $p=3$,
\begin{align}\label{eqmartelid5}
    \left\{\begin{matrix}   e^{\im k x} \( 1-k^2 -2\im k\tanh (x) -2\sech ^2 (x),  \im  \(  1-k^2 -2\im k \tanh (x) \)   \) ^\intercal      \\
   e^{\mu x} \( \mu ^2 +1- 2\mu \tanh (x) -2\sech ^2 (x) ,  \im  \(  \mu ^2 +1- 2\mu \tanh (x) \)      \) ^\intercal   .
\end{matrix}\right.
\end{align}
Notice that
\begin{align}\label{eqres1}
  \left . e^{\im k x} \( 1-k^2 -2\im k\tanh (x) -2\sech ^2 (x),  \im  \(  1-k^2 -2\im k \tanh (x) \)   \) ^\intercal \right | _{k=0}=    \( 1- 2\sech ^2 (x),  \im     \) ^\intercal
\end{align}
yields the resonance at the threshold $\im$, see formula (3.54) Chang at al. \cite{Chang}.
  Eigenfunctions  for the operator $H$ are obtained  applying to the vectors in  \eqref{eqmartelid5} the matrix $U^{-1}$ yielding
\begin{align*} (u_1, u_2) ^\intercal :=2^{-1} ( \xi_1-\im \xi_2  ,  \xi_1+\im \xi_2) ^\intercal .
\end{align*}
Entering in this formula  the functions       in  \eqref{eqmartelid5} we obtain
\begin{align}\label{eq:plwave1}
   \begin{pmatrix} u_1   \\  u_2 \end{pmatrix} &= e^{\im kx } \begin{pmatrix} 1-k^2 -2\im k\tanh (x) - \sech ^2 (x)   \\   - \sech ^2 (x) \end{pmatrix} \text{   and } \\   \begin{pmatrix} u_1   \\  u_2 \end{pmatrix} &=  e^{\mu x} \begin{pmatrix}  - \sech ^2 (x)   \\   \mu ^2 +1- 2\mu \tanh (x) - \sech ^2 (x) \end{pmatrix}  , \label{eq:plwave2}
\end{align}
which are the Jost functions of $H$  for $p=3$.

\section{The linear approximation  of  $\gamma (p,1) $ at $p=3$.}\label{sec:FGRconst}

In this section we prove Lemma \ref{lem:FGRnondeg} by following Martel \cite{Martel2}.
We will focus only on  $\gamma (p ) :=\gamma (p,1) $, since the general $\omega>0$  case follows from the   $\omega =1$
by  scaling.
We write $\phi_p$ to denote $\phi$ given in \eqref{eq:sol}.
Similarly, when it is necessary to stress the dependence on  $p$, we write $g^{(1)}=g_p=(g_{p,1},g_{p,2})^\intercal $, $\xi_{\omega=1}=\xi_p=(\xi_{p,1},\xi_{p,2})^\intercal$, $\left.\mathcal{L}_\omega\right|_{\omega=1}=\mathcal{L}_p$ and $\left.L_{\pm \omega}\right|_{\omega=1}=L_{p\pm}$.
For $g_p$, like in \eqref{eq:reimxi1}, we  take $g_{p,1}=\Re g_{p,1}$ and $g_{p,2}=\im \Im g_{p,2}$.
In the following, we choose $$\xi_3=\(1-\phi_3^2,\im\) ^\intercal \text{ and }g_3=\(\frac{1}{2}\phi_3^2\cos(x)+\frac{\phi_3'}{\phi_3}\sin(x),\im \frac{\phi_3'}{\phi_3}\sin(x)\) ^\intercal,$$
where $\xi_3$ is   just a resonance and not an eigenfunction.
\begin{remark}
Notice that here for $\xi_p$ we are not using the normalization in \eqref{eq:xinormliz1} and instead  we are defining it  as a solution of \eqref{eq:eqv01} which will be defined    in Lemma \ref{lem:expv}.
On the other hand, $g_3$ is given multiplying  by $-\frac{1}{2}$  the vector given in \eqref{eqmartelid5} with $k=1$ and then taking the real part for the first component and imaginary part for the second component respectively.
\end{remark}

Also, for $g_p$ we have the following lemma which is a variant of Lemma 19 of Martel \cite{Martel2}.
\begin{lemma}
We can choose $g_p$ so that
\begin{align*}
\|g_p-\(\frac{1}{2}\phi_3^2\cos(\tau x)+\frac{\phi_3'}{\phi_3}\sin(\tau x),\im \frac{\phi_3'}{\phi_3}\sin(\tau x)\)\|_{L^\infty}\lesssim |p-3|,
\end{align*}
where $\tau=\sqrt{1-\lambda(p,1)^2}$.
\end{lemma}

\begin{proof}
The proof is parallel to Lemma 19 of Martel \cite{Martel2}.
\end{proof}

We will start by writing an expansion in $p$ of the eigenfunction $\xi_p$.
Along the way we give  a new proof, based on Martel \cite{Martel2}, of the result  by Coles and Gustafson \cite{coles} about  the existence of an eigenvalue.

\begin{lemma}
	\label{lem:existeig} There exists a small $\delta _1>0$ and a function   $\alpha \in C^\infty (D _{\R} (3,\delta _1) , \R )$ such that  \begin{align}\label{eq:expansalpha}
		\alpha (p)=   (p-3) ^2 \(   2^{-2} + 2^{-5}2^{-\frac{1}{2}}\< \phi_3^2, \mathbf{ T}        \>      \) + O\( (p-3) ^3    \),
	\end{align}
where
\begin{equation}\label{eq:defboldT}
	\mathbf{T}:= \dfrac{e^{- \sqrt{2}  |\cdot | } }{2 } * \phi_3^2,
\end{equation}
	and such that $\im (1-\alpha ^2(p))$ is an eigenvalue of $ \mathcal{L}_p$  for $0<|p-3|<\delta _1$.
	That is, $\lambda(p,1)=1-\alpha(p)^2$.
	
\end{lemma}

\begin{remark}
Notice that $ (-\partial ^2_x+2) \mathbf{T}=\sqrt{2}\phi_3^2  $.
\end{remark}
\proof  We are looking  to solutions to
\begin{align}  &
	\left\{\begin{matrix}
		L_{-  }\xi  _{p,2} =\im (1-\alpha ^2)  \xi  _{p,1} \\ L_{+  }\xi  _{p,1}  =-\im (1-\alpha ^2) \xi  _{p,2}
	\end{matrix}\right.  . \label{eq:eqv01}
\end{align}
By \eqref{eqmartelid3}, this is equivalent to the existence of  $w_p = (w  _{p,1}  ,  w _{p,2}) ^\intercal$ such that
\begin{align}\nonumber & \begin{pmatrix} 0 &    -L_3   \\  L_2      &  0\end{pmatrix}    \begin{pmatrix} w  _{p,1}   \\  w _{p,2} \end{pmatrix} \\& = \( - J(- \partial _x^2   +1) +k_{2}(p)\frac{2}{p+1}\phi_p  ^{p-1}
	\begin{pmatrix} 0 &    k_{3 }(p)  \\  -1     &  0\end{pmatrix}\) w _p = \im (1-\alpha ^2)  w _p  . \label{eq:eqw01}
\end{align}
Applying  $U ^{-1}$, recall $U$ is given in \eqref{def:U},  to this equation
and introducing $Z _{p }=:U ^{-1}  w _{p, } $,  after elementary computations we get the equivalent problem
\begin{align*}&      \( -\sigma _3 (- \partial _x^2   +1)  +k_{2}(p)\frac{1}{p+1}\phi_p  ^{p-1}
	\begin{pmatrix} k_{3 }(p)+1 &    1-k_{3 }(p)  \\  k_{3 }(p)-1     &  -(k_{3 }(p)+1)\end{pmatrix}  \)  Z _{p }=    (1-\alpha ^2)  Z  _{p } .
\end{align*}
Substituting the values  of $k_{2}(p)$   and $k_{3}(p)$ and multiplying by $-\sigma _3$,  this can be written
\begin{align*}
	  \(  (- \partial _x^2   +1)   + (p-3)  \mathbf{P}_p(x)   \)  Z _{p }= - \sigma _3  (1-\alpha ^2)  Z  _{p }
\end{align*}
with
\begin{align*}\mathbf{P}_p(x) = \begin{pmatrix} 3-p & p-1   \\  p-1 & 3-p \end{pmatrix}  \frac{1}{2(p+1)}\phi  ^{p-1}_p (x).
\end{align*}
Notice in particular that
\begin{align}\label{eq:formulap3}
	\mathbf{P}_3(x) =  \sigma _1 \frac{1}{4}\phi_3  ^{2}  (x) = \sigma _1 \frac{1}{2}\sech   ^{2}  (x).
\end{align}
For
\begin{align*}     H _\alpha  = \begin{pmatrix}  - \partial _x^2   + \kappa ^2 &   0  \\      0  &  - \partial _x^2   + \alpha ^2 \end{pmatrix} \text{  with }\kappa  =\sqrt{2- \alpha^2},
\end{align*}
we can write
\begin{align*}
	\( H _\alpha   + (p-3)  \mathbf{P}_p \) Z_p=0,
\end{align*}
 which is equivalent to
 \begin{align*}
 	 \( 1 + (p-3) H _\alpha ^{-1}\mathbf{P}_p \) Z_p=0.
 \end{align*}
If we set
\begin{align*}
	|\mathbf{P}_p(x) | ^{\frac{1}{2}}& :=   \begin{pmatrix}
		1+\sqrt{p-2} & 1-\sqrt{p-2}\\
		1-\sqrt{p-2} & 1+\sqrt{p-2}
	\end{pmatrix}      \frac{1}{2 \sqrt{p+1} }\phi_p  ^{\frac{p-1}{2}} (x) \text{ and }  \\ \mathbf{P}_p ^{\frac{1}{2}}(x) &:=    \sigma _1  |\mathbf{P}_p(x) | ^{\frac{1}{2}} =   \begin{pmatrix}
		1-\sqrt{p-2} & 1+\sqrt{p-2}\\
		1+\sqrt{p-2} & 1-\sqrt{p-2}
	\end{pmatrix}      \frac{1}{2 \sqrt{p+1} }\phi_p  ^{\frac{p-1}{2}} (x) ,
\end{align*}
from the elementary computation
\begin{align*}
	\begin{pmatrix}
		1-c & 1+c\\
		1+c & 1-c
	\end{pmatrix}
	\begin{pmatrix}
		1+c & 1-c\\
		1-c & 1+c
	\end{pmatrix}
	&=\begin{pmatrix}
		2(1-c^2) & (1-c)^2+(1+c)^2\\
		(1-c)^2+(1+c)^2 & 2(1-c^2)
	\end{pmatrix}\\&
	=2\begin{pmatrix}
		1-c^2 & 1+c^2\\
		1+c^2 & 1-c^2
	\end{pmatrix},
\end{align*}
it follows that $ \mathbf{P}_p (x) =  \mathbf{P}_p ^{\frac{1}{2}}(x) |\mathbf{P}_p(x)| ^{ \frac{1}{2}} $  and furthermore these matrices commute.

\noindent Setting
\begin{equation}\label{eq:defPsi}
	\Psi   _{p }:=  \mathbf{P}_p^{\frac{1}{2}} Z _{p },
\end{equation}
the equation for $Z _{p }$ writes
\begin{align*}
	\(  1+  (p-3) K _{\alpha p}     \)  \Psi  _{p }=0    \text{ where }       K _{\alpha p}   = \mathbf{P}_p^{\frac{1}{2}} H _\alpha ^{-1}|\mathbf{P}_p| ^{\frac{1}{2}}.
\end{align*}
We expand
\begin{align*}&
	K _{\alpha p}  =  L _{\alpha p} +  M _{\alpha p}    \text{  with }  M _{\alpha p}  := \mathbf{P}_p^{\frac{1}{2}} N _{\alpha  }|\mathbf{P}_p| ^{\frac{1}{2}} \text{ with integral kernels }  \\&  N _{\alpha  }(x,y)= \frac{1}{2\alpha}    \begin{pmatrix} \frac{\alpha}{\kappa} e^{-\kappa |x-y|} &   0  \\  0     &   e^{-\alpha |x-y|} -1 \end{pmatrix}    \text{ and }    \\&   L _{\alpha p}(x,y) = \frac{1}{2\alpha} \mathbf{P}_p^{\frac{1}{2}}(x) \ \diag (0,1) \ |\mathbf{P}_p(y)| ^{\frac{1}{2}}.
\end{align*}
Here $(  p , \alpha )\to M _{\alpha p}$ is  in $C^\infty (  D _{\R ^2} ( ( 3,0), \delta _1),    L^2(\R , \C ^2))$ for a small enough $\delta _1>0$.
The equation for $\Psi  _{p }$ becomes   (for $\< \cdot , \cdot \> =  \< \cdot , \cdot \>_{\C ^2}$)
\begin{align}\label{eq:eqpsi1}
	\frac{1}{2\alpha}(1+(p-3) M _{\alpha p} ) ^{-1} \mathbf{P}_p^{\frac{1}{2}}(x)  e_2 \<   |\mathbf{P}_p| ^{\frac{1}{2}}\Psi  _{p } , e_2 \>   = -\frac{1}{p-3} \Psi  _{p } .
\end{align}
To have a solution in \eqref{eq:eqpsi1} it is not restrictive to posit
\begin{align}\label{eq:formpsi}
	\Psi  _{p } &= (1+(p-3) M _{\alpha p} ) ^{-1}  \mathbf{P}_p^{\frac{1}{2}}   e_2 \in C^\infty (  D _{\R ^2} ( ( 3,0), \delta _1),    L^2(\R , \C ^2))  \\  \Psi  _{p } &=  \mathbf{P}_p^{\frac{1}{2}}(1+(p-3)  N _{\alpha  } \mathbf{P}_p  ) ^{-1}  e_2 , \label{eq:formpsi2}
\end{align}
where the two formulas for $\Psi  _{p }$ are equivalent.  With them,
\eqref{eq:eqpsi1} is equivalent to
\begin{align}\label{eq:implalpha}
	\alpha  =-\frac{p-3}{2}
	s(p,\alpha )   \text{  with }  s(p,\alpha ):= \<   |\mathbf{P}_p| ^{\frac{1}{2}}(1+(p-3) M _{\alpha p} ) ^{-1}  \mathbf{P}_p^{\frac{1}{2}}   e_2   , e_2 \>   .
\end{align}
Notice that $ s(\cdot ,\cdot ) \in  C^\infty (  D _{\R ^2} ( (3,0), \delta _1),   \R )$ with
\begin{align*}
	s(p,\alpha )&=  \< e_2 ,  \mathbf{P}_p      e_2  \>   - (p-3) \< e_2 ,  |\mathbf{P}_p| ^{\frac{1}{2}} M _{\alpha p}   \mathbf{P}_p^{\frac{1}{2}}      e_2  \>  +O\( (p-3)^2   \)  \\&  =  \< e_2 ,  \mathbf{P}_p      e_2  \>   -  (p-3) \< e_2 ,   \mathbf{P}_p N_{\alpha }   \mathbf{P}_p       e_2  \>  +O\( (p-3)^2    \)  .
\end{align*}
Since, we have
\begin{align*}
	\< e_2 ,  \mathbf{P}_p      e_2  \>   &=     \frac{p-3}{ 2(p+1) }  \< e_2 ,  \begin{pmatrix} -1 & 1   \\  1 & -1 \end{pmatrix}  \phi_p  ^{p-1}      e_2  \>     + \frac{1}{p+1} \cancel{\< e_2 ,     \phi_p  ^{p-1}  \sigma _1    e_2  \> } \\& =    -   \frac{p-3}{ 2(p+1) }   \int _\R  \phi_p  ^{p-1} dx   =  -    \frac{  p-3    }{ 2( p-1 )  }   \int _\R   \sech ^2 \(   x  \)  dx   = -  \frac{ p-3   }{  p-1    }
	,
\end{align*}
with the canceled term null,
and
\begin{align*}
	\< e_2 ,   \mathbf{P}_p N_{\alpha }   \mathbf{P}_p       e_2  \> & =  \< e_2 ,   \mathbf{P}_3 N_{0 }   \mathbf{P}_3      e_2  \>  +O\( (p-3) \) +O\( \alpha  \) \\& = 4 ^{-2}\< e_2 ,  \phi_3^2 \sigma _1   N_{0 }   \phi_3^2  \sigma _1   e_2  \>  +O\( (p-3) \) +O\( \alpha  \)
	\\&  = 2^{-4} 2^{-\frac{1}{2}}\< \phi_3^2,\mathbf{ T}        \>  +O\( (p-3) \) +O\( \alpha  \) ,
\end{align*}
we obtain
\begin{align*}&
	s(p,\alpha )=  - (p-3) \(   \frac{ 1  }{  p-1  } + 2^{-5}2^{-\frac{1}{2}}\< \phi_3^2, \mathbf{ T}        \>      \) + O\( (p-3) ^3    \) +  O\( (p-3) \alpha    \) .
\end{align*}
Applying  the implicit function theorem  to \eqref{eq:implalpha} we get \eqref{eq:expansalpha}.

\qed

In analogy to Martel \cite{Martel2} we give an expansion of a  $  \xi_p$.
\begin{lemma}	\label{lem:expv} There exists an open interval $\mathcal{I}$ containing 3 and  for each $p\in \mathcal{I}$
	there exists a  solution  $ \xi_p=(\xi_{p,1},\xi_{p,2})^\intercal$  of \eqref{eq:eqv01}  of the form
	\begin{align}
		\xi _{p,1}&=1-\phi_3^2+(p-3)R_1+(p-3)^2\tilde{\xi} _{p,1},\label{eq:expandxip1}\\
		\xi _{p,2} &=\im\(1+(p-3)R_2+(p-3)^2\tilde{\xi} _{p,2}\),\label{eq:expandxip2}
	\end{align}
	where
	\begin{align*}
		R_1&=-x\phi_3\phi_3'-\frac{1}{4\sqrt{2}}(3-\phi_3^2)\mathbf{T}-\frac{\phi_3'}{2\sqrt{2}\phi_3}\mathbf{ T}' \text{ and }\\
		R_2&=\frac{1}{2}\phi_3^2+\frac{3}{4\sqrt{2}}\mathbf{T}+\frac{\phi_3'}{2\sqrt{2}\phi_3}\mathbf{ T}'
	\end{align*}
	and where furthermore, for any $k\ge 0$ there exists a constant $C_k$ such that
	\begin{align} \label{eq:exptildev}
		& |  \widetilde{\xi} _{p,j} ^{(k)} (x) |\le  C_k  \< x\>  ^{3} \text{ for all } x\in \R \text{ and all $p\in \mathcal{I}$}.
	\end{align}

\end{lemma}
\proof    From \eqref{eq:defPsi} and \eqref{eq:formpsi2}, and in particular expanding the latter,   we have
\begin{align*}
	Z  _{p }&= (1+(p-3) N _{\alpha  } \mathbf{P}_p) ^{-1}     e_2  =     e_2- (p-3) N _{\alpha  } \mathbf{P}_p    e_2  +(p-3)^2   \widetilde{{Z}}_2 \\& =  e_2-  (p-3)2^{-2} 2^{-\frac{1}{2}}\mathbf{T} e_1 + (p-3) ^2\widetilde{Z}  \text{ with }
	\widetilde{Z} =  \widetilde{{Z}}_1   + \widetilde{{Z}}_2
	\\ \widetilde{{Z}}_1  &:= -\int _0 ^1 \left .\partial   _{p'} \(   N _{\alpha  (p') }\mathbf{P} _{p'} \)\right |_{p'=3+t(p-3 )}   e_2  dt \\ \widetilde{{Z}}_2&:=
	N _{\alpha  }\mathbf{P}_pN _{\alpha  }|\mathbf{P}_p | ^{\frac{1}{2}} (1+(p-3) M _{\alpha p} ) ^{-1}   \mathbf{P}_p^{\frac{1}{2}}    e_2
\end{align*}
where we used
$N_0 \mathbf{P} _{3} = 2^{-2} \mathbf{T} e_1$. By standard computations for any $k\ge 0$ there exists a constant $C_k$ such that
\begin{align*}
	| \partial _x ^k \widetilde{{Z}}_2 |\le C_k \< x \>    \text{ for all } x\in \R \text{ and all $p$ near 3}.
\end{align*}
It is also elementary to see that for any $k\ge 0$ there exists a constant $C_k$ such that
\begin{align*}
	|\partial   _{p } \(   N _{\alpha  (p )     } (x,y)  \partial _y ^k\mathbf{P} _{p } (y) \)    e_2| \lesssim  \<x-y \> ^3 \sech \( \frac{p-1}{2} y\)  \text{ for all } x,y\in \R ,
\end{align*}
which implies for any $k\ge 0$ there exists a constant $C_k$ such that
\begin{align}\label{eq:esttildez}
	| \partial _x ^k \widetilde{{Z}}  |\le C_k \< x \>  ^3  \text{ for all } x\in \R \text{ and all $p$ near 3}.
\end{align}
Going back to $w_{p}=(w _{p,1}, w_{p,2})^\intercal $ and for $ \widetilde{ w}= U \widetilde{Z}$ we have
\begin{align*}
	& w_{p }= UZ_{p }=       \begin{pmatrix}     1
		\\    -\im      \end{pmatrix} -    (p-3)
	2^{-2} 2^{-\frac{1}{2}}\mathbf{T}       \begin{pmatrix}         1    \\     \im       \end{pmatrix}  +   (p-3) ^2   \widetilde{ w} .
\end{align*}
Notice that the first term in the expansion is exactly what we get entering $k=0$   in \eqref{eq:formw00}.

Going back to $\xi_p=(\xi_{p,1},\xi_{p,2}) ^\intercal $ by means of  \eqref{eqmartelid4} we have
\begin{align*}
	\xi_{p,1} &=(S_1^*)^2 w_{p,1}= \( \partial _x + \frac{\phi_p'}{\phi_p} \) ^2 w_{p,1}= \(  \partial _x ^2 +2 \frac{\phi_p'}{\phi_p} \partial _x + \(  \frac{\phi_p'}{\phi_p} \) '  + \frac{\phi_3 ^{\prime 2 }}{\phi_p^2}   \) w_{p,1}\\& =\(  \partial _x ^2 +2 \frac{\phi_p'}{\phi_p} \partial _x +  \frac{\phi_p''}{\phi_p}    \)  w_{p,1} .
\end{align*}
We will use also the expansion  in \cite{coles}
\begin{align}\label{eq:colesexp1}&
	\phi_p ^{p-1}=\phi_3^2+(p-3)q_1+  (p-3) ^2 q_R(x)   \text{ with }q_1(x) =\sech ^2(x)  \(  2^{-1}-2x \tanh (x)\)
\end{align}
where
\begin{align*}
	q_R(x)&= \int _0^1 \left .  \partial  _{p'} ^2\phi ^{p'-1} _{p'}(x)\right |_{p'=(1-t)3 +tp} t dt = \int _0^1 \left .  \partial  _{p'} ^2 \(  \frac{p'+1}{2}  \sech ^{2}  \(  \frac{p'-1}{2}x \) \)   \right |_{p'=(1-t)3 +tp} t dt .
\end{align*}
Notice that for any $k\ge 0$ there exists a constant $C_k$ such that
\begin{align}\label{eq:colesexp11}&
	|  q_R^{(k)}(x)|\le C_k \< x \> ^2\sech ^2 \(  \min  \left  \{  \frac{p-1}{2} ,1     \right \}    x\)   \text{ for all } x\in \R \text{ and all $p$ near 3}.
\end{align}
Recalling  the  identities
\begin{align*}&
	-\phi_p''+ \phi_p-\phi_p^p=0 \text{  and }\\& -\phi^{\prime 2}_p +\phi^{  2}_p-\frac{2}{p+1}\phi^{p+1}_p=0
\end{align*}
we get
\begin{align*}
	\frac{\phi_p''}{\phi_p}  =  1-  \phi_p^{p-1}   =1-\phi_3^2-(p-3) q_1-  (p-3) ^2 q_R ,
\end{align*}
so that    using also the expansion in \eqref{eq:colesexp1} we have
\begin{align*}
	\xi_{p,1} &=   1-\phi_3^2     +      (p-3 )  R_1           +  (p-3) ^2 \widetilde{\xi}_1
\end{align*}
where, using the equation in \eqref{eq:defboldT},
\begin{align*}
	R_1&:= -q_1    - 2^{-2}2^{-\frac{1}{2}} \(  \partial _x ^2 +2 \frac{\phi_3' }{\phi_3} \partial _x +   1-\phi_3^2 \) \mathbf{T} \\&  = -q_1+2^{-2} \phi_3^2  -2^{-2} 2^{-\frac{1}{2}}  (3-\phi_3^2)  \mathbf{T}    +2^{-1}2^{-\frac{1}{2}}\tanh (x)\mathbf{T}'
\end{align*}
which by    \eqref{eq:colesexp1}  yields the desired expression of $R_1$
and
\begin{align*}
	\widetilde{\xi}_1&:= -   \frac{2^{-\frac{1}{2}}}{4(p-3) } \(  \frac{\phi_p'}{\phi_p} -   \frac{\phi_3' }{\phi_3}\) \mathbf{T}'   +  2^{-2}2^{-\frac{1}{2}} q_R \mathbf{T} +    (S_1^*)^2 \widetilde{w}_1 .
\end{align*}
By \eqref{eq:esttildez} and \eqref{eq:colesexp11}, we have \eqref{eq:exptildev} for $j=1$.
 Next, by \eqref{eqmartelid4}, we have $-\im (1-\alpha^2)\xi_{p,2}= L_{ +p}\xi_{p,1}$.
Substituting the expansions \eqref{eq:expandxip1} and $$L_{ +p}=L_{+3}-(p-3)(\phi_3^2+3q_1)-(p-3)^2\(3q_R+q_1+(p-3)q_1q_R\),$$
which follows from \eqref{eq:colesexp1}, we have
\begin{align*}
	-\im \xi_{p,2}&=L_{+p}\xi_{p1}+(p-3)^2\(\frac{\alpha^2}{(p-3)^2} \frac{1}{1-\alpha^2} L_{+p} \xi_{p,1}\)\\&
	=L_{+3}(1-\phi_3^2) + (p-3)\(L_{+3}R_1 - (\phi_3^2+3q_1)(1-\phi_3^2)\)\\&\quad+(p-3)^2\(L_{+3} \tilde{\xi}_1 -(\phi_3^2+3q_1)\(R_1+(p-3)\tilde{\xi}_1\)-\(3q_R+q_1+(p-3)q_1q_R\)\xi _{p,1}\)\\&\quad
	+(p-3)^2\(\frac{\alpha^2}{(p-3)^2} \frac{1}{1-\alpha^2} L_{+p} \xi_{p,1}\)
\end{align*}
By looking the coefficients of $(p-3)^0$ and $(p-3)^1$ we have \eqref{eq:expandxip2}.
Further, the estimate of $\tilde{\xi}_2$ follows from the estimate of $\tilde{\xi}_1$ given by \eqref{eq:exptildev}, \eqref{eq:colesexp11} and the explicit form of $R_1$ and $q_1$.

\qed

%
%

Now, Lemma \ref{lem:FGRnondeg} is a direct consequence of the following lemma.
\begin{lemma}\label{lem:comgamma1}
For $|p-3|\ll 1$, we have
\begin{align*}
\gamma(p,1)=\frac{\pi}{\sqrt{2}\cosh(\pi/2)}(p-3) + o(p-3).
\end{align*}
\end{lemma}

\begin{proof}
We set
\begin{align*}
E&:=\left.\partial_p\right|_{p=3}\phi_p=\frac{1}{2}\phi_3\(\frac{1}{4}-\log \phi_3\) +\frac{1}{2}x \phi' _3,\\
F&:=\left.\partial_p\right|_{p=3}\phi_p^{p-2}=E+\phi_3\log\phi_3.
\end{align*}
Further, we set
\begin{align}
\tilde{F}=\phi_{p}^{p-2}-\phi_3-(p-3)F.\label{eq:ftilde}
\end{align}
Recall,
\begin{align*}
\gamma(p,1)=\<\phi_p^{p-2}(p\xi_{p,1}^2+\xi_{p,2}^2),g_{p,1}\>+2\<\phi_p^{p-2}\xi_{p,1}\xi_{p,2},g_{p,2}\>
\end{align*}
and setting, following   notation and   argument  in Martel \cite{Martel2},
\begin{align*}
G_{p,1}:=\phi_p^{p-2}(p\xi_{p,1}^2+\xi_{p,2}^2),\\
G_{p,2}:=-2\im \phi_p^{p-2}\xi_{p,1}\xi_{p,2},
\end{align*}
we have
\begin{align*}
G_{p,1}&=\phi_3\(3(1-\phi_3^2)^2-1\)+(p-3)\Delta_1+\tilde{\Delta}_1,\\
\frac{1}{2}G_{p,2}&=\phi_3(1-\phi_3^2)+(p-3)\Delta_2+\tilde{\Delta}_2,
\end{align*}
where
\begin{align*}
\Delta_1&=F\(3(1-\phi_3^2)^2-1\) + \phi_3(1-\phi_3^2)^2+6\phi_3(1-\phi_3^2)R_1-2\phi_3R_2,\\
\Delta_2&=F(1-\phi_3^2)+\phi_3R_1+\phi_3(1-\phi_3^2)R_2
\end{align*}
and $\tilde{\Delta}_1, \tilde{\Delta}_2$ are   remainder terms of $(p-3)^2$ order.
Since it is easy to verify that the $\tilde{F}$ in  \eqref{eq:ftilde}   is decaying exponentially, we see from Lemma \ref{lem:expv} that the contribution of $\tilde{\Delta}_1$ and $\tilde{\Delta}_2$ to $\gamma(p,1)$ are $(p-3)^2$ order.
Thus, we can ignore these terms.

\noindent We have $\<\phi_p,g_{p1}\>=0$  like in Martel \cite{Martel2}.
Now, by $g_{p,2}=\frac{\im}{2\lambda(p,1)}L_{p+}g_{p,1}$, we have
\begin{align*}
\gamma(p,1)&=\<G_{p,1},g_{p,1}\>+\<G_{p,2},\frac{1}{2\lambda(p,1)}L_{p+} g_{p,1}\>\\&
=\<G_{p,1}+\frac{1}{2\lambda(p,1)} L_{p+}G_{p,2},g_{p,1}\>
\end{align*}
By direct computation, see the   proof of Lemma 20 of \cite{Martel2}, we have
\begin{align*}
G_{3,1}+\frac{1}{2}(-\partial_x^2+1-3\phi_3^2)G_{3,2}=\phi_3\(3(1-\phi_3^2)^2-1\)+\frac{1}{2}(-\partial_x^2+1-3\phi_3^2)\(\phi_3(1-\phi_3)^2\)=2\phi_3.
\end{align*}
Thus, we see $\gamma(3,1)=0$.
Further, expanding $G_{p,1}, G_{p,2}$ and $L_{p+}$, we have
\begin{align*}
&\gamma(p,1)=\<2\phi_3+(p-3)\(\Delta_1+L_{3+}\Delta_2 +\frac{1}{2}\(\left.\partial_p\right|_{p=3}L_{p+} \)G_{3,2}\),g_{p,1}\>+o(p-3)\\&
=(p-3)\(-2\<E,g_{3,1}\>+\<\Delta_1,g_{3,1}\> +2\<\Delta_2,-\im g_{3,2}\> - \frac{1}{2}\< \phi_3\(\frac{7}{4}\phi _3+3x\phi_3'\)G_{3,2},g_{3,1} \> \)\\&\quad +o(p-3).
\end{align*}
Thus, it suffices to compute the coefficient of $p-3$, which we denote $\gamma_1$ (i.e.\ $\gamma(p,1)=(p-3)\gamma_1+o(p-3)$).
Following Martel \cite{Martel2} we will consider the following constants,
 \begin{align*}
p_k&=\int\sech^k\cos,\\
q_k&=\int \sech^k  \log\circ \sech   \cos \\
r_k&=\int \sech^k   \mathbf{T}\cos\\
s_k&=\int \sech^k \mathbf{T} \tanh \sin\\
a_k&=\int x \sech^k \tanh \cos\\
b_k&=\int \sech^k\tanh \sin\\
c_k&=\int \sech^k \log\circ \sech  \tanh\sin\\
d_k&=\int x \sech^k  \sin\\
e_k&=\int \sech^k \tanh  \mathbf{T}'\cos\\
f_k&=\int \sech^k \mathbf{T}'\sin    .
\end{align*}
Then, after a quite long but elementary  computation, we arrive to
\begin{align*}
&\gamma_1=\sqt\(-\frac{3}{2}(2\log 2+1)p_5+\frac{3}{2}(2\log 2+1)p_7
+\frac{3}{2}q_3-6q_5+6q_7
-\frac{3}{2}a_3+6a_5-6a_7\)\\&
+\sqt\(\frac{3}{2}(2\log 2+1)b_3-\frac{3}{2}(2\log 2+1)b_5
-\frac{3}{2}c_1+6c_3-6c_5\)\\&
+\sqt\(\frac{3}{2}d_1-\frac{3}{2}d_3-6d_3+6d_5+6d_5-6d_7)\)\\&
+\sqt\(
-\frac{1}{2}q_3
+\frac{1}{2}a_3
+\frac{1}{2}c_1
-\frac{1}{2}d_1+\frac{1}{2}d_3\)\\&
+\sqt \(-4 p_5+4 p_7+4 b_3 -4 b_5+12 a_5-24a_7-12d_3+36 d_5-24d_7 -2p_5+2  b_3\)\\&-\frac{9}{2}r_3+12 r_5-6r_7+3e_3-6e_5
+\frac{9}{2}s_1-12 s_3+6s_5-3f_1+9f_3-6f_5-\frac{3}{2}r_3+e_3 +\frac{3}{2} s_1 - f_1+f_3.
\end{align*}
We can further simplify this quantity.
First, we can eliminate $b_k,c_k,d_k,e_k$ and $f_k$ by the identities obtained by integration by parts,
\begin{align*}
b_k
&=(k+1)p_{k+2}-kp_k,\\
c_k&
=(k+1)q_{k+2} -kq_k +p_{k+2}-p_k,\\
d_k&=-ka_k+p_k,\\
e_k&
=s_k+kr_k-(k+1)r_{k+2},\\
f_k&=-r_k+ks_k.
\end{align*}
The expression given by $p_k,q_k,r_k,s_k$ and $a_k$ ($k=1,3,5,7$) can be reduced to $p_1,q_1,r_k,s_1$ and $a_1$ by the identities, again obtained by integration by parts,
\begin{align*}
p_{k+2}&=\frac{1+k^2}{k(k+1)}p_k,\\
q_{k+2}&=\frac{1}{k(k+1)}\((1+k^2)q_k -(2k+1)p_{k+2} + (k+1)p_k\),\\
r_{k+2}&=\frac{1}{k(k+1)}\((k^2-3) r_k+2ks_k+2\sqt p_{k+2}\),\\
s_{k+2}&=\frac{1}{(k+1)(k+2)}\( (k^2-3) s_k +2(k+1)r_{k+2} -2k r_k +2\sqt (k+3) p_{k+4}-2\sqt (k+2) p_{k+2}\),\\
a_{k+2} &=\frac{1}{(k+1)(k+2)} \((k^2+1)a_k -2kp_k +2(k+1)p_{k+2}\).
\end{align*}
Now, as in Martel \cite{Martel2}, a quite surprising simplification occurs. That is,
after reducing the expression to a linear combination of $p_1,q_1,r_1,s_1$ and $a_1$ by means of lengthy but nonetheless very elementary computations, the coefficients of $q_1,r_1,s_1$ and $a_1$ vanish  and we are left with the very simple formula
\begin{align*}
\gamma_1=\frac{1}{\sqrt{2}}p_1,
\end{align*}
like in  Martel \cite{Martel2}.
We have $p_1=\pi/\cosh(\pi/2)$, see Martel \cite{Martel2}. We have the relation  $p_1=p_{1,\text{Martel}}/\sqrt{2}$. This completes the proof of  Lemma \ref{lem:comgamma1}. We   provide all the elementary computations
which we have skipped here  in the   note \cite{CM242}.
\end{proof}

\section*{Acknowledgments}
C. was supported   by the Prin 2020 project \textit{Hamiltonian and Dispersive PDEs} N. 2020XB3EFL.
M.  was supported by the JSPS KAKENHI Grant Number 19K03579, G19KK0066A, 23H01079 and 24K06792. The authors are very grateful to prof. Dmitry Pelinovsky for the information and the computations he provided for the plane waves of the linearization of the cubic NLS and for his comments on the manuscript.

Department of Mathematics and Geosciences,  University
of Trieste, via Valerio  12/1  Trieste, 34127  Italy.
{\it E-mail Address}: {\tt scuccagna@units.it}

Department of Mathematics and Informatics,
Graduate School of Science,
Chiba University,
Chiba 263-8522, Japan.
{\it E-mail Address}: {\tt maeda@math.s.chiba-u.ac.jp}

\end{document}